\documentclass[12pt]{article}
\pdfoutput=1

\usepackage{amsmath, amsthm, amssymb}
\usepackage{enumerate}
\usepackage{ifpdf}
\ifpdf
\usepackage[pdftex]{graphicx}
\else
\usepackage[dvips]{graphicx}
\fi
\usepackage{tikz}
 	 \usetikzlibrary{arrows,backgrounds}
\usepackage[all]{xy}

\input xy
\xyoption{all}

\usepackage[pdftex,plainpages=false,hypertexnames=false,pdfpagelabels]{hyperref}
\newcommand{\arXiv}[1]{\href{http://arxiv.org/abs/#1}{\tt arXiv:\nolinkurl{#1}}}
\newcommand{\doi}[1]{\href{http://dx.doi.org/#1}{{\tt DOI:#1}}}

\newcommand{\googlebooks}[1]{(preview at \href{http://books.google.com/books?id=#1}{google books})}

\usepackage{xcolor}
\definecolor{dark-red}{rgb}{0.7,0.25,0.25}
\definecolor{dark-blue}{rgb}{0.15,0.15,0.55}
\definecolor{medium-blue}{rgb}{0,0,.8}
\definecolor{DarkGreen}{RGB}{0,150,0}
\hypersetup{
   colorlinks, linkcolor={purple},
   citecolor={medium-blue}, urlcolor={medium-blue}
}

\theoremstyle{plain}
\newtheorem{thm}{Theorem}[section]
\newtheorem*{thm*}{Theorem}
\newtheorem{cor}[thm]{Corollary}
\newtheorem{lem}[thm]{Lemma}

\newtheorem{quest}[thm]{Question}
\newtheorem*{quest*}{Question}

\theoremstyle{definition}
\newtheorem{defn}[thm]{Definition}
\newtheorem{assumption}[thm]{Assumption}
\newtheorem{nota}[thm]{Notation}

\newtheorem{note}[thm]{Note}

\newtheorem{ex}[thm]{Example}
\newtheorem{rem}[thm]{Remark}

\DeclareMathOperator{\coev}{coev}

\DeclareMathOperator{\End}{End}
\DeclareMathOperator{\Epi}{Epi}
\DeclareMathOperator{\Hom}{Hom}
\DeclareMathOperator{\Gr}{Gr}
\DeclareMathOperator{\ev}{ev}
\DeclareMathOperator{\id}{id}

\DeclareMathOperator{\OP}{op}

\DeclareMathOperator{\Tr}{Tr}
\DeclareMathOperator{\tr}{tr}

\newcommand{\comment}[1]{}

\newcommand{\be}{\begin{enumerate}}
\newcommand{\ee}{\end{enumerate}}

\newcommand{\N}{\mathbb{N}}
\newcommand{\Z}{\mathbb{Z}}

\newcommand{\F}{\mathbb{F}}
\newcommand{\R}{\mathbb{R}}
\newcommand{\C}{\mathbb{C}}

\renewcommand{\P}{\mathbb{P}}
\newcommand{\I}{\infty}
\newcommand{\set}[2]{\left\{#1 \middle| #2\right\}}

\newcommand{\cA}{\mathcal{A}}
\newcommand{\cB}{\mathcal{B}}
\newcommand{\cC}{\mathcal{C}}

\newcommand{\cF}{\mathcal{F}}
\newcommand{\cG}{\mathcal{G}}
\newcommand{\cH}{\mathcal{H}}

\newcommand{\cL}{\mathcal{L}}
\newcommand{\cM}{\mathcal{M}}
\newcommand{\cN}{\mathcal{N}}

\newcommand{\cS}{\mathcal{S}}

\newcommand{\cV}{\mathcal{V}}

\newcommand{\sC}{{\sf C}}
\newcommand{\PA}{{\mathcal{P}\hspace{-.1cm}\mathcal{A}}}
\newcommand{\Pro}{{\sf Pro}}
\newcommand{\CF}{{\sf CF}}
\newcommand{\Bim}{{\sf Bim}}
\newcommand{\Rep}{{\sf Rep}}
\newcommand{\op}{^{\OP}}

\newcommand{\noshow}[1]{}
\newcommand{\MR}[1]{}
\newcommand{\bluecup}{\textcolor{\cupcolor}{\cup}}

\tikzstyle{shaded}=[fill=red!10!blue!20!gray!30!white]
\tikzstyle{unshaded}=[fill=white]
\tikzstyle{empty box}=[circle, draw, thick, fill=white, opaque, inner sep=2mm]
\tikzstyle{annular}=[scale=.7, inner sep=1mm, baseline]
\tikzstyle{rectangular}=[scale=.75, inner sep=1mm, baseline=-.1cm]

\newcommand{\cupcolor}{DarkGreen}
\newcommand{\alphacolor}{blue}
\newcommand{\betacolor}{red}

\newcommand{\Mbox}[2]{
\begin{tikzpicture}[baseline=-.1cm]
	\draw (0,0)--(0,.8);
	\draw[thick, unshaded] (-.4, -.4) -- (-.4, .4) -- (.4, .4) -- (.4,-.4) -- (-.4, -.4);
	\node at (.2,.6) {{\scriptsize{$#1$}}};
	\node at (0,0) {$#2$};
\end{tikzpicture}
}

\newcommand{\nbox}[3]{
\begin{tikzpicture}[baseline=-.1cm]
	\draw (0,-.8)--(0,.8);
	\draw[thick, unshaded] (-.4, -.4) -- (-.4, .4) -- (.4, .4) -- (.4,-.4) -- (-.4, -.4);
	\node at (.2,.6) {{\scriptsize{$#1$}}};
	\node at (0,0) {$#2$};
	\node at (.2,-.6) {{\scriptsize{$#3$}}};
\end{tikzpicture}
}

\newcommand{\ColorNBox}[3]{
\begin{tikzpicture}[baseline=-.1cm]
	\draw[thick, #3] (0,-.8)--(0,0);
	\draw (0,0)--(0,.8);
	\draw[thick, unshaded] (-.4, -.4) -- (-.4, .4) -- (.4, .4) -- (.4,-.4) -- (-.4, -.4);
	\node at (.2,.6) {{\scriptsize{$#1$}}};
	\node at (0,0) {$#2$};
\end{tikzpicture}
}

\newcommand{\widenbox}[3]{
\begin{tikzpicture}[baseline=-.1cm]
	\draw (0,-.8)--(0,.8);
	\draw[thick, unshaded] (-.65, -.4) -- (-.65, .4) -- (.65, .4) -- (.65,-.4) -- (-.65, -.4);
	\node at (.25,.6) {{\scriptsize{$#1$}}};
	\node at (0,0) {$#2$};
	\node at (.25,-.6) {{\scriptsize{$#3$}}};
\end{tikzpicture}
}

\newcommand{\rotateccw}[5]{
\begin{tikzpicture}[baseline=-.1cm]
	\clip (-1.1,-.9)--(-1.1,.9)--(1.1,.9)--(1.1,-.9);	
	\draw (0,.4) arc (180:0:.4cm)--(.8,-.8);
	\draw (0,-.4) arc (0:-180:.4cm)--(-.8,.8);
	\draw[thick, unshaded] (-.4, -.4) -- (-.4, .4) -- (.4, .4) -- (.4,-.4) -- (-.4, -.4);
	\node at (0,0) {$#5$};
	\node at (-1,.6) {{\scriptsize{$#1$}}};
	\node at (-.2,.6) {{\scriptsize{$#2$}}};
	\node at (.2,-.6) {{\scriptsize{$#3$}}};
	\node at (1,-.6) {{\scriptsize{$#4$}}};
\end{tikzpicture}
}

\newcommand{\rotatecw}[5]{
\begin{tikzpicture}[baseline=-.1cm]
	\clip (-1.1,-.9)--(-1.1,.9)--(1.1,.9)--(1.1,-.9);	
	\draw (0,.4) arc (0:180:.4cm)--(-.8,-.8);
	\draw (0,-.4) arc (-180:0:.4cm)--(.8,.8);
	\draw[thick, unshaded] (-.4, -.4) -- (-.4, .4) -- (.4, .4) -- (.4,-.4) -- (-.4, -.4);
	\node at (0,0) {$#5$};
	\node at (-1,-.6) {{\scriptsize{$#1$}}};
	\node at (-.2,-.6) {{\scriptsize{$#2$}}};
	\node at (.2,.6) {{\scriptsize{$#3$}}};
	\node at (1,.6) {{\scriptsize{$#4$}}};
\end{tikzpicture}
}

\newcommand{\trace}[4]{
\begin{tikzpicture}[baseline = -.1cm]
	\draw (0,.4) arc (180:0:.4cm)--(.8,-.4) arc (0:-180:.4cm);
	\filldraw[unshaded,thick] (-.4,.4)--(.4,.4)--(.4,-.4)--(-.4,-.4)--(-.4,.4);
	\node at (1,0) {{\scriptsize{$#1$}}};
	\node at (-.2,.6) {{\scriptsize{$#2$}}};
	\node at (0,0) {$#3$};
	\node at (-.2,-.6) {{\scriptsize{$#4$}}};
\end{tikzpicture}
}

\newcommand{\traceop}[4]{
\begin{tikzpicture}[baseline = -.1cm]
	\draw (0,.4) arc (0:180:.4cm)--(-.8,-.4) arc (-180:0:.4cm);
	\filldraw[unshaded,thick] (-.4,.4)--(.4,.4)--(.4,-.4)--(-.4,-.4)--(-.4,.4);
	\node at (-1,0) {{\scriptsize{$#1$}}};
	\node at (.2,.6) {{\scriptsize{$#2$}}};
	\node at (0,0) {$#3$};
	\node at (.2,-.6) {{\scriptsize{$#4$}}};
\end{tikzpicture}
}

\newcommand{\PAMultiply}[5]{
\begin{tikzpicture}[baseline=.5cm]
	\draw (0,2)--(0,-.8);
	\filldraw[unshaded,thick] (-.4,.4)--(.4,.4)--(.4,-.4)--(-.4,-.4)--(-.4,.4);
	\draw[thick, unshaded] (-.4, .8) -- (-.4, 1.6) -- (.4, 1.6) -- (.4,.8) -- (-.4, .8);
	\node at (.2,1.8) {{\scriptsize{$#1$}}};
	\node at (0,1.2) {$#2$};
	\node at (.2,.6) {{\scriptsize{$#3$}}};
	\node at (0,0) {$#4$};
	\node at (.2,-.6) {{\scriptsize{$#5$}}};
\end{tikzpicture}
}

\newcommand{\ColorMultiply}[3]{
\begin{tikzpicture}[baseline=.5cm]
	\draw (0,2)--(0,1);
	\draw [thick, #3] (0,1)--(0,-.8);
	\filldraw[unshaded,thick] (-.4,.4)--(.4,.4)--(.4,-.4)--(-.4,-.4)--(-.4,.4);
	\draw[thick, unshaded] (-.4, .8) -- (-.4, 1.6) -- (.4, 1.6) -- (.4,.8) -- (-.4, .8);
	\node at (0,1.2) {$#1$};
	\node at (0,0) {$#2$};
\end{tikzpicture}
}

\newcommand{\PAMultiplyThree}[7]{
\begin{tikzpicture}[baseline=-.1cm]
	\draw (0,1.8)--(0,-2);
	\filldraw[unshaded,thick] (-.4,.4)--(.4,.4)--(.4,-.4)--(-.4,-.4)--(-.4,.4);
	\draw[thick, unshaded] (-.4, .8) -- (-.4, 1.6) -- (.4, 1.6) -- (.4,.8) -- (-.4, .8);
	\draw[thick, unshaded] (-.4, -.8) -- (-.4, -1.6) -- (.4, -1.6) -- (.4,-.8) -- (-.4, -.8);
	\node at (.2,1.8) {{\scriptsize{$#1$}}};
	\node at (0,1.2) {$#2$};
	\node at (.2,.6) {{\scriptsize{$#3$}}};
	\node at (0,0) {$#4$};	
	\node at (.2,-.6) {{\scriptsize{$#5$}}};
	\node at (0,-1.2) {$#6$};
	\node at (.2,-1.8) {{\scriptsize{$#7$}}};
\end{tikzpicture}
}

\newcommand{\ColorMultiplyThree}[5]{
\begin{tikzpicture}[baseline=-.1cm]
	\draw (0,1.8)--(0,1);
	\draw[thick, #4] (0,1)--(0,-1);
	\draw[thick, #5] (0,-1)--(0,-2);
	\filldraw[unshaded,thick] (-.4,.4)--(.4,.4)--(.4,-.4)--(-.4,-.4)--(-.4,.4);
	\draw[thick, unshaded] (-.4, .8) -- (-.4, 1.6) -- (.4, 1.6) -- (.4,.8) -- (-.4, .8);
	\draw[thick, unshaded] (-.4, -.8) -- (-.4, -1.6) -- (.4, -1.6) -- (.4,-.8) -- (-.4, -.8);
	\node at (0,1.2) {$#1$};
	\node at (0,0) {$#2$};	
	\node at (0,-1.2) {$#3$};
\end{tikzpicture}
}

\newcommand{\MInnerProduct}[3]{
\begin{tikzpicture}[baseline=-.1cm]
	\draw [thick, #2] (0,.6)--(0, -.4);
	\draw [thick, #2] (0, -.4) .. controls ++(270:.4cm) and ++(270:.4cm) .. (1.2,-.4);
	\draw[thick, #2] (1.2,-.4) -- (1.2,.6);
	\draw[thick, unshaded] (1.6, -.4) -- (1.6, .4) -- (.8, .4) -- (.8, -.4) -- (1.6, -.4);
	\draw[thick, unshaded] (-.4, -.4) -- (-.4, .4) -- (.4, .4) -- (.4, -.4) -- (-.4, -.4);
	\node at (1.2, 0) {$#3$};
	\node at (0, 0) {$#1$};
\end{tikzpicture}
}

\newcommand{\MInnerProductOp}[3]{
\begin{tikzpicture}[baseline=-.1cm]
	\draw[thick, #2] (0,-.6)--(0, .4);
	\draw[thick, #2] (1.2,.4) -- (1.2,-.6);
	\draw[thick, #2] (0, .4) .. controls ++(90:.4cm) and ++(90:.4cm) .. (1.2,.4);
	\draw[thick, unshaded] (1.6, -.4) -- (1.6, .4) -- (.8, .4) -- (.8, -.4) -- (1.6, -.4);
	\draw[thick, unshaded] (-.4, -.4) -- (-.4, .4) -- (.4, .4) -- (.4, -.4) -- (-.4, -.4);
	\node at (1.2, 0) {$#3$};
	\node at (0, 0) {$#1$};
\end{tikzpicture}
}

\begin{document}
\title{Rigid $C^*$-tensor categories of bimodules over interpolated free group factors}
\author{Arnaud Brothier\footnote{KU Leuven, Department of Mathematics, Celestijnenlaan 200B - box 2400, 3001 Heverlee, Belgium}, Michael Hartglass\footnote{UC Berkeley, Department of Mathematics, 970 Evans Hall box 3840, Berkeley, CA 94720-3840 USA}, and David Penneys\footnote{Department of Mathematics, University of Toronto, Bahen Centre, 40 St. George St., Room 6290, Toronto, Ontario M5S 2E4, Canada} }
\date{\today}
\maketitle
\begin{abstract}
Given a countably generated rigid $C^*$-tensor category $\sC$, we construct a planar algebra $P_\bullet$ whose category of projections $\Pro$ is equivalent to $\sC$. From $P_\bullet$, we use methods of Guionnet-Jones-Shlyakhtenko-Walker to construct a rigid $C^*$-tensor category $\Bim$ whose objects are bifinite bimodules over an interpolated free group factor, and we show $\Bim$ is equivalent to $\Pro$. We use these constructions to show $\sC$ is equivalent to a category of bifinite bimodules over $L(\F_\infty)$.
%This is the published version of \arXiv{...}.
\end{abstract}
%\tableofcontents
%%%%%%%%%%%%%%%%%%%%%%%%%%%%%%%%%%%%%%%%%%%%%%%%%%
%%%%%%%%%%%%%%%%%%%%%%%%%%%%%%%%%%%%%%%%%%%%%%%%%%
%%%%%%%%%%%%%%%%%%%%%%%%%%%%%%%%%%%%%%%%%%%%%%%%%%
\section{Introduction}

Jones initiated the modern theory of subfactors in his breakthrough paper \cite{MR696688} in which he classified the possible values for the index of a $II_1$ subfactor to the range $\set{4\cos^2(\pi/n)}{n\geq 3}\cup [4,\infty]$, and he found a subfactor of the hyperfinite $II_1$-factor $R$ for each allowed index.

A finite index subfactor $N\subset M$ is studied by analyzing its standard invariant, i.e., two rigid $C^*$-tensor categories of $N-N$ and $M-M$ bimodules and the module categories of $N-M$ and $M-N$ bimodules which arise from the Jones tower. The standard invariant has been axiomatized in three similar ways, each emphasizing slightly different structure: Ocneanu's paragroups \cite{MR996454,MR1642584}, Popa's $\lambda$-lattices \cite{MR1334479}, and Jones' planar algebras \cite{math/9909027}.

In \cite{MR1198815,MR1334479,MR1887878}, Popa starts with a $\lambda$-lattice $A_{\bullet,\bullet}=(A_{i,j})$ and constructs a $II_1$-subfactor whose standard invariant is $A_{\bullet,\bullet}$. Hence for each subfactor planar algebra $P_\bullet$, there is some subfactor whose planar algebra is $P_\bullet$. However, the following question remains unanswered:

\begin{quest}\label{quest:WhichP}
For which subfactor planar algebras $P_\bullet$ is there a subfactor of $R$ whose planar algebra is $P_\bullet$?
\end{quest}

Using his reconstruction theorems, Popa gave a positive answer to Question \ref{quest:WhichP} for (strongly) amenable subfactor planar algebras \cite{MR1278111}.

In \cite{MR2051399}, Popa and Shlyakhtenko were able to identify the factors in certain cases of Popa's reconstruction theorems. Using this, they gave a positive answer to Question \ref{quest:WhichP} for $L(\F_\infty)$, i.e., every subfactor planar algebra arises as the standard invariant of some subfactor $N\subset M$ such that $N,M$ are both isomorphic to $L(\F_\infty)$. This theorem was reproduced by Hartglass \cite{1208.2933} using the reconstruction results of Guionnet-Jones-Shlyakhtenko-Walker (GJSW) \cite{MR2732052,MR2645882,MR2807103} which produce subfactors of interpolated free group factors.

It is natural to extend these questions to rigid $C^*$-tensor categories, i.e.,
\begin{quest}\label{quest:WhichC}
For which rigid $C^*$-tensor categories $\sC$ is there a category $\sC_{bim}$ of bifnite bimodules over $R$ such that $\sC_{bim}$ is equivalent to $\sC$?
\end{quest}

As in the subfactor case, Hayashi and Yamagami gave a positive result for amenable rigid $C^*$-tensor categories \cite{MR1749868} (amenability for $C^*$-tensor categories was first studied by Hiai and Izumi \cite{MR1644299}). Moreover, given a rigid $C^*$-tensor category $\sC$, Yamagami constructed a category of bifite bimodules $\sC_{bim}$ over an amalgamated free product $II_1$-factor such that $\sC_{bim}$ is equivalent to $\sC$ \cite{MR1960417}. However, one can show these factors have property $\Gamma$, so they are not interpolated free group factors (we briefly sketch this in Appendix \ref{sec:Appendix}).

In this paper, we give a result analogous to Popa and Shlyakhtenko's results for $L(\F_\infty)$ for countably generated rigid $C^*$-tensor categories, which answers part of Question 9 in \cite[Section 6]{MR2681261}. Recall that a rigid $C^*$-tensor category $\sC$ is generated by a set of objects $\cS$ if for every $Y\in\sC$, there are $X_1,\dots, X_n\in \cS$ such that
$$
\sC(X_1\otimes\cdots \otimes X_n,Y)\neq (0),
$$
i.e., $Y$ is (isomorphic to) a sub-object of $X_1\otimes\cdots \otimes X_n$.

\begin{thm}\label{thm:Main}
Every countably generated rigid $C^*$-tensor category can be realized as a category of bifinite bimodules over $L(\F_\infty)$.
\end{thm}

\begin{rem}
Note that when $\sC$ is finitely generated, we can prove Theorem \ref{thm:Main} using \cite{MR2051399} by adapting the technique in \cite[Theorem 4.1]{1112.4088}. We provide a sketch of the proof in Appendix \ref{sec:Appendix}, where we also point out some difficulties of using the results of \cite{MR2051399} when $\sC$ is not finitely generated  (see also \cite[Section 4]{MR1960417}).
%We outline a natural approach using Popa systems in Appendix \ref{sec:Appendix} which ultimately fails

Hence we choose to use planar algebra technology to prove Theorem \ref{thm:Main} since it offers the following advantages. First, the same construction works for both the finitely and infinitely generated cases. Second, planar diagrams arise naturally in the study of tensor categories, and a reader familiar with the diagrams may benefit from a planar algebraic approach.
Third, we get an elegant description of the bimodules over $L(\F_\I)$ directly from the planar algebra (see Sections \ref{sec:MoreBimodules} and \ref{sec:CategoriesOfBimodules}).
\end{rem}

There are three steps to the proof of Theorem \ref{thm:Main}.
\begin{enumerate}[(1)]
\item
Given a countably generated $C^*$-tensor category $\sC$, we get a factor planar algebra $P_\bullet$ such that the $C^*$-tensor category $\Pro$ of projections of $P_\bullet$ is equivalent to $\sC$.

A \underline{factor planar algebra} (called a fantastic planar algebra in \cite{1208.3637}) is an unshaded, spherical, evaluable $C^*$-planar algebra.

This step is well known to experts; we give most of the details in Section \ref{sec:TCandPA}.

\item
Given a factor planar algebra $P_\bullet$, we construct a $II_1$-factor $M$ and two rigid $C^*$-tensor categories of bifinite bimodules over $M$:
\begin{itemize}
\item
$\Bim$, built entirely from $P_\bullet$ and obviously equivalent to $\Pro$, and
\item
$\CF$, formed using Connes' fusion and linear operators.
\end{itemize}
These categories are defined in Definitions \ref{defn:Bim} and \ref{defn:CF}. We then show $\Bim\simeq \CF$ in Theorem \ref{thm:BimCF}.

We use results of GJSW to accomplish this step in Section \ref{sec:GJS}. Along the way, we adapt Brothier's treatment \cite{1202.1298} of GJSW results \cite{MR2732052,MR2645882} for unshaded planar algebras.

\item
We show $M\cong L(\F_\I)$.

This last step is similar to results of GJS \cite{MR2807103} and Hartglass \cite{1208.2933}.
\end{enumerate}

One can use similar analysis as in the proof of Theorem \ref{thm:Main} to prove the following theorem:

\begin{thm}\label{thm:Finite}
Suppose that in addition, $\sC$ has finitely many isomorphism classes of simple objects, i.e., $\sC$ is a unitary fusion category. Picking an object $X\in\sC$ which generates $\sC$, then $\sC$ can be realized as a category of bifinite bimodules over $L(\F_t)$ with
$$
t=1+\dim(\sC)(\dim(X\oplus\overline{X})-1)=1+\dim(\sC)(2\dim(X)-1),
$$
where $\dim(\sC)$ is the Frobenius-Perron dimension of $\sC$.
\end{thm}

\begin{rem}

Note that a version of Theorem \ref{thm:Finite} can be obtained by using \cite[Theorem 4.1]{1112.4088} together with \cite{MR2807103}.  If $Z_{1},\dots, Z_{n}$ are representatives for the simple objects in $\sC$ and $Y = \bigoplus_{k=1}^{n} Z_{k}$, then one obtains the factor $L(\F_s)$ for
$$
s = 1+\dim(\sC)(\dim(Y)-1),
$$
which will be a different parameter than what we obtained in Theorem \ref{thm:Finite}.
%Also note that by altering our choice of $X$ in Theorem \ref{thm:Finite}, we can obtain countably many possible values of the parameter $t$.
\end{rem}

On the other end of the spectrum, one should also note that there has been interesting work on rigid $C^*$-tensor categories of bimodules over $II_1$-factors by Vaes, Falgui\`{e}res, and Raum \cite{0811.1764,1112.4088}. Given a rigid $C^*$-tensor category $\sC$ which is either $\Rep(G)$ for $G$ a compact quantum group \cite{0811.1764} or a unitary fusion category \cite{1112.4088}, they construct a $II_1$-factor $M$ whose category $\Bim(M)$ of bifinite bimodules is \underline{exactly} $\sC$ (up to equivalence). Their results can be interpreted as rigidity results in contrast to the universality of $L(\F_\infty)$ to rigid $C^*$-tensor categories.

\paragraph{Acknowledgements.} We would like to thank Vaughan Jones, Scott Morrison, Noah Snyder, and Stefaan Vaes for many helpful conversations. The majority of this work was completed at the 2012 NCGOA on Conformal field theory and von Neumann algebras at Vanderbilt University and the 2012 Subfactors in Maui conference. The authors would like to thank Dietmar Bisch, Vaughan Jones, James Tener, and the other organizers for those opportunities. The authors were supported by DOD-DARPA grants HR0011-11-1-0001 and HR0011-12-1-0009. Michael Hartglass and David Penneys were also supported by NSF Grant DMS-0856316.  Arnaud Brothier was also supported by ERC Starting Grant VNALG-200749.

%%%%%%%%%%%%%%%%%%%%%%%%%%%%%%%%%%%%%%%%%%%%%%%%%%
\section{Tensor categories and planar algebras}\label{sec:TCandPA}

We briefly recall how to go back and forth between rigid $C^*$-tensor categories and factor planar algebras.
The contents of this subsection are well known to experts.
Our treatment follows \cite{MR2559686,JonesPAnotes,1207.1923,MR2811311,MR1960417, MR2681261}.

\begin{nota}
Categories will be denoted with the sans-serif font ${\sf ABC}\dots$. We write $X\in\sC$ to mean $X$ is an object in $\sC$, and we write $\sC(X\to Y)$ or $\sC(X,Y)$ for the set of morphisms from $X$ to $Y$ in $\sC$.
\end{nota}

%%%%%%%%%%%%%%%%%%%%%%%%%%%%%%%%%%%%%%%%%%%%%%%%%%%%%
\subsection{Factor planar algebras and principal graphs}

We briefly recall the definition of an unshaded factor planar algebra where the strings are labelled.
For more details, see \cite{JonesPAnotes}.

\begin{defn}
Given a set $\cL$, the \underline{planar operad with string labels $\cL$}, denoted $\P_{\cL}$, is the set of all planar tangles whose strings are labelled by elements of $\cL$, e.g.,
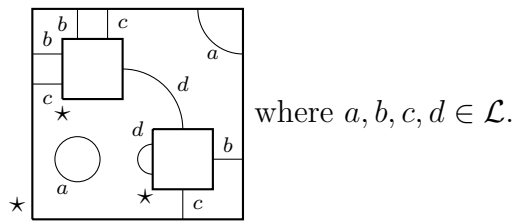
$$
\begin{tikzpicture}[baseline = .5cm]
	\draw[thick] (-2,-.8)--(-2,2)--(.8,2)--(.8,-.8)--(-2,-.8);
	\draw (.8,1.4) arc (-90:-180:.6cm);
	\draw (0,.4) arc (0:90:.8cm);
	\draw (-.4,-.2) arc (270:90:.2cm);
	\draw (-1,2)--(-1,1)--(-2,1);
	\draw (-1.4,2)--(-1.4,1.4)--(-2,1.4);
	\draw (0,-.8)--(0,0)--(.8,0);
	\draw (-1.4,0) circle (.3cm);
	\filldraw[unshaded,thick] (-1.6,.8)--(-1.6,1.6)--(-.8,1.6)--(-.8,.8)--(-1.6,.8);
	\filldraw[unshaded,thick] (-.4,.4)--(.4,.4)--(.4,-.4)--(-.4,-.4)--(-.4,.4);
	\node at (-1.6,-.4) {{\scriptsize{$a$}}};
	\node at (-1.6,1.8) {{\scriptsize{$b$}}};
	\node at (-1.8,1.6) {{\scriptsize{$b$}}};
	\node at (-.8,1.8) {{\scriptsize{$c$}}};
	\node at (-1.8,.8) {{\scriptsize{$c$}}};
	\node at (.4,1.4) {{\scriptsize{$a$}}};
	\node at (0,1) {{\scriptsize{$d$}}};
	\node at (-.6,.4) {{\scriptsize{$d$}}};
	\node at (.6,.2) {{\scriptsize{$b$}}};
	\node at (.2,-.6) {{\scriptsize{$c$}}};
	\node at (-.5,-.5) {$\star$};
	\node at (-1.6,.6) {$\star$};
	\node at (-2.2,-.6) {$\star$};
\end{tikzpicture}
\text{ where }a,b,c,d\in\cL.
$$
A \underline{planar tangle} $T$ consists of the following data:
\begin{enumerate}[(1)]
\item
a rectangle $D_0(T)\subset \R^2$,
\item
Finitely many disjoint rectangles $D_1(T),\dots ,D_s(T)$ in the interior of $D_0(T)$ ($s$ may be zero),
\item
Finitely many disjoint smooth arcs in $D_0(T)$ called the strings of $T$ which do not meet the interior of any $D_1(T),\dots, D_s(T)$.
The boundary points of a string of $T$ (if it has any) lie in the boundaries of the $D_i(T)$, and they meet these boundaries transversally (if they meet the boundaries at all).
\end{enumerate}
The boundaries of the strings divide the boundaries of the rectangles into \underline{intervals}. For each rectangle, there is a distinguished interval denoted $\star$. The intervals of $D_i(T)$ are divided by the \underline{marked points} of $D_i(T)$, i.e., the points at which the strings meet the boundary of $D_i(T)$. Starting at $\star$, the marked points on each rectangle are number clockwise.

Each string is labelled by an element from $\cL$, which induces a labeling on the marked points of the $D_i(T)$. Reading clockwise around the boundary of $D_i(T)$ starting at $\star$, we get a word $\alpha_i\in \Lambda$, the set of all finite words on $\cL$. We call $D_i(T)$ an $\alpha_i$-rectangle, and we call such a planar tangle $T$ a \underline{planar $\alpha_0$-tangle}.

If we have two planar tangles $S,T$ satisfying the following \underline{boundary conditions}:
\begin{itemize}
\item
some internal rectangle $D_i(S)$ agrees with $D_0(T)$,
\item
the marked points of $D_i(S)$ agree with the marked points of $D_0(T)$,
\item
the distinguished interval of $D_i(S)$ agrees with the distinguished interval of $T$, and
\item
the label from $\cL$ of each marked point of $D_i(S)$ agrees with the label of each marked point from $D_0(T)$,
\end{itemize}
then we may compose $S$ and $T$ to get the planar tangle $S\circ_i T$ by taking $S$ union the interior of $D_0(T)$, removing the boundary of $D_i(S)$, and smoothing the strings.
\end{defn}

\begin{rem}
When we draw a planar tangle, we will often suppress the external rectangle, which is assumed to be large. If we omit the $\star$, it is always assumed the $\star$ is in the \underline{lower left corner}. Finally, if $a_1,\dots,a_k\in\cL$ and  $\alpha=a_1\dots a_k\in\Lambda$, we draw one string labelled $\alpha$ rather than $k$ parallel strings labelled $a_1,\dots ,a_k$ where we always read the strands from \underline{left to right} and \underline{top to bottom}.

For each word $\alpha=a_1\cdots,a_n\in\Lambda$, we write $\overline{\alpha}=a_n\cdots a_1$ for the word in the reverse order.
\end{rem}

\begin{defn}
A \underline{planar algebra $P_\bullet$ with string labels $\cL$} is
\begin{itemize}
\item
a collection of vector spaces $(P_\alpha)_{\alpha\in\Lambda}$ (recall $\Lambda$ is the set of finite words on $\cL$)
\item
an action of planar tangles by multilinear maps, i.e., for each planar $\alpha_0$-tangle $T$, whose rectangles $D_i(T)$ are $\alpha_i$-rectangles, there is a multilinear map
$$
Z_T\colon \prod_{i=1}^s P_{\alpha_i}\to P_{\alpha_0}
$$
satisfying the following axioms:
\begin{enumerate}
\item[\underline{\text{Isotopy}:}]
If $\theta$ is an orientation preserving diffeomorphism of $\R^2$, then $Z_{\theta(T)}=Z_T$.
\item[\underline{\text{Naturality}:}]
For $S,T$ composable tangles, $Z(S\circ_i T) = Z(S)\circ_i Z(T)$, where the composition on the right hand side is the composition of multilinear maps.
\end{enumerate}
\end{itemize}
Moreover, $P_\bullet$ is called a \underline{factor planar algebra} if $P_\bullet$ is
\begin{itemize}
\item
\underline{evaluable}, i.e., $\dim(P_\alpha)<\infty$ for all $\alpha\in\Lambda$ and $P_\emptyset\cong \C$ via the map that sends the empty diagram to $1\in\C$. Hence, by naturally, to each $a\in\cL$, there is a scalar $\delta_a\in\C$ such that any labelled diagram containing a closed loop labelled $a$ is equal to the same diagram without the closed loop multiplied by $\delta_a$. We use the notation $\delta_{\alpha}=\delta_{a_1}\cdots \delta_{a_n}$ if $\alpha = a_1\cdots a_n$.
\item
\underline{involutive}, i.e., for each $\alpha\in\Lambda$, there is a map $*\colon P_\alpha\to P_{\overline{\alpha}}$ with $*\circ *=\id$ which is compatible with the reflection of tangles, i.e., if $T$ is a planar tangle labelled by $\alpha_1,\dots,\alpha_s$, then
$$
T(\alpha_1,\dots, \alpha_n)^*=T^*(\alpha_1^*,\dots,\alpha_n^*)
$$
where $T^*$ is the reflection of $T$.
\item
\underline{spherical}, i.e., for all $\alpha\in\Lambda$ and all $x\in P_{\alpha\overline{\alpha}}$, we have
$$
\tr(x)
=
\trace{\overline{\alpha}}{\alpha}{x}{\alpha}
=
\traceop{\overline{\alpha}}{\alpha}{x}{\alpha}\,.
$$
\item
\underline{positive}, i.e., for every $\alpha\in\Lambda$, the map $\langle\cdot,\cdot\rangle \colon P_\alpha\times P_\alpha\to P_\emptyset\cong\C$ given by
$$
\langle x,y\rangle =
\begin{tikzpicture}[baseline = -.1cm]
	\draw (0,0)--(1.2,0);	
	\filldraw[unshaded,thick] (-.4,.4)--(.4,.4)--(.4,-.4)--(-.4,-.4)--(-.4,.4);
	\filldraw[unshaded,thick] (.8,.4)--(1.6,.4)--(1.6,-.4)--(.8,-.4)--(.8,.4);
	\node at (0,0) {$x$};
	\node at (1.2,0) {$y^*$};
	\node at (.6,.2) {{\scriptsize{$\alpha$}}};
\end{tikzpicture}
$$
is a positive definite inner product. Hence for all $a\in\cL$, $\delta_a>0$.
\end{itemize}
\end{defn}

\begin{nota}
For each $\alpha,\beta\in\Lambda$, we write $P_{\alpha\to\beta}$ to denote the box space $P_{\alpha\overline{\beta}}$ of elements of the form
$$
\nbox{\alpha}{x}{\beta}\,.
$$
\end{nota}

\begin{rem}\label{rem:multiply}
Note that for each $\alpha\in\Lambda$, the multiplication tangle
$$
\PAMultiply{\alpha}{}{\alpha}{}{\alpha}
$$
makes $P_{\alpha\to\alpha}$ into an associative algebra. If $P_\bullet$ is a factor planar algebra, then the multiplication tangle makes $P_{\alpha\to\alpha}$ a finite dimensional $C^*$-algebra.
\end{rem}

\begin{defn}
Suppose $P_\bullet$ is a factor planar algebra. A \underline{projection} in $P_\bullet$ is an element $p\in P_{\alpha\to\alpha}$ satisfying $p=p^2=p^*$ where the multiplication is as in Remark \ref{rem:multiply}. A projection $p\in P_{\alpha\to\alpha}$ is called \underline{simple} if it is a minimal projection in $P_{\alpha\to\alpha}$. Since $P_{\alpha\to \alpha}$ is a finite dimensional $C^*$-algebra, every projection is the sum of finitely many simple projections. This property is called \underline{semi-simpicity}.

Given a projection $p\in P_{\alpha\to\alpha}$, the dual projection $\overline{p}\in P_{\overline{\alpha}\to\overline{\alpha}}$ is obtained by
$$
\overline{p}
=
\rotateccw{\overline{\alpha}}{\alpha}{\alpha}{\overline{\alpha}}{p}
=
\rotatecw{\overline{\alpha}}{\alpha}{\alpha}{\overline{\alpha}}{p}
\,.
$$

Projections $p\in P_{\alpha\to\alpha}$ and $q\in P_{\beta\to\beta}$ are \underline{isomorphic} or \underline{equivalent}, denoted $p\simeq q$, if there is a $u\in P_{\alpha\to\beta}$ such that
$$
u^*u
=
\PAMultiply{\alpha}{u}{\beta}{u^*}{\alpha}
=
p
\text{ and }
uu^*
=
\PAMultiply{\beta}{u^*}{\alpha}{u}{\beta}
=
q.
$$

Given a projection $p\in P_{\alpha\to\alpha}$ and a $\beta\in \Lambda$, we can form the projection
$$
p\otimes \id_{\beta} =
\nbox{\alpha}{p}{\alpha}\,
\begin{tikzpicture}[baseline=-.1cm]
	\draw (0,-.8)--(0,.8);
	\node at (.2,0) {{\scriptsize{$\beta$}}};
\end{tikzpicture}
\in P_{\alpha\beta\to\alpha\beta}.
$$
The \underline{principal graph  of $P_\bullet$ with respect to $\beta\in\Lambda$}, denoted $\Gamma_\beta$, is the graph whose vertices are the isomorphism classes of simple projections in $P_\bullet$, and if $p\in P_{\alpha\to\alpha}$ and $q\in P_{\alpha\beta\to\alpha\beta}$ are simple projections, then the vertices $[p]$ and $[q]$ are connected by $\dim(qP_{\alpha\beta\to\alpha\beta}(p\otimes\id_{\beta}))$ edges.

The \underline{principal graph of $P_\bullet$}, denoted $\Gamma$, is the push out of the $\Gamma_b$ for $b\in\cL$ over the isomorphism classes of simple projections, i.e., the vertices are the same as before, and the edge set is the union of the edge sets of the $\Gamma_b$ for $b\in \cL$.

Since $P_\bullet$ is factor and $\cL$ is countable, $\Gamma$ has countably many vertices, although it may not be locally finite.
However, $\Gamma_b$ is always locally finite for $b\in \cL$.

Given a vertex $[p]$ of $\Gamma$, the number $\tr(p)$ is independent of the choice of representative of $[p]$. The vector $(\tr(p))_{[p]\in V(\Gamma)}$ defines a \underline{Frobenius-Perron weight vector} on the vertices of $\Gamma$ satisfying the following equation for each $b\in\cL$:
$$
\delta_b \tr(p) = \sum_{[q]\in V(\Gamma_b)} n_{[p],[q]}^b \tr(q)
$$
where $n_{[p],[q]}^b$ is the number of edges connecting $[p]$ and $[q]$ in $\Gamma_b$.
\end{defn}

%%%%%%%%%%%%%%%%%%%%%%%%%%%%%%%%%%%%%%%%%%%%%%%%%
\subsection{Rigid $C^*$-tensor categories and fusion graphs}
We briefly recall the definition of a rigid $C^*$-tensor category.

\begin{defn}
A \underline{rigid $C^*$-tensor category} is a pivotal, spherical, positive/unitary, rigid, semisimple, linear ({\sf Vect}-enriched) monoidal category such that $\End(1)\cong\C$.
\end{defn}

We now unravel this definition and state many properties that follow. The interested reader should see \cite{MR1960417, MR2681261} for more details. As we go through the properties, we will also go through the well-known graphical calculus used for strict tensor categories. We will immediately see that we get a factor planar algebra from a rigid $C^*$-tensor category.

We start with an abelian category $\sC$ together with
\begin{itemize}
\item
a bifunctor $\otimes\colon \sC\times \sC\to \sC$ which is associative up to a natural isomorphism (the pentagon axiom is satisfied), and
\item
a unit object $1\in\sC$ which is a left and right identity for $\otimes$ up to natural isomorphism (the triangle identity is satisfied).
\end{itemize}

\begin{rem}
Recall that a tensor category is called strict if the above natural isomorphisms are identities, i.e., for each $X,Y,Z\in\sC$, we have
\begin{align*}
(X\otimes Y)\otimes Y &= X(\otimes (Y\otimes Z)\text{ and}\\
1\otimes X &= X= X\otimes 1.
\end{align*}
Since each tensor category is equivalent to a strict tensor category by Theorem 7.2.1 in \cite{MR1712872}, our tensor categories will be assumed to be strict unless otherwise stated. Note that even with all the properties we want, we can still restrict our attention to strict categories.
\end{rem}

First, since $\sC$ is {\sf Vect}-enriched, for each $X,Y\in\sC$, $\sC(X\to Y)$ is a finite dimensional complex vector space. The morphisms in $\sC$ are drawn as boxes with strings emanating from the top and bottom. The strings are labelled by the objects, and the diagram is read from \underline{top to bottom}. For example,
$$
f
=
\nbox{X}{f}{Y}
\in\sC(X\to Y),
$$
and the identity morphism $\id_X$ is denoted by the horizontal strand labelled $X$. We compose morphisms by vertical concatenation
$$
\widenbox{X}{f\circ g}{Z}
=
\PAMultiply{X}{g}{Y}{f}{Z}
\in \sC(X\to Z),
$$
and we tensor morphisms by horizontal concatenation:
$$
\widenbox{\hspace{.8cm}X_1\otimes Y_1}{f_1\otimes f_2}{\hspace{.8cm}X_2\otimes Y_2}
=
\nbox{X_1}{f_1}{Y_1}
\,
\nbox{X_2}{f_2}{Y_2}
\in\sC(X_1\otimes X_2\to Y_1\otimes Y_2).
$$

Since $\sC$ is rigid, for each $X\in\sC$, there is a \emph{dual} or \emph{conjugate} $\overline{X}\in\sC$, and there is a natural isomorphism $\overline{\overline{X}}\cong X$. Along with the dual object, we have an \emph{evaluation map} $\ev_X\colon \overline{X}\otimes X\to 1$ and a \emph{coevaluation} map $\coev_X\colon 1\to X\otimes \overline{X}$ such that the diagram
$$
\xymatrix{
&X\otimes \overline{X}\otimes X \ar[dr]^{1\otimes \ev_X }\\
X\ar[ur]^{\coev_{X}\otimes 1}\ar[dr]_{1\otimes \coev_{\overline{X}}}\ar[rr]^{\id_X} && X\\
& X\otimes \overline{X}\otimes X \ar[ur]_{\ev_{\overline{X}}\otimes 1}
}
$$
commutes. The evaluation is denoted by a cap, and we a draw a cup for the coevaluation:
$$
\ev_X=
\begin{tikzpicture}[baseline=-.1cm]
	\draw[dotted] (0,-.4)--(0,0);
	\draw (-.4,.4) arc (-180:0:.4cm);
	\node at (.6,.4) {{\scriptsize{$X$}}};
	\node at (-.6,.4) {{\scriptsize{$\overline{X}$}}};
	\node at (0,-.6) {$1$};
\end{tikzpicture}
\text{ and }
\coev_X=
\begin{tikzpicture}[baseline=.1cm]
	\draw[dotted] (0,.4)--(0,0);
	\draw (-.4,-.4) arc (180:0:.4cm);
	\node at (-.6,-.4) {{\scriptsize{$X$}}};
	\node at (.6,-.4) {{\scriptsize{$\overline{X}$}}};
	\node at (0,.6) {$1$};
\end{tikzpicture}\,.
$$
The diagram above commuting is sometimes referred to as the \emph{zig-zag relation}, since it is the straightening of the kinked string:
$$
\begin{tikzpicture}[baseline=-.1cm]
	\draw (0,-.8)--(0,.8);
	\node at (.2,0) {{\scriptsize{$X$}}};
\end{tikzpicture}
=
\begin{tikzpicture}[baseline=-.1cm]
	\draw (0,.4) arc (180:0:.4cm)--(.8,-1.2);
	\draw (0,-.4)--(0,.4);
	\draw[dotted] (.4,.8)--(.4,1.2);
	\draw[dotted] (-.4,-.8)--(-.4,-1.2);
	\draw (0,-.4) arc (0:-180:.4cm)--(-.8,1.2);
	\node at (.2,0) {{\scriptsize{$\overline{X}$}}};
	\node at (1,-.6) {{\scriptsize{$X$}}};
	\node at (-.6,.6) {{\scriptsize{$X$}}};
\end{tikzpicture}
=
\begin{tikzpicture}[baseline=-.1cm]
	\draw (0,.4) arc (0:180:.4cm)--(-.8,-1.2);
	\draw (0,-.4)--(0,.4);
	\draw[dotted] (-.4,.8)--(-.4,1.2);
	\draw[dotted] (.4,-.8)--(.4,-1.2);
	\draw (0,-.4) arc (-180:0:.4cm)--(.8,1.2);
	\node at (-.2,0) {{\scriptsize{$\overline{X}$}}};
	\node at (-1,-.6) {{\scriptsize{$X$}}};
	\node at (.6,.6) {{\scriptsize{$X$}}};
\end{tikzpicture}\,.
$$
In general, we don't draw a string connected to the trivial object $1\in\sC$. For each $X,Y\in\sC$, $\overline{X\otimes Y}$ is naturally isomorphic to $\overline{Y}\otimes \overline{X}$, and the diagram
$$
\xymatrix{
\overline{X\otimes Y} \otimes (X\otimes Y)\ar[d]  \ar[rr]^(.6){\ev_{X\otimes Y}}&& 1\\
\overline{Y}\otimes \overline{X}\otimes X\otimes Y\ar[rr]^(.6){1\otimes \ev_X\otimes 1} && \overline{Y}\otimes Y\ar[u]^{\ev_Y}\\
}
$$
commutes, and similarly for the $\coev$'s. This diagram just means that we can write one cap labelled $X\otimes Y$ and its dual instead of two separate caps labelled $X$ and $Y$ and their duals:
$$
\begin{tikzpicture}[baseline=-.1cm]
	\draw (-.4,.4) arc (-180:0:.4cm);
	\draw (-.8,.4) arc (-180:0:.8cm);
	\node at (.6,.4) {{\scriptsize{$X$}}};
	\node at (-.6,.4) {{\scriptsize{$\overline{X}$}}};
	\node at (1,.4) {{\scriptsize{$Y$}}};
	\node at (-1,.4) {{\scriptsize{$\overline{Y}$}}};
\end{tikzpicture}
=
\begin{tikzpicture}[baseline=-.1cm]
	\draw (-.4,.4) arc (-180:0:.4cm);
	\node at (.9,.4) {{\scriptsize{$X\otimes Y$}}};
	\node at (-.9,.4) {{\scriptsize{$\overline{X\otimes Y}$}}};
\end{tikzpicture},
$$
and similarly for the cups.

The \emph{pivotality} axiom in $\sC$ requires that for all $f\in \sC(X\to Y)$,
$$
(\ev_{Y}\otimes \id_{\overline{X}})\circ(\id_{\overline{Y}}\otimes f\otimes \id_{\overline{X}})\circ(\id_{\overline{Y}} \otimes \coev_X)
=(\id_X\otimes \ev_{Y})\circ(\id_{\overline{Y}}\otimes f\otimes \id_{\overline{X}})\circ(\coev_{\overline{X}}\otimes\id_{\overline{Y}}).
$$
The equation above has an elegant representation in diagrams:
$$
\rotateccw{\overline{Y}}{X}{Y}{\overline{X}}{f}
=
\rotatecw{\overline{X}}{Y}{X}{\overline{Y}}{f}\,.
$$
For $f\in \sC(X\to Y)$, the above diagram defines a dual map $\overline{f}\in \sC(\overline{Y}\to \overline{X})$, and $\overline{\overline{f}}=f$.

The evaluations and coevaluations together with pivotality allow us to define a \emph{left} and \emph{right trace} on $\End_\sC(X)$:
\begin{align*}
\tr_L(f)&=\ev_{X}\circ (\id_{\overline{X}}\otimes f)\circ \coev_{\overline{X}}
=
\traceop{\overline{X}}{X}{f}{X}
\in \End(1)\cong \C\text{ and}\\
\tr_R(f)&=\ev_{\overline{X}}\circ (f\otimes \id_{\overline{X}})\circ \coev_{X}
=
\trace{\overline{X}}{X}{f}{X}
\in \End(1)\cong \C.
\end{align*}
Similarly, for each $X\in\sC$, there are numbers $d^L_{X}$ and $d^R_{X}$ which are the left and right traces of the identity morphism respectively, and $d_{X}^L=d_{\overline{X}}^R$ and $d_{X}^R=d_{\overline{X}}^L$

\emph{Sphericality} means that these two traces are equal, and we denote the common number by $\tr(f)$. The sphericality allows us to perform isotopy on closed diagrams as if they were drawn on a sphere. Hence $\dim(X):=d_X^L=d_X^R$ and $\dim(X)=\dim(\overline{X})$ for all $X\in \sC$.

The \emph{positivity} or \emph{unitarity} of $\sC$ means there is a contravariant functor $*\colon \sC\to \sC$ which is the identity on all objects, and on morphisms, it is anti-linear, involutive ($*\circ *=\id_{\sC}$), monoidal ($(f\circ g)^*=g^*\circ f^*$ for composable $f,g$), and positive ($f^*\circ f=0$ implies $f=0$). We require $*$ to be compatible with the duality ($\overline{f}^*=\overline{f^*}$) and with the evaluations and coevaluations (for all $X\in\sC$, $\coev_X=\ev_{\overline{X}}^*$). On diagrams, we perform $*$ by reflecting the diagram, keeping the labels on the strings, and placing a $*$ on all morphisms.

For all $X,Y\in\sC$, we now have that $\sC(X\to Y)$ is a Banach space with positive definite inner product
$$
\langle f,g\rangle = \tr(g^*f)=
\begin{tikzpicture}[baseline=.5cm]
	\draw (0,1.6) arc (180:0:.4cm) -- (.8,-.4) arc (0:-180:.4cm);
	\draw (0,.4)--(0,.8);
	\filldraw[unshaded,thick] (-.4,.4)--(.4,.4)--(.4,-.4)--(-.4,-.4)--(-.4,.4);
	\draw[thick, unshaded] (-.4, .8) -- (-.4, 1.6) -- (.4, 1.6) -- (.4,.8) -- (-.4, .8);
	\node at (-.2,1.8) {{\scriptsize{$X$}}};
	\node at (0,1.2) {$f$};
	\node at (-.2,.6) {{\scriptsize{$Y$}}};
	\node at (0,0) {$g^*$};
	\node at (-.2,-.6) {{\scriptsize{$X$}}};
	\node at (1,.6) {{\scriptsize{$\overline{X}$}}};
\end{tikzpicture}.
$$
The inner product makes $\End_{\sC}(X)$ a finite dimensional $C^*$-algebra, so in particular, all projections are sums of finitely many simple projections, and $\sC$ is \emph{semi-simple}, i.e., every exact sequence in $\sC$ splits. This also means that any object in $\sC$ can be written as a finite direct sum of simple objects. Recall that $X\in\sC$ is \emph{simple} if $\dim(\End_{\sC}(X))=1$. Thus if $X,Y$ are non-isomorphic simple objects, $\sC(X,Y)=(0)$. This means that for each simple $X,Y,Z\in \sC$, there are non-negative integers $N_{X,Y}^Z $ such that $X\otimes Y = \bigoplus_{Z\in\sC} N_{X,Y}^Z Z$, i.e.,
$$
N_{X,Y}^Z = \dim(\sC(X\otimes Y\to Z)).
$$
Moreover, we have \emph{Frobenius reciprocity}, i.e., for each $X,Y,Z\in\sC$, there are natural isomorphisms
$$
\sC(X\otimes Y\to Z)\cong \sC(X\to Z\otimes \overline{Y}) \cong \sC(Y\to \overline{X}\otimes Z)
$$
which are implemented by the evaluation and coevaluation maps:
$$
\begin{tikzpicture}[baseline=-.1cm]
	\draw (0,0)--(0,-.8);
	\draw (-.2,.8)--(-.2,0);
	\draw (.2,.8)--(.2,0);
	\draw[thick, unshaded] (-.4, -.4) -- (-.4, .4) -- (.4, .4) -- (.4, -.4) -- (-.4, -.4);
	\node at (0, 0) {$f$};
	\node at (-.4,.6) {\scriptsize{$X$}};
	\node at (.4,.6) {\scriptsize{$Y$}};
	\node at (.2,-.6) {\scriptsize{$Z$}};
\end{tikzpicture}
\leftrightarrow
\begin{tikzpicture}[baseline=-.1cm]
	\draw (0,0)--(0,-.8);
	\draw (-.2,.8)--(-.2,0);
	\draw (.2,.4) arc (180:0:.2cm) -- (.6,-.8);
	\draw[thick, unshaded] (-.4, -.4) -- (-.4, .4) -- (.4, .4) -- (.4, -.4) -- (-.4, -.4);
	\node at (0, 0) {$f$};
	\node at (-.4,.6) {\scriptsize{$X$}};
	\node at (.4,.8) {\scriptsize{$Y$}};
	\node at (.8,-.6) {\scriptsize{$\overline{Y}$}};
	\node at (.2,-.6) {\scriptsize{$Z$}};
\end{tikzpicture}
\leftrightarrow
\begin{tikzpicture}[baseline=-.1cm]
	\draw (0,0)--(0,-.8);
	\draw (.2,.8)--(.2,0);
	\draw (-.2,.4) arc (0:180:.2cm) -- (-.6,-.8);
	\draw[thick, unshaded] (-.4, -.4) -- (-.4, .4) -- (.4, .4) -- (.4, -.4) -- (-.4, -.4);
	\node at (0, 0) {$f$};
	\node at (-.4,.8) {\scriptsize{$X$}};
	\node at (-.8,-.6) {\scriptsize{$\overline{X}$}};
	\node at (.4,.8) {\scriptsize{$Y$}};
	\node at (.2,-.6) {\scriptsize{$Z$}};
\end{tikzpicture}.
$$
Hence, for all simple $X,Y,Z\in\sC$, we have $N_{X,Y}^Z=N_{Z,\overline{Y}}^X=N_{\overline{X},Z}^{Y}$.

\begin{defn}
An object $X\in \sC$ has a \underline{self-duality} if there is an invertible $\varphi\in \sC(X,\overline{X})$, which must satisfy certain compatibility axioms. We would like this $\varphi$ to allow us to define evaluation and coevaluation maps $X\otimes X\to 1$ and $1\to X\otimes X$, i.e, they are adjoint to each other, satisfy the zig-zag relation, and give a positive scalar for $\dim(X)$ when composed in the natural way. We define these maps by
$$
\begin{tikzpicture}[baseline=-.1cm]
	\draw (0,.6) -- (0,-.2) arc (-180:0:.4cm) -- (.8,.6);
	\filldraw[unshaded,thick] (-.25,.25)--(.25,.25)--(.25,-.25)--(-.25,-.25)--(-.25,.25);
	\node at (-.2,.5) {{\scriptsize{$X$}}};
	\node at (0,0) {$\varphi$};
	\node at (-.15,-.5) {{\scriptsize{$\overline{X}$}}};
	\node at (1,0) {{\scriptsize{$X$}}};
\end{tikzpicture}
\text{ and }
\begin{tikzpicture}[baseline=.1cm]
	\draw (0,-.6) -- (0,.2) arc (180:0:.4cm) -- (.8,-.6);
	\filldraw[unshaded,thick] (-.25,.25)--(.25,.25)--(.25,-.25)--(-.25,-.25)--(-.25,.25);
	\node at (-.2,-.5) {{\scriptsize{$X$}}};
	\node at (0,0) {$\varphi^*$};
	\node at (-.15,.5) {{\scriptsize{$\overline{X}$}}};
	\node at (1,0) {{\scriptsize{$X$}}};
\end{tikzpicture}
$$
respectively. Since $X$ is naturally isomorphic to $\overline{\overline{X}}$, $\overline{\varphi}$ is naturally in $\sC(X,\overline{X})$. Therefore, the compatibility requirements are that $\varphi$ must satisfy $\varphi\overline{\varphi^*}=\id_X$ and $\tr(\varphi\varphi^*)=\dim(X)$.
However, to be able to draw these diagrams naively by just a cup and a cap without the label $\varphi$, we must have that each of these maps is preserved by rotation:
$$
\begin{tikzpicture}[baseline=-.1cm]
	\draw (0,.6) -- (0,-.2) arc (-180:0:.4cm) -- (.8,.6) arc (180:0:.3cm) -- (1.4,-.4) .. controls ++(270:.6cm) and ++(270:.6cm) .. (-.6,-.4)--(-.6,.6);
	\filldraw[unshaded,thick] (-.25,.25)--(.25,.25)--(.25,-.25)--(-.25,-.25)--(-.25,.25);
	\node at (-.2,.5) {{\scriptsize{$X$}}};
	\node at (0,0) {$\varphi$};
	\node at (-.15,-.5) {{\scriptsize{$\overline{X}$}}};
	\node at (1,0) {{\scriptsize{$X$}}};
	\node at (1.6,0) {{\scriptsize{$\overline{X}$}}};
	\node at (-.8,0) {{\scriptsize{$X$}}};
\end{tikzpicture}
=
\begin{tikzpicture}[baseline=-.1cm]
	\draw (0,.6) -- (0,-.2) arc (0:-180:.4cm) -- (-.8,.6);
	\filldraw[unshaded,thick] (-.25,.25)--(.25,.25)--(.25,-.25)--(-.25,-.25)--(-.25,.25);
	\node at (.2,.5) {{\scriptsize{$X$}}};
	\node at (0,0) {$\varphi$};
	\node at (.15,-.5) {{\scriptsize{$\overline{X}$}}};
	\node at (-1,0) {{\scriptsize{$X$}}};
\end{tikzpicture}
=
\begin{tikzpicture}[baseline=-.1cm]
	\draw (0,.6) -- (0,-.2) arc (-180:0:.4cm) -- (.8,.6);
	\filldraw[unshaded,thick] (-.25,.25)--(.25,.25)--(.25,-.25)--(-.25,-.25)--(-.25,.25);
	\node at (-.2,.5) {{\scriptsize{$X$}}};
	\node at (0,0) {$\overline{\varphi}$};
	\node at (-.15,-.5) {{\scriptsize{$\overline{X}$}}};
	\node at (1,0) {{\scriptsize{$X$}}};
\end{tikzpicture}
=
\begin{tikzpicture}[baseline=-.1cm]
	\draw (0,.6) -- (0,-.2) arc (-180:0:.4cm) -- (.8,.6);
	\filldraw[unshaded,thick] (-.25,.25)--(.25,.25)--(.25,-.25)--(-.25,-.25)--(-.25,.25);
	\node at (-.2,.5) {{\scriptsize{$X$}}};
	\node at (0,0) {$\varphi$};
	\node at (-.15,-.5) {{\scriptsize{$\overline{X}$}}};
	\node at (1,0) {{\scriptsize{$X$}}};
\end{tikzpicture}
$$
i.e., the \underline{Frobenius-Schur indicator} \cite{MR2381536} of the evaluation must be equal to $+1$. This tells us that $\varphi=\overline{\varphi}$, and $\varphi$ is unitary ($\varphi^*\varphi=\id_X$ and $\varphi\varphi^*=\id_{\overline{X}}$). A self-duality satisfying this extra axiom is called a \underline{symmetric self-duality}.
\end{defn}

\begin{assumption}\label{assume:Countable}
We assume that our rigid $C^*$-tensor category $\sC$ is countably generated, i.e., there is a countable set $\cS$ of objects in $\sC$ such that for each $Y\in\sC$, there are $X_1,\dots, X_n\in \cS$ such that
$$
\sC(X_1\otimes\cdots \otimes X_n,Y)\neq (0),
$$
i.e., $Y$ is (isomorphic to) a sub-object of $X_1\otimes\cdots \otimes X_n$.

To perform the calculations needed to prove Theorem \ref{thm:Main}, we want the planar algebra $P_\bullet$ associated to $\sC$ in Definition \ref{defn:PAfromTC} to be non-oriented, have a non-oriented fusion graph, and have all loop parameters greater than 1.

Hence given a countable generating set $\cS$, we work with the generating set $\cL=\set{X\oplus \overline{X}}{X\in\cS}$. Note that the objects in $\cL$ are \underline{not} simple, but they are symmetrically self-dual and have dimension greater than 1.
\end{assumption}

\begin{defn}
The \underline{fusion graph of $\sC$ with respect to $Y\in\sC$}, denoted $\cF_\sC(Y)$, is the oriented graph whose vertices are the isomorphism classes of simple objects of $\sC$, and between simple objects $X,Z\in\sC$, there are $N_{X,Y}^Z=\dim(\sC(X\otimes Y\to Z))$ oriented edges pointing from $X$ to $Z$. Note that if $Y$ is self-dual, then by semi-simplicity, we have $N_{X,Y}^Z=N_{Z,Y}^X$, and we may ignore the orientation of the edges.

The \underline{fusion graph of $\sC$ with respect to $\cL$} (with $\cL$ as in Assumption \ref{assume:Countable}), denoted $\cF_\sC(\cL)$, is the push out of the $\cF_\sC(Y)$ over the isomorphism classes of simple objects $Y\in\cL$, i.e., the vertices are the same as before, and the edge set is the union of the edge sets of the $\cF_\sC(Y)$ for $Y\in\cL$. If $e$ is an edge in $\cF_\sC(\cL)$ which comes from an edge in $\cF_\sC(Y)$, then we color $e$ by $Y$.

Since $\cL$ is countable, $\cF_\sC(\cL)$ has countably many vertices, although it may not be locally finite.
However, $\cF_\sC(X)$ is always locally finite for $X\in \sC$.

Given a vertex $[X]$ of $\cF_\sC(\cL)$, the number $\dim(X)$ is independent of the choice of representative of $[X]$.
Again, we get a \underline{Frobenius-Perron weight vector} on the vertices of $\cF_\sC(\cL)$, given by $(\dim(X))_{[X]\in V(\cF_\sC(\cL))}$, which satisfies the following equation for each $Y\in\cL$:
$$
\dim(X) \dim(Y) = \sum_{[Z]\in V(\cF_\sC(Y))}N_{X,Y}^Z \dim(Z).
$$
\end{defn}

For convenience, we will identify words on $\cL$ with their products, i.e., the word $\alpha= X_1X_2\dots X_n$ is identified with $X_1\otimes X_2\otimes \cdots \otimes X_n$.

\begin{defn}\label{defn:PAfromTC}
To get a factor planar algebra $\PA(\sC)_\bullet$, for each word $\alpha$ on $\cL$, let $\PA(\sC)_{\alpha} = \sC(\alpha\to 1)$, whose elements are represented diagrammatically as
$$
\Mbox{\alpha}{f}
$$
Frobenius reciprocity allows us to identify $\PA(\sC)_\alpha$ with $\sC(\beta\to\gamma)$ where $\alpha=\beta\overline{\gamma}$:
$$
\Mbox{\alpha}{f}
=
\nbox{\beta}{f}{\gamma}\,.
$$

We may now interpret any planar tangle in $\mathbb{P}_\cL$ labelled by morphisms of $\sC$ as one morphism in $\sC$ in the usual way. First, isotope the tangle so that each string travels transversally to each horizontal line, except at finitely many critical points. Then isotope the tangle so that each labelled rectangle and each critical point occurs at a different vertical height, and read the diagram from bottom to top to see what the morphism is. The zig-zag relation, Frobenius-reciprocity, pivotality, and symmetric self-dualities of the objects in $\cL$ ensure that the answer is well-defined.
\end{defn}

%%%%%%%%%%%%%%%%%%%%%%%%%%%%%%%%%%%%%%%%%%%%%%%%%%%%%%%%%%%%%%%%
\subsection{From planar algebras to tensor categories}

Given a factor planar algebra $P_\bullet$, we obtain its $C^*$-tensor category $\Pro({P_\bullet})$ of projections as described in \cite{MR2559686}. We briefly recall the construction here.

\begin{defn}
Let $\Pro(P_\bullet)$ (abbreviated $\Pro$ when $P_\bullet$ is understood) be the rigid $C^*$-tensor category given as follows.
\begin{enumerate}
\item[\text{\underline{Objects:}}]
The objects of $\Pro$ are formal finite direct sums of \underline{projections} in $P_\bullet$, i.e., all $p\in P_{\alpha\to\alpha}$ satisfying $p=p^2=p^*$ for all words $\alpha$ on $\cL$. The trivial object is the empty diagram.

\item[\text{\underline{Tensor:}}]
We tensor objects in $\Pro$ by horizontal concatenation; e.g., if $p\in P_{\alpha\to\alpha}$ and $q\in P_{\beta\to\beta}$, then $p\otimes q\in P_{\alpha\beta\to\alpha\beta}$ is given by
$$
\widenbox{\alpha\beta}{p\otimes q}{\alpha\beta}
=
\nbox{\alpha}{p}{\alpha}
\,
\nbox{\beta}{q}{\beta}
\in P_{\alpha\beta\to\alpha\beta}.
$$
Note that the simple objects in $\Pro$ are the simple projections in $P_\bullet$.

We extend the tensor product to direct sums of projections linearly.

\item[\text{\underline{Morphisms:}}]
The morphisms in $\Pro$ are matrices of intertwiners between the projections. If $p\in P_{\alpha\to\alpha}$ and $q\in P_{\beta\to\beta}$, then elements in $\Pro(p,q)=qP_{\alpha\to \beta}p$ are all $x\in P_{\alpha\to\beta}$ such that $x=qxp$, i.e.,
$$
\nbox{\alpha}{x}{\beta}
=
\PAMultiplyThree{\alpha}{p}{\alpha}{x}{\beta}{q}{\beta}\,.
$$
We compose morphisms by vertical concatenation of elements in the planar algebra. If we have $x\in\Pro(p\to q)$ and $y\in\Pro(q\to r)$, then the composite $xy$ is given by
$$
\nbox{}{xy}{}
=
\PAMultiply{}{x}{}{y}{}
\in\Pro(p\to r).
$$
Composition of matrices of morphisms occurs in the usual way.

\item[\text{\underline{Tensoring:}}]
We tensor morphisms by horizontal concatenation. If $x\in\Pro(p_1\to q_1)$ and $y\in\Pro(p_2\to q_2)$, then the tensor product $x\otimes y$ is given by
$$
\widenbox{}{x\otimes y}{}
=
\nbox{}{x}{}
\,
\nbox{}{y}{}\,.
$$
The tensor product of matrices of intertwiners is the tensor product of matrices followed by tensoring of morphisms.

\item[\text{\underline{Duality:}}]
The \emph{duality} operation on objects and morphisms is rotation by $\pi$
$$
\nbox{}{\overline{p}}{}
=
\rotateccw{}{}{}{}{p}
=
\rotatecw{}{}{}{}{p}.
$$
The evaluation and coevaluation maps are given by the caps and cups between the projections in the obvious way.

\item[\text{\underline{Adjoint:}}]
The \emph{adjoint} operation in $\Pro$ is
the identity on objects. The adjoint of a $1$-morphism
is the same as the adjoint operation in the planar algebra $P_\bullet$.
If $x\in\Pro(p\to q)$ where $p\in P_{\alpha\to \alpha}$ and $q\in P_{\beta\to \beta}$, then consider $x\in P_{\alpha\to \beta}$, take the adjoint, which is an element in $P_{\beta\to \alpha}$, and consider the result $x^*$ as an element in $\Pro(q\to p)$.

For matrices of intertwiners, the adjoint is the $*$-transpose.
\end{enumerate}
\end{defn}

\begin{ex}
We copy the example from \cite{MR2559686} as it is highly instructional. If $p,q\in P_{\alpha\to\alpha}$ are orthogonal, then if we define the matrix
$$
u=
\begin{pmatrix}
p & q
\end{pmatrix}\in \Pro(\id_\alpha\oplus\id_\alpha\to \id_\alpha),
$$
we get an isomorphism $p\oplus q = u^*u \simeq uu^* = p+q$.
\end{ex}

\begin{rem}
Note that $\Pro$ is strict. For any projection $p\in P_{\alpha\to\alpha}$, $p\otimes 1_{\Pro}=1_{\Pro}\otimes p = p$ since $1_{\Pro}$ is the empty diagram. For all projections $p,q,r\in P_\bullet$,
$$
(p\otimes q)\otimes r
=
\nbox{}{p}{}
\,
\nbox{}{q}{}
\,
\nbox{}{r}{}
=
p\otimes (q\otimes r).
$$
\end{rem}

The following theorem is well-known to experts, and one can easily work it out from the definitions. See part (ii) of the remark on page 10 of \cite{1207.1923} for more details.

\begin{thm}\label{thm:PAandTC}\mbox{}
\begin{enumerate}[(1)]
\item
Let $\sC$ be a strict rigid $C^*$-tensor category. Then $\Pro(\PA(\sC)_\bullet)$ is equivalent to $\sC$.
\item
Let $P_\bullet$ be a factor planar algebra. Then $\PA(\Pro(P_\bullet))_\bullet$=$P_\bullet$.
\end{enumerate}
\end{thm}

\begin{cor}
Suppose that
\begin{itemize}
\item
$\sC=\Pro(P_\bullet)$ and $P_\bullet$ has a countable set of string labels $\cL$, or
\item
$P_\bullet=\PA(\sC)_\bullet$ and $\sC$ has countable generating set $\cL$ of symmetrically self-dual objects.
\end{itemize}
Then we may identify the fusion graph $\Gamma$ of $P_\bullet$ with the fusion graph $\cF_\sC(\cL)$ of $\sC$.
\end{cor}

%%%%%%%%%%%%%%%%%%%%%%%%%%%%%%%%%%%%%%%%%%%%%%%%%%
\section{GJS results for factor planar algebras}\label{sec:GJS}

Given a subfactor planar algebra $P_\bullet$, GJSW constructed a subfactor $N\subset M$ whose planar algebra is $P_\bullet$ \cite{MR2732052,MR2645882}. Moreover, they identified the factors as interpolated free group factors \cite{MR2807103}.

Suppose we have a factor planar algebra $P_\bullet$ with a countable set of string labels $\cL$ such that for each $c\in\cL$, $\delta_c> 1$. (One can assume $P_\bullet$ is the factor planar algebra associated to a rigid $C^*$-tensor category $\sC$ with generating set $\cL$ as in Assumption \ref{assume:Countable}.) We mimic the construction of GJSW to obtain a factor $M_{0}$ and rigid $C^*$-tensor categories $\Bim$ and $\CF$ of bifinite bimodules over $M_0$ such that $\Pro$ is equivalent to $\Bim$ and $\CF$.

\begin{rem}
Recall that when we suppress the $\star$ of an input rectangle, it is assumed that $\star$ is in the \underline{lower-left} corner.
Recall that if a string is labelled by the word $\alpha\in \Lambda$, it is read either \underline{top to bottom} or \underline{left to right}.
\end{rem}

%%%%%%%%%%%%%%%%%%%%%%%%%%%%%%%%%%%%%%%%%%%%%%%%%%%%
\subsection{The graded algebras and their orthogonalized pictures} \label{sec:graded}

To start, we set $\Gr_{0}(P) = \bigoplus_{\alpha \in \Lambda} P_{\alpha}$ where $\Lambda$ denotes the set of all finite sequences of colorings for strings and endow $\Gr_{0}(P_\bullet)$  with a multiplication $\wedge$ which satisfies
$$
x \wedge y =
\Mbox{\alpha}{x}
\,
\Mbox{\beta}{y}
$$
where $x \in P_{\alpha}$ and $y \in P_{\beta}$. We endow $\Gr_{0}(P)$ with the following trace:
\begin{equation}\label{r}
\tr(x) =
\begin{tikzpicture}[baseline=.5cm]
	\draw (0,0)--(0,.8);
	\filldraw[unshaded,thick] (-.4,.4)--(.4,.4)--(.4,-.4)--(-.4,-.4)--(-.4,.4);
	\draw[thick, unshaded] (-.7, .8) -- (-.7, 1.6) -- (.7, 1.6) -- (.7,.8) -- (-.7, .8);
	\node at (0,0) {$x$};
	\node at (0,1.2) {$\Sigma CTL$};
	\node at (.2,.6) {{\scriptsize{$\alpha$}}};
%	\node at (.6,.2) {{\scriptsize{$\alpha$}}};
%	\node at (-.6,1.2) {{\scriptsize{$\beta$}}};
%	\node at (.6,1.2) {{\scriptsize{$\beta$}}};
\end{tikzpicture}
\end{equation}
where $x \in P_{\alpha}$ and $\sum CTL$ denotes the sum of all colored Temperely-Lieb diagrams, i.e. all planar ways of pairing the colors on top of $x$ in a way which respects the word $\alpha$.

\begin{lem} \label{lem:PositiveBounded}
The inner product on $\Gr_{0}(P_\bullet)$ given by $\langle x, y \rangle = \tr(y^{*}x)$ is positive definite.  Furthermore, left and right multiplication by elements in $\Gr_{0}(P_\bullet)$ is bounded with respect to this inner product
\end{lem}

The proof of the above lemma will closely follow the orthogonalization approach in \cite{MR2645882}.  To begin, we define a new algebraic structure $\star$ on $\Gr_{0}(P_\bullet)$ defined as follows.  Suppose $x \in P_{\alpha}$ and $y \in P_{\beta}$.  Then by letting $|\alpha|$ denote the length of $\alpha$, we have
$$
x \star y =
\sum_{
\substack{
\gamma \text{ s.t.}\\
\alpha=\alpha'\gamma\\
\beta=\overline{\gamma}\beta'
}}
\begin{tikzpicture}[baseline = .6cm]
	\draw (-.2,0)--(-.2,1);	
	\draw (1.4,0)--(1.4,1);
	\draw (.2,.4) arc (180:0:.4);	
	\filldraw[unshaded,thick] (-.4,.4)--(.4,.4)--(.4,-.4)--(-.4,-.4)--(-.4,.4);
	\filldraw[unshaded,thick] (.8,.4)--(1.6,.4)--(1.6,-.4)--(.8,-.4)--(.8,.4);
	\node at (0,0) {$x$};
	\node at (1.2,0) {$y$};
	\node at (-.4,.6) {{\scriptsize{$\alpha'$}}};
	\node at (1.6,.6) {{\scriptsize{$\beta'$}}};
	\node at (.6,1) {{\scriptsize{$\gamma$}}};
\end{tikzpicture}
$$
where it is understood that if a string connects two different colors, then that term in the sum is zero. We let $F_{0}(P_\bullet)$ be the vector space $\Gr_{0}(P_\bullet)$ endowed with the multiplication $\star$.  Given $x$ in $F_0(P_{\bullet})$, let $x_{\emptyset}$ denote the component of $x$ in $P_{\emptyset} \cong \C$.  We define a trace $\tr_{F}$ on $F_{0}(P_\bullet)$ by $\tr_{F}(x) = x_\emptyset$.  Since $P_\bullet$ is a $C^{*}$-planar algebra, the sesquilinear form
$$
\langle x, y \rangle
= \tr_{F_0(P_\bullet)}(x\star y^{*})
=
\begin{tikzpicture}[baseline = -.1cm]
	\draw (0,0)--(1.2,0);	
	\filldraw[unshaded,thick] (-.4,.4)--(.4,.4)--(.4,-.4)--(-.4,-.4)--(-.4,.4);
	\filldraw[unshaded,thick] (.8,.4)--(1.6,.4)--(1.6,-.4)--(.8,-.4)--(.8,.4);
	\node at (0,0) {$x$};
	\node at (1.2,0) {$y^*$};
%	\node at (.6,.2) {{\scriptsize{$\alpha$}}};
\end{tikzpicture}
$$
is a positive definite inner product.

Set $\Epi(CTL)$ to be the set of colored Temperely-Lieb boxes with strings at the top and bottom where any string touching the top of the box must be through.  One can argue exactly as in section 5 of \cite{MR2645882} that the map $\Phi: \Gr_{0}(P_\bullet)\rightarrow F_{0}(P_\bullet)$ given by
$$
\Phi(x) = \sum_{E \in \Epi(CTL)}
\begin{tikzpicture}[baseline=.7cm]
	\draw (0,0)--(0,2);
	\filldraw[unshaded,thick] (-.4,.4)--(.4,.4)--(.4,-.4)--(-.4,-.4)--(-.4,.4);
	\draw[thick, unshaded] (-.4, .8) -- (-.4, 1.6) -- (.4, 1.6) -- (.4,.8) -- (-.4, .8);
	\node at (0,0) {$x$};
	\node at (0,1.2) {$E$};
%	\node at (.2,.6) {{\scriptsize{$\alpha$}}};
%	\node at (.6,.2) {{\scriptsize{$\alpha$}}};
%	\node at (-.6,1.2) {{\scriptsize{$\beta$}}};
%	\node at (.6,1.2) {{\scriptsize{$\beta$}}};
\end{tikzpicture}
$$
is a bijection with the property that $\Phi(x \wedge y) = \Phi(x) \star \Phi(y)$, $\Phi(x^{*}) = \Phi(x)^{*}$ and $\tr(x) = \tr_{F}(\Phi(x))$.  Hence $\star$ is an associative multiplication, $F_{0}(P_\bullet)$ and $\Gr_{0}(P_\bullet)$ are isomorphic as $*$-algebras, and the inner product on $\Gr_{0}(P_\bullet)$ is positive definite.

We now prove that left multiplication by $x \in F_{0}(P_\bullet)$ is bounded (this will closely follow arguments in \cite{1202.1298}).   We may assume $x \in P_{\alpha}$ for a fixed word $\alpha$.  For fixed words $\beta$ and $\gamma$ such that $\alpha = \overline{\beta}\gamma$, the element
$$
\begin{tikzpicture}[baseline = -.1cm]
	\draw (-.8,0)--(2,0);	
	\filldraw[unshaded,thick] (-.4,.4)--(.4,.4)--(.4,-.4)--(-.4,-.4)--(-.4,.4);
	\filldraw[unshaded,thick] (.8,.4)--(1.6,.4)--(1.6,-.4)--(.8,-.4)--(.8,.4);
	\node at (0,0) {$x^*$};
	\node at (1.2,0) {$x$};
	\node at (-.6,.2) {{\scriptsize{$\gamma$}}};
	\node at (.6,.2) {{\scriptsize{$\beta$}}};
	\node at (1.8,.2) {{\scriptsize{$\gamma$}}};
\end{tikzpicture}
$$
is positive in the finite dimensional $C^{*}$ algebra $P_{\overline{\gamma}\to\overline{\gamma}}$, since for any $y \in P_{\overline{\gamma}\to\overline{\gamma}}$,
$$
\langle x^*x\cdot y, y \rangle_{P_{\overline{\gamma}\to\overline{\gamma}}}
=
\begin{tikzpicture}[baseline = .4cm]
	\draw (-.4,0)--(1.6,0) arc (-90:90:.5cm) -- (-.4,1) arc (90:270:.5cm);	
	\filldraw[unshaded,thick] (-.4,1.4)--(.4,1.4)--(.4,.6)--(-.4,.6)--(-.4,1.4);
	\filldraw[unshaded,thick] (.8,1.4)--(1.6,1.4)--(1.6,.6)--(.8,.6)--(.8,1.4);
	\filldraw[unshaded,thick] (-.4,.4)--(.4,.4)--(.4,-.4)--(-.4,-.4)--(-.4,.4);
	\filldraw[unshaded,thick] (.8,.4)--(1.6,.4)--(1.6,-.4)--(.8,-.4)--(.8,.4);
	\node at (0,0) {$x^*$};
	\node at (1.2,0) {$x$};
	\node at (0,1) {$y^*$};
	\node at (1.2,1) {$y$};
	\node at (-.6,.2) {{\scriptsize{$\gamma$}}};
	\node at (.6,.2) {{\scriptsize{$\beta$}}};
	\node at (.6,1.2) {{\scriptsize{$\overline{\gamma}$}}};
	\node at (1.8,.2) {{\scriptsize{$\gamma$}}};
	\node at (0,-.55) {$\star$};
	\node at (1.2,-.55) {$\star$};
	\node at (0,1.55) {$\star$};
	\node at (1.2,1.55) {$\star$};
\end{tikzpicture}
=
\left\|
\begin{tikzpicture}[baseline = .3cm]
	\draw (-.2,0)--(-.2,1);	
	\draw (1.4,0)--(1.4,1);
	\draw (.2,.4) arc (180:0:.4);	
	\filldraw[unshaded,thick] (-.4,.4)--(.4,.4)--(.4,-.4)--(-.4,-.4)--(-.4,.4);
	\filldraw[unshaded,thick] (.8,.4)--(1.6,.4)--(1.6,-.4)--(.8,-.4)--(.8,.4);
	\node at (0,0) {$x$};
	\node at (1.2,0) {$y$};
	\node at (-.4,.6) {{\scriptsize{$\overline{\beta}$}}};
	\node at (1.6,.6) {{\scriptsize{$\gamma$}}};
	\node at (.6,1) {{\scriptsize{$\gamma$}}};
\end{tikzpicture}
\right\|^2_{L^2(F_0(P_\bullet))}\geq 0.
$$
Given $x$ and $w$ with $x \in P_{\alpha}$ with $\alpha=\overline{\beta}\gamma$, $x \star w$ is a sum of terms of the form
$$
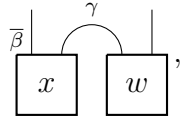
\begin{tikzpicture}[baseline = .3cm]
	\draw (-.2,0)--(-.2,1);	
	\draw (1.4,0)--(1.4,1);
	\draw (.2,.4) arc (180:0:.4);	
	\filldraw[unshaded,thick] (-.4,.4)--(.4,.4)--(.4,-.4)--(-.4,-.4)--(-.4,.4);
	\filldraw[unshaded,thick] (.8,.4)--(1.6,.4)--(1.6,-.4)--(.8,-.4)--(.8,.4);
	\node at (0,0) {$x$};
	\node at (1.2,0) {$w$};
	\node at (-.4,.6) {{\scriptsize{$\overline{\beta}$}}};
%	\node at (1.6,.6) {{\scriptsize{$\beta'$}}};
	\node at (.6,1) {{\scriptsize{$\gamma$}}};
\end{tikzpicture}\,,
$$
and we see that the 2-norm of the above diagram is
\begin{equation} \label{eqn:TwoNorm}
\begin{tikzpicture}[baseline = .4cm]
	\draw (-.4,0)--(1.6,0) arc (-90:90:.5cm) -- (-.4,1) arc (90:270:.5cm);	
	\filldraw[unshaded,thick] (-.4,1.4)--(.4,1.4)--(.4,.6)--(-.4,.6)--(-.4,1.4);
	\filldraw[unshaded,thick] (.8,1.4)--(1.6,1.4)--(1.6,.6)--(.8,.6)--(.8,1.4);
	\filldraw[unshaded,thick] (-.4,.4)--(.4,.4)--(.4,-.4)--(-.4,-.4)--(-.4,.4);
	\filldraw[unshaded,thick] (.8,.4)--(1.6,.4)--(1.6,-.4)--(.8,-.4)--(.8,.4);
	\node at (0,0) {$x^*$};
	\node at (1.2,0) {$x$};
	\node at (0,1) {$w^*$};
	\node at (1.2,1) {$w$};
	\node at (-.6,.2) {{\scriptsize{$\gamma$}}};
	\node at (.6,.2) {{\scriptsize{$\beta$}}};
%	\node at (.6,1.2) {{\scriptsize{$\overline{\beta}$}}};
	\node at (1.8,.2) {{\scriptsize{$\gamma$}}};
	\node at (0,-.55) {$\star$};
	\node at (1.2,-.55) {$\star$};
	\node at (0,1.55) {$\star$};
	\node at (1.2,1.55) {$\star$};
\end{tikzpicture}
=
\tr(x^*xww^*)
\leq \|x^*x\|_{P_{\overline{\gamma}\to\overline{\gamma}}} \| ww^*\|_{L^2(F_0(P_\bullet))}
\end{equation}
where $\|\cdot\|_{P_{\overline{\gamma}\to\overline{\gamma}}}$ is the operator norm in the $C^*$-algebra $P_{\overline{\gamma}\to\overline{\gamma}}$.
Hence using Equation \eqref{eqn:TwoNorm} repeatedly, we have
$$
\|x \star w\|_{L^{2}(F_{0}(P))} \leq \left(\sum_{\alpha=\overline{\beta}\gamma} \|x\|_{P_{\overline{\gamma}\to\overline{\gamma}}}\right) \cdot \|w\|_{L^{2}(F_{0}(P_\bullet))},
$$
and thus left multiplication is bounded on $F_0(P_\bullet)$. The boundedness of right multiplication is similar.

Since the multiplication is bounded, we can represent $F_{0}(P_\bullet)$ on $L^{2}(F_{0}(P_\bullet))$ acting by left multiplication.  We denote $M_{0} = F_{0}(P_\bullet)''$.  We also use $M_{0}$ to denote $\Gr_{0}(P_\bullet)''$ acting on $L^{2}(\Gr_{0}(P_\bullet))$, but it will be clear from context which picture we are using.  Of course, from the discussion above, both von Neumann algebras are isomorphic.

Given $\alpha \in \Lambda$, we draw a \textcolor{\alphacolor}{\alphacolor} string for a string labelled $\alpha$. We will provide the $\alpha$ label only when it is possible to confuse $\alpha$ and $\overline{\alpha}$.

We define the graded algebra $\Gr_{\alpha}(P_\bullet) = \bigoplus_{\beta \in \Lambda} P_{\overline{\alpha} \beta \alpha}$ with multiplication $\wedge_{\alpha}$ by
$$
x \wedge_{\alpha} y =
\begin{tikzpicture}[baseline = -.1cm]
	\draw [thick, \alphacolor] (-.8,0)--(2,0);	
	\draw (0,0)--(0,.8);
	\draw (1.2,0)--(1.2,.8);
	\filldraw[unshaded,thick] (-.4,.4)--(.4,.4)--(.4,-.4)--(-.4,-.4)--(-.4,.4);
	\filldraw[unshaded,thick] (.8,.4)--(1.6,.4)--(1.6,-.4)--(.8,-.4)--(.8,.4);
	\node at (0,0) {$x$};
	\node at (1.2,0) {$y$};
%	\node at (-.6,.2) {{\scriptsize{$\alpha$}}};
%	\node at (.6,.2) {{\scriptsize{$\alpha$}}};
%	\node at (1.8,.2) {{\scriptsize{$\alpha$}}};
	\node at (.2,.6) {{\scriptsize{$\beta$}}};
	\node at (1.4,.6) {{\scriptsize{$\gamma$}}};
\end{tikzpicture}
$$
for $x\in P_{\overline{\alpha}\beta\alpha}$ and $x\in P_{\overline{\alpha}\gamma\alpha}$, and trace
$$
\tr(x) = \frac{1}{\delta^{\alpha}}
\begin{tikzpicture}[baseline=.3cm]
	\draw (0,0)--(0,.8);
	\draw [thick, \alphacolor] (.4,0) arc (90:-90:.4cm) -- (-.4,-.8) arc (270:90:.4cm);
	\filldraw[unshaded,thick] (-.4,.4)--(.4,.4)--(.4,-.4)--(-.4,-.4)--(-.4,.4);
	\draw[thick, unshaded] (-.7, .8) -- (-.7, 1.6) -- (.7, 1.6) -- (.7,.8) -- (-.7, .8);
	\node at (0,0) {$x$};
	\node at (0,1.2) {$\Sigma CTL$};
	\node at (.2,.6) {{\scriptsize{$\beta$}}};
	\node at (.6,.2) {{\scriptsize{$\alpha$}}};
	\node at (-.6,.2) {{\scriptsize{$\alpha$}}};
	\node at (0,-.6) {{\scriptsize{$\overline{\alpha}$}}};
\end{tikzpicture}.
$$
\begin{note}\label{note:ReverseMultiplication}
Be warned that the multiplication $\wedge_{\alpha}$ in the GJSW diagrams when restricted to $P_{\overline{\alpha} \to \overline{\alpha}}$ is in the opposite order with the multiplication in the introduction!
\end{note}
The inner product
$$
\langle \cdot, \cdot \rangle_{F_\alpha(P_\bullet)}=
\begin{tikzpicture}[baseline = -.1cm]
	\draw [thick, \alphacolor] (-.4,0)--(1.6,0) arc (90:-90:.4cm) -- (-.4,-.8) arc (270:90:.4cm);	
	\draw (0, .4) .. controls ++(90:.4cm) and ++(90:.4cm) .. (1.2,.4);
	\filldraw[unshaded,thick] (-.4,.4)--(.4,.4)--(.4,-.4)--(-.4,-.4)--(-.4,.4);
	\filldraw[unshaded,thick] (.8,.4)--(1.6,.4)--(1.6,-.4)--(.8,-.4)--(.8,.4);
	\node at (0,0) {};
	\node at (1.2,0) {};
%	\node at (.6,.2) {{\scriptsize{$\alpha$}}};
	\node at (-.6,.2) {{\scriptsize{$\alpha$}}};
%	\node at (1.8,.2) {{\scriptsize{$\alpha$}}};
\end{tikzpicture}
$$
makes $P_{\overline{\alpha}\beta\alpha}\perp P_{\overline{\alpha}\gamma\alpha}$ for $\beta\neq \gamma$. We get a trace on $F_\alpha(P_\bullet)$ by $\tr_{F_\alpha(P_\bullet)}(x)=\langle x, 1_\alpha\rangle_{F_\alpha(P_\bullet)}$ where $1_\alpha$ the the horizontal strand labelled $\alpha$.

The multiplication $\star_\alpha$ given by
$$
x \star_\alpha y =
\sum_{
\substack{
\kappa \text{ s.t.}\\
\beta=\beta'\kappa\\
\gamma=\overline{\kappa}\gamma'
}}
\begin{tikzpicture}[baseline = .6cm]
	\draw [thick, \alphacolor] (-.8,0)--(2,0);
	\draw (-.2,0)--(-.2,1);	
	\draw (1.4,0)--(1.4,1);
	\draw (.2,.4) arc (180:0:.4);	
	\filldraw[unshaded,thick] (-.4,.4)--(.4,.4)--(.4,-.4)--(-.4,-.4)--(-.4,.4);
	\filldraw[unshaded,thick] (.8,.4)--(1.6,.4)--(1.6,-.4)--(.8,-.4)--(.8,.4);
	\node at (0,0) {$x$};
	\node at (1.2,0) {$y$};
%	\node at (-.6,.2) {{\scriptsize{$\alpha$}}};
%	\node at (.6,.2) {{\scriptsize{$\alpha$}}};
%	\node at (1.8,.2) {{\scriptsize{$\alpha$}}};
	\node at (-.4,.6) {{\scriptsize{$\beta'$}}};
	\node at (1.6,.6) {{\scriptsize{$\gamma'$}}};
	\node at (.6,1) {{\scriptsize{$\kappa$}}};
\end{tikzpicture}
$$
for $x\in P_{\overline{\alpha}\beta\alpha}$ and $x\in P_{\overline{\alpha}\gamma\alpha}$ makes $F_{\alpha}(P)$ isomorphic as a $*$-algebra to $\Gr_{\alpha}(P_\bullet)$, preserving the inner product.  The same techniques as above with heavier notation show that the inner product on $F_{\alpha}(P_\bullet)$ (hence $\Gr_{\alpha}(P_\bullet)$) is positive definite and that left and right multiplication in $F_{\alpha}(P_\bullet)$ (hence $\Gr_{\alpha}(P_\bullet)$) is bounded.  We can therefore form the von Neumann algebra $M_{\alpha} = (F_{\alpha}(P_\bullet))''$ acting by left multiplication on $L^{2}(F_{\alpha})(P_\bullet)$ (or $(\Gr_{\alpha}(P_\bullet))''$ acting by left multiplication on $L^{2}(\Gr_{\alpha})(P_\bullet)$).  Again, it will be clear from context which picture we are considering.

%%%%%%%%%%%%%%%%%%%%%%%%%%%%%%%%%%%%%%%%%%%%%%%%%%%%%%%%%%%%%%%%%%
\subsection{Factorality of $M_{\alpha}$} \label{sec:factor}

In this section, we aim to prove the following theorem:  \begin{thm} \label{thm:Factor}

The algebra $M_{\alpha}$ is a $II_{1}$ factor. We have an embedding $\iota_{\alpha}: M_0 \hookrightarrow M_{\alpha}$ which is the extension of the map $F_0(P_\bullet)\to F_\alpha(P_\bullet)$ given by
$$
\iota_{\alpha}
\left(
\Mbox{}{x}
\right)
=
\begin{tikzpicture}[baseline = 0cm]
	\draw (0,0)--(0,.8);
	\draw [thick, \alphacolor] (-.8,-.6)--(.8,-.6);
	\filldraw[unshaded,thick] (-.4,.4)--(.4,.4)--(.4,-.4)--(-.4,-.4)--(-.4,.4);
%	\node at (0,-.75) {{\scriptsize{$\alpha$}}};
	\node at (0,0) {$x$};
%	\node at (.6,.-.4) {{\scriptsize{$\alpha$}}};
\end{tikzpicture},
$$
and $\iota_{\alpha}(M_{0})' \cap M_{\alpha} = P_{\overline{\alpha}\to\overline{\alpha}}\op\cong P_{\alpha\to \alpha}$.
\end{thm}

Throughout this subsection, we use the orthogonal picture $M_{\alpha} = (F_{\alpha}(P_\bullet))''$.  Pick a specific color, $c\in\cL$, which we will denote by the color \textcolor{\cupcolor}{green}, and recall $\delta_c>1$.
Let $A$ be the abelian von Neumann subalgebra of $M_{\alpha}$ generated by the cup element
$$
\bluecup_{\alpha}
 =
 \begin{tikzpicture} [baseline = 0cm]
 	\draw[thick, \alphacolor]  (-.4, -.2) -- (.4, -.2);
	\draw[\cupcolor, thick] (-.25, .4) arc(-180:0: .25cm);
%	\node at (-.3, .2)  {{\scriptsize{$c$}}};
%	\node at (0, -.35) {{\scriptsize{$\alpha$}}};
 	\draw[thick] (-.4, -.4) -- (-.4, .4) -- (.4, .4) -- (.4, -.4) -- (-.4, -.4);
\end{tikzpicture}\, .
$$
We will obtain the factorality of $M_{\alpha}$ be first examining $L^{2}(M_{\alpha})$ as an $A-A$ bimodule.

To begin, assuming $|\gamma| \geq 2$, we set
$$
V_{\gamma} = \left\{x \in P_{\overline{\alpha} \gamma \alpha} :
 \begin{tikzpicture} [baseline = 0cm]
 	\draw [thick, \alphacolor] (-.9, 0) -- (.9,0);
	\draw (.3,.4)--(.3,.8);
	\draw[thick, \cupcolor] (-.3, .4) arc (180:0: .2cm);
 	\draw[thick, unshaded] (-.5, -.4) -- (-.5, .4) -- (.5, .4) -- (.5, -.4) -- (-.5, -.4);
	\node at (-.1, .75)  {{\scriptsize{$c$}}};
%	\node at (-.7, .2) {{\scriptsize{$\alpha$}}};
%	\node at (.7, .2) {{\scriptsize{$\alpha$}}};
	\node at (0,0) {$x$};
\end{tikzpicture}
=
 \begin{tikzpicture} [baseline = 0cm]
 	\draw [thick, \alphacolor] (-.9, 0) -- (.9,0);
	\draw (-.3,.4)--(-.3,.8);
	\draw[thick, \cupcolor] (-.1, .4) arc (180:0: .2cm);
 	\draw[thick, unshaded] (-.5, -.4) -- (-.5, .4) -- (.5, .4) -- (.5, -.4) -- (-.5, -.4);
	\node at (.1, .75)  {{\scriptsize{$c$}}};
%	\node at (-.7, .2) {{\scriptsize{$\alpha$}}};
%	\node at (.7, .2) {{\scriptsize{$\alpha$}}};
	\node at (0,0) {$x$};
\end{tikzpicture}
= 0 \right\},
$$
$V = \oplus_{\gamma \in \Lambda} V_{\gamma}$ and $\cV$ the $A-A$ bimodule generated by $V$.  For $b\in\cL$, let $\cV_{b}$ be the $A-A$ bimodule generated by elements of the form
$$
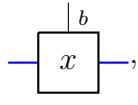
\begin{tikzpicture}[baseline = 0cm]
	\draw (0,0)--(0,.8);
	\draw [thick, \alphacolor] (-.8,0)--(.8,0);
	\filldraw[unshaded,thick] (-.4,.4)--(.4,.4)--(.4,-.4)--(-.4,-.4)--(-.4,.4);
	\node at (0,0) {$x$};
	\node at (.2,.6) {{\scriptsize{$b$}}};
%	\node at (-.6,.2) {{\scriptsize{$\alpha$}}};
%	\node at (.6,.2) {{\scriptsize{$\alpha$}}};
\end{tikzpicture},
$$
and let $W$ be the $A-A$ bimodule generated by $P_{\alpha\to\alpha}$.
We claim that we have the following decomposition:

\begin{lem} \label{lem:decomposition}
 As $A-A$ bimodules, $L^{2}(M_{\alpha}) = \displaystyle W\oplus \cV\oplus  \bigoplus_{b \in \cL} \cV_{b}$.
\end{lem}

\begin{proof}
The proof is exactly the same as \cite{1202.1298} Proposition 2.1 except that we induct on the length of a word in $\Lambda$ as opposed to the number of strands of a single color.
\end{proof}

Notice that we can decompose $V$ further as $V = V_{cc} \oplus V_{co} \oplus V_{oc} \oplus V_{oo}$.  Here, $V_{cc}$ consists of elements in $V$ whose leftmost and rightmost strings are colored $c$, $V_{co}$ consists of elements of $V$ whose leftmost string is colored $c$ and rightmost string is colored differently than $c$, $V_{oc}$ consists of elements of $V$ whose leftmost string is colored differently than $c$ and whose rightmost string is colored $c$, and $V_{oo}$ consists of the elements in $V$ whose leftmost and rightmost strings are colored other than $c$.  We set $\cV_{cc}$ the $A-A$ bimodule generated by $V_{cc}$ and we define $\cV_{co}$, $\cV_{oc}$, and $\cV_{oo}$ analogously.

Let $(V_{cc})_{n}$ be the subspace of $V_{cc}$ spanned by boxes with a word of length $n$ on top and let $\{\zeta_{n, i}\}$ be an orthonormal basis for $(V_{cc})_{n}$.  It straightforward to see that the set
$$
\left\{
\zeta_{n, i}^{l, r} = \delta_{c}^{-(l + r)/2}\,
\begin{tikzpicture} [baseline = -.1 cm]
	\draw (0,.4) -- (0, .7);
	\draw [thick, \cupcolor] (-1.8, .7) arc(-180:0: .15cm);
	\draw [thick, \cupcolor] (-.8, .7) arc(-180:0: .15cm);
	\draw [thick, \cupcolor] (.4, .7) arc(-180:0: .15cm);
	\draw [thick, \cupcolor] (1.4, .7) arc(-180:0: .15cm);
	\draw [thick, \alphacolor] (-2, 0) -- (2, 0);
	\filldraw[unshaded,thick] (-.4,.4)--(.4,.4)--(.4,-.4)--(-.4,-.4)--(-.4,.4);
	\node at (0, 0) {$\zeta_{n,i}$};
	\node at (-1.1, .6) {{\scriptsize{$\cdots$}}};
	\node at (1.1, .6)  {{\scriptsize{$\cdots$}}};
	\node at (1.1, .4)  {{\scriptsize{$r$}}};
	\node at (-1.1, .4)  {{\scriptsize{$l$}}};
	\draw [thick] (-2, .7) -- (2, .7) -- (2, -.7) -- (-2, -.7) -- (-2, .7);
\end{tikzpicture}\,
\colon
n, \, l, \, r \in \N
\right\}
 $$
is an orthonormal basis for $\cV_{cc}$.  There are similar orthonormal bases for $\cV_{co}$, $\cV_{oc}$ and $\cV_{oo}$. Let $B$ denote $V_{cc}$, $V_{co}$, $V_{oc}$, $V_{oo}$ or $P_{\overline{\alpha} b \alpha}$ for $b \neq c$ and let $\cB$ denote $\cV_{cc}$, $\cV_{co}$, $\cV_{oc}$, $\cV_{oo}$, or $\cV_{b}$.   Let $\pi,\rho$ denote the left, right representation of $M_{0}$ on $L^{2}(M_{0})$ respectively. We have the following lemma whose proof is straightforward:

\begin{lem} \label{lem:unitary}
Let $\phi: \cB \rightarrow \ell^{2}(\N) \otimes B \otimes \ell^{2}(\N)$ be defined on the orthonormal basis of $\cB$ by
$$
\phi\left(\delta_{c}^{-(l + r)/2} \,
\begin{tikzpicture} [baseline = -.1 cm]
	\draw (0,.4) -- (0, .7);
	\draw [thick, \cupcolor] (-1.8, .7) arc(-180:0: .15cm);
	\draw [thick, \cupcolor] (-.8, .7) arc(-180:0: .15cm);
	\draw [thick, \cupcolor] (.4, .7) arc(-180:0: .15cm);
	\draw [thick, \cupcolor] (1.4, .7) arc(-180:0: .15cm);
	\draw [thick, \alphacolor] (-2, 0) -- (2, 0);
	\filldraw[unshaded,thick] (-.4,.4)--(.4,.4)--(.4,-.4)--(-.4,-.4)--(-.4,.4);
	\node at (0, 0) {$v$};
	\node at (-1.1, .6) {{\scriptsize{$\cdots$}}};
	\node at (1.1, .6)  {{\scriptsize{$\cdots$}}};
	\node at (1.1, .4)  {{\scriptsize{$r$}}};
	\node at (-1.1, .4)  {{\scriptsize{$l$}}};
	\draw [thick] (-2, .7) -- (2, .7) -- (2, -.7) -- (-2, -.7) -- (-2, .7);
\end{tikzpicture}\,
\right)
= \xi_{l} \otimes v \otimes \xi_{r}
$$
with $v \in B_{n}$, where $\set{\xi_i}{i\in\N}$ is the usual orthonormal basis of $\ell^2(\N)$.  Then $\phi$ extends to a unitary operator and we have the following representations of $\bluecup_{\alpha}$:

\begin{itemize}
\item
If $\cB = \cV_{cc}$, $\cV_{co}$, $\cV_{oc}$ or $\cV_{oo}$ then
$$
\phi \cdot \pi\left(\frac{\bluecup_{\alpha} - 1}{\delta_{c}^{1/2}}\right)\cdot \phi^{*} = x \otimes 1 \otimes 1 \text{ where }
$$
$x$ is $s + s^{*}$ if the top leftmost color in for boxes in $B$ is $c$ and $s + s^{*} - \delta_{c}^{-1/2}e_{\xi_{0}}$ if the this color differs from $c$.  Similarly,
$$
\phi \cdot \rho\left(\frac{\bluecup_{\alpha} - 1}{\delta_{c}^{1/2}}\right)\cdot \phi^{*} = 1 \otimes 1 \otimes x \text{ where }
$$
$x$ is $s + s^{*}$ if the top rightmost color in for boxes in $B$ is $c$ and $s + s^{*} - \delta_{c}^{-1/2}e_{\xi_{0}}$ if the this color differs from $c$.
\item
If $\cB = \cV_{b}$ and $b \neq c$ then
\begin{align*}
\phi\cdot \pi\left(\frac{\bluecup_{\alpha} - 1}{\delta_{c}^{1/2}}\right)\cdot \phi^{*}
&= (s + s^{*} - \delta_{c}^{-1/2}e_{\xi_{0}}) \otimes 1 \otimes 1 \textrm{ and }\\
\phi\cdot \rho\left(\frac{\bluecup_{\alpha} - 1}{\delta_{c}^{1/2}}\right)\cdot \phi^{*}
&= 1 \otimes 1 \otimes (s + s^{*} - \delta_{c}^{-1/2}e_{\xi_{0}}).
\end{align*}
\item
(See \cite{1202.1298}) If $\cB = \cV_{c}$ then there is a unitary (which is \underline{not} $\phi$) $v: \cV_{c} \rightarrow \ell^{2}(\cN) \otimes L^{2}(P_{\overline{\alpha}c\alpha}) \otimes \ell^{2}(\N)$ such that
\begin{align*}
v \cdot \pi\left(\frac{\bluecup_{\alpha} - 1}{\delta_{c}^{1/2}}\right)\cdot v^{*} &= (s+s^{*}) \otimes 1 \otimes 1 \textrm{ and }\\
v\cdot \rho\left(\frac{\bluecup_{\alpha} - 1}{\delta_{c}^{1/2}}\right)\cdot v^{*} &= 1 \otimes 1 \otimes (s + s^{*}).
\end{align*}
\end{itemize}
Here, $s$ is the unilateral shift operator on $\ell^{2}(\N)$ and $e_{\xi_{0}}$ is the orthogonal projection in $\cB(\ell^{2}(\N))$ onto the one-dimensional space spanned by $\xi_{0}$.
\end{lem}

We now show that the operators $s+s^{*}$ and $s + s^{*} - \delta_{c}^{-1/2}e_{\xi_{0}}$ are unitary equivalent in $\cB(\ell^{2}(\N))$.  We begin with the following lemma
\begin{lem} \label{lem:essential}
The spectra of $s+s^{*}$ and $y = s + s^{*} - \delta_{c}^{-1/2}e_{\xi_{0}}$ are the same.
\end{lem}

\begin{proof}
Since the operators differ by a finite rank operator, by the Weyl-von Neumann Theorem (see \cite{MR1335452} p.523), they have the same essential spectrum.  The operator $s+s^{*}$ has essential spectrum $[-2, 2]$ and since the complement of the essential spectrum in the spectrum is an isolated set of eigenvalues, we just need to show that $y$ has no eigenvalues outside of $[-2, 2]$ since $y$ is self adjoint and must have real spectrum.

To this end, let $\lambda > 2$ and $\xi = \sum_{j} x_{j}\xi_{j}$.  If $\xi$ is an eigenvector of $y$ with eigenvalue $\lambda$ then we have the equations
$$
x_{1} = \left(\lambda + \delta_{c}^{-1/2}\right)x_{0} \textrm{ and } \lambda x_{n+1} = x_{n} + x_{n+2}
$$
for $n \geq 0$. The characteristic equation for this linear recurrence is $x^{2} - \lambda x + 1 = 0$ which has roots
$$
l = \frac{\lambda + \sqrt{\lambda^{2} - 4}}{2} \textrm{ and } r = \frac{\lambda - \sqrt{\lambda^{2} - 4}}{2}.
$$
This implies the existence of constants $C$ and $D$ with $x_{n} = C\cdot l^{n} + D \cdot r^{n}$ for all $n$.  Since the sequence $x_{n}$ is $\ell^{2}$, $C = 0$ meaning $x_{n} = x_{0}\cdot r^{n}$ for all $n$.  However, this means $x_{1} = (\lambda + 1/\sqrt{\delta_{c}})x_{0}$ and $x_{1} = rx_{0}$ must both be satisfied.  This can only happen if $x_{0} = 0$ since $\lambda + 1/\sqrt{\delta_{c}} > 1$ and $r < 1$.  This implies $\xi = 0$.\\

Now let $\lambda < -2$.  With $l$ and $r$ as above we must have $x_{n} = C \cdot l^{n} + D \cdot r^{n}$.  This time we must have $D = 0$ and we obtain the equation $x_{n} = x_{0}l^{n}$.  This gives us the two equations
$$
x_{1} = \left(\lambda + \delta_{c}^{-1/2}\right)x_{0} \text{ and }  x_{1} = l x_{0}
$$
which implies $(\lambda - \frac{1}{2}(\lambda + \sqrt{\lambda^{2} - 4}) + \delta_{c}^{-1/2})x_{0} = 0$.  This forces $x_{0}$ to be 0 since by the choice of $\lambda$, $\lambda - \frac{1}{2}(\lambda + \sqrt{\lambda^{2} - 4}) < -1$ and $|\delta_{c}^{-1/2}| < 1$.  Therefore $\xi=0$ and $y$ has no eigenvectors.
\end{proof}

For any self-adjoint operator $a \in B(\cH)$, we set $\sigma(a)$ to be the spectrum of $a$.  Given $\xi \in \cH$, a Radon measure $\mu_{\xi}$ on the real line is induced by the formula $\mu_{\xi}(f) = \langle f(a)\xi,  \xi \rangle$ for any bounded continuous $f$.  We set $\cH_{ac}$ the Hilbert space of vectors $\xi$ where $\mu_{\xi}$ is absolutely continuous with respect to the Lebesgue measure, $\cH_{sc}$ the Hilbert space of vectors $\xi$ where $\mu_{\xi}$ is singular with respect to the Lebesgue measure, and $\cH_{pp}$ the Hilbert space of vectors $\xi$ where $\mu_{\xi}$ has purely atomic measure.  We define the absolutely continuous spectrum of $a$ as the spectrum of $a$ on $\cH_{ac}$ and denote it as  $\sigma_{ac}(a)$.  We define $\sigma_{cc}(a)$ and $\sigma_{pp}(a)$ in a similar manner.

We say $a$ has uniform multiplicity 1 if there is a measure $\mu$ on the spectrum of $a$ and a unitary $w: \cH \rightarrow L^{2}(\sigma(a), \mu)$ such that for any $f \in L^{2}(\sigma(a), \mu)$, $waw^{*}(f)(x) = xf(x)$.  In this case $\sigma_{ac}(a)$, $\sigma_{cc}(a)$ and $\sigma_{pp}(a)$ form a partition of $\sigma(a)$.  See \cite{MR1335452} for more details.

\begin{lem}

$y$ is unitary equivalent to $s + s^{*}$.

\end{lem}

\begin{proof}

We first show that $y$ has uniform multiplicity $1$.  To this end, consider the following sequence of polynomials:
\begin{align*}
&p_{0}(x) = 1 \\
&p_{1}(x) = x + \delta_{c}^{-1/2}\\
&p_{n+2}(x) = xp_{n+1} - p_{n}(x) \textrm{ for } n \geq 0
\end{align*}

By induction, it is straightforward to check that $p_{n}(y)(\xi_{0}) = \xi_{n}$ and hence the map $w: \cH \rightarrow L^{2}(\sigma(a), \mu_{\xi_{0}})$ given by $w(\xi_{n}) = p_{n}$ is a unitary satisfying $wyw^{*}(f)(x) = xf(x)$ for all $f \in L^{2}(\sigma(a), \mu_{\xi_{0}})$.  By the Kato-Rosenblum theorem (see \cite{MR1335452} p.540), if $a$ and $b$ are self adjoint operators and $b$ is trace class then the absolutely continuous parts of $b$ and $a+b$ are unitary equivalent.  This implies that the absolutely continuous parts of $y$ and $s + s^{*}$ are unitary equivalent.  The absolutely continuous spectrum of $s+s^{*}$ is $[-2, 2] = \sigma(s + s^{*})$ so the same must hold for $y$.  As was discussed in the previous lemma, $y$ has no spectral values outside $[-2, 2]$ and from above, the spectrum of $y$ is partitioned into $\sigma_{ac}$, $\sigma_{cc}$ and $\sigma_{pp}$.  Therefore $\sigma_{pp} = \sigma_{cc} = \emptyset$ and $\sigma_{ac} = \sigma$ showing the unitary equivalence.
\end{proof}

Therefore by conjugating by a unitary, we may assume that $L^{2}(M_{\alpha}) \cong L^{2}(A) \oplus \ell^{2}(\N) \otimes \cH \otimes \ell^{2}(\N)$ for some Hilbert space $\cH$ where on the second summand, $\pi\left(\frac{\cup - 1}{\sqrt{\delta_{c}}}\right)$ acts as $(s + s^{*}) \otimes 1 \otimes 1$ and $\rho\left(\frac{\cup - 1}{\sqrt{\delta_{c}}}\right)$ acts as $1 \otimes 1 \otimes (s+s^{*})$. The lemmas below follow \cite{MR2645882} but we supply proofs for the readers' convenience.

\begin{lem} \label{lem:shift}
$A' \cap M_{\alpha} = AP_{\overline{\alpha}\to\overline{\alpha}}\op$.
\end{lem}

\begin{proof}
First, note that $AP_{\overline{\alpha}\to\overline{\alpha}}\op\subseteq A'\cap M_\alpha$ in the obvious way
$$
\begin{tikzpicture} [baseline = -.1cm]
	\draw [thick, \cupcolor] (-1.4, .8) arc(180:360: .3cm);
	\draw [thick, \alphacolor] (-1.4, 0)--(.8, 0);
	\filldraw[thick, unshaded] (-.4, -.4) -- (-.4, .4) -- (.4, .4) -- (.4, -.4) -- (-.4, -.4);
	\node at (0,0) {$x$};
\end{tikzpicture}
=
\begin{tikzpicture} [baseline = -.1cm]
	\draw [thick, \cupcolor] (-.3, .8) arc(180:360: .3cm);
	\draw [thick, \alphacolor] (-.8, 0)--(.8, 0);
	\filldraw[thick, unshaded] (-.4, -.4) -- (-.4, .4) -- (.4, .4) -- (.4, -.4) -- (-.4, -.4);
	\node at (0,0) {$x$};
\end{tikzpicture}
=
\begin{tikzpicture} [baseline = -.1cm]
	\draw [thick, \cupcolor] (.8, .8) arc(180:360: .3cm);
	\draw [thick, \alphacolor] (-.8, 0)--(.8, 0);
	\filldraw[thick, unshaded] (-.4, -.4) -- (-.4, .4) -- (.4, .4) -- (.4, -.4) -- (-.4, -.4);
	\node at (0,0) {$x$};
\end{tikzpicture}
$$
where we need the ``op" since multiplication in the GJSW picture happens in the opposite order (see Note \ref{note:ReverseMultiplication}).

If $A' \cap M_{\alpha}$ were larger than $AP_{\overline{\alpha}\to\overline{\alpha}}\op$ then by looking at the orthogonal compliment of $AP_{\overline{\alpha}\to\overline{\alpha}}\op$ in $L^{2}(M_{\alpha})$, there is a nonzero vector $\xi \in \ell^{2}(\N) \otimes \ell^{2}(\N)$ with $((s+s^{*}) \otimes 1)\xi = (1 \otimes s+s^{*})\xi$.  Viewing $\ell^{2}(\N) \otimes \ell^{2}(\N)$ as the Hilbert Schmidt operators on $\ell^{2}(\N)$ this means $\xi(s+s^{*}) = (s+s^{*})\xi$ which is impossible since $s+s^{*}$ has no eigenvalues.
\end{proof}

We will realize $M_{0}$ as a unital subalgebra of $M_{\alpha}$ via the map $\iota_{\alpha}$.

\begin{lem} \label{lem:RelativeCommutant}
$M_{0}' \cap M_{\alpha} = P_{\overline{\alpha}\to\overline{\alpha}}\op\cong P_{\alpha\to \alpha}$ as an algebra.
\end{lem}

\begin{proof}

If $x \in M_{0}' \cap M_{\alpha}$ then by Lemma \ref{lem:shift}, $x \in AP_{\overline{\alpha}\to\overline{\alpha}}\op$, so we write $x$ as an $\ell^{2}$ sum
$$
x = \sum_{n=0}^{\infty}\frac{1}{\sqrt{\delta_{c}^{n}}}\cdot x_{n}
\star
\,
\begin{tikzpicture} [baseline = -.1cm]
	\draw [thick, \cupcolor] (-.8, .4) arc(180:360: .2cm);
	\draw [thick, \cupcolor] (.8, .4)  arc(0:-180: .2cm);
	\node at (0, .3) {{\scriptsize{$\cdots$}}};
	\node at (0,.1) {{\scriptsize{$n$}}};
	\draw [thick, \alphacolor] (-1, -.2)--(1, -.2);
	\draw [thick] (-1, -.4) -- (-1, .4) -- (1, .4) -- (1, -.4) -- (-1, -.4);
\end{tikzpicture}
\, ,
$$
where $x_{n} \in P_{\overline{\alpha}\to\overline{\alpha}}\op$ for all $n$.
Consider the following elements of $M_{0}$:
$$
z =
\begin{tikzpicture} [baseline = -.1cm]
	\draw [thick, \cupcolor] (-.3, .4) arc(180:360: .3cm);
	\draw [thick, \cupcolor] (-.15, .4) arc(180:360: .15cm);
	\draw [thick, \alphacolor] (-.4, -.1)--(.4, -.1);
	\draw[thick] (-.4, -.4) -- (-.4, .4) -- (.4, .4) -- (.4, -.4) -- (-.4, -.4);
\end{tikzpicture} \,,
l_{n}
=  \frac{1}{\sqrt{\delta_{c}^{n}}}\,
\begin{tikzpicture} [baseline = -.1cm]
	\draw [thick, \cupcolor] (-.8, .4) arc(180:360: .2cm);
	\draw [thick, \cupcolor] (.8, .4)  arc(0:-180: .2cm);
	\draw [thick, \cupcolor] (-1.8, .4) arc(180:360: .4cm);
	\draw [thick, \cupcolor](-1.6, .4) arc(180:360: .2cm);
	\draw [thick, \alphacolor] (-2, -.2)--(1, -.2);
	\node at (0, .3) {{\scriptsize{$\cdots$}}};
	\node at (0,.1) {{\scriptsize{$n$}}};
	\draw [thick] (-2, -.4) -- (-2, .4) -- (1, .4) -- (1, -.4) -- (-2, -.4);
\end{tikzpicture}\,,
\text{ and }
r_{n} = \frac{1}{\sqrt{\delta_{c}^{n}}} \,
\begin{tikzpicture} [baseline = -.1cm]
	\draw [thick, \cupcolor] (.4, .4) arc(-180:0: .2cm);
	\draw [thick, \cupcolor] (-.8, .4)  arc(-180:0: .2cm);
	\draw [thick, \cupcolor] (1, .4) arc(180:360: .4cm);
	\draw [thick, \cupcolor](1.2, .4) arc(180:360: .2cm);
	\draw [thick, \alphacolor] (2, -.2)--(-1, -.2);
	\node at (0, .3) {{\scriptsize{$\cdots$}}};
	\node at (0,.1) {{\scriptsize{$n$}}};
	\draw [thick] (2, -.4) -- (2, .4) -- (-1, .4) -- (-1, -.4) -- (2, -.4);
\end{tikzpicture}\,.
$$
By a direct diagrammatic computation, $z\star x - x\star z$ is the orthogonal sum
$$
\sum_{n=1}^{\infty} (x_{n} + \frac{1}{\sqrt{\delta_{c}}}\cdot x_{n+1})\star(l_{n} - r_{n}).
$$
For this to be zero, we must have $x_{n+1} = -\sqrt{\delta_{c}}\cdot x_{n}$ for $n \geq 1$.  Since the sum for $x$ must be $\ell^{2}$, this implies $x_{n} = 0$ for $n \geq 1$, i.e., $x \in P_{\overline{\alpha}\to\overline{\alpha}}\op$.

Obviously the rotation by 180 degrees gives the isomorphism $P_{\overline{\alpha}\to\overline{\alpha}}\op\cong P_{\alpha\to \alpha}$.
\end{proof}

\begin{proof} [Proof of Theorem \ref{thm:Factor}]
If $x$ were in the center of $M_{\alpha}$ then we know that $x \in P_{\overline{\alpha}\to\overline{\alpha}}\op$.  The element $x$ must commute with
$$
\begin{tikzpicture}[baseline = -.1cm]
	\draw [thick, \alphacolor] (-.4, .1) arc(-90:0: .3cm);
	\draw [thick, \alphacolor] (.4, .1) arc(-90:-180: .3cm);
	\draw[thick] (-.4, -.4) -- (-.4, .4) -- (.4, .4) -- (.4, -.4) -- (-.4, -.4);
\end{tikzpicture}\,,
$$
giving the equation
$$
\begin{tikzpicture} [baseline = -.1cm]
	\draw[thick, \alphacolor]  (-1,0) -- (-.4, 0) arc(-90:0: .3cm) -- (-.1,.4);
	\draw[thick, \alphacolor] (1, 0) -- (.4, 0) arc(-90:-180: .3cm)--(.1,.4);
	\draw[thick] (-1, -.4) -- (-1, .4) -- (1, .4) -- (1, -.4) -- (-1, -.4);
	\draw[thick, unshaded] (-.8, -.2) -- (-.8, .2) -- (-.4, .2) -- (-.4, -.2) -- (-.8, -.2);
	\node at (-.6, 0) {$x$};
\end{tikzpicture}
=
\begin{tikzpicture} [baseline = -.1cm]
	\draw[thick, \alphacolor]  (-1,0) -- (-.4, 0) arc(-90:0: .3cm) -- (-.1,.4);
	\draw[thick, \alphacolor] (1, 0) -- (.4, 0) arc(-90:-180: .3cm)--(.1,.4);
	\draw[thick] (-1, -.4) -- (-1, .4) -- (1, .4) -- (1, -.4) -- (-1, -.4);
	\draw[thick, unshaded] (.8, -.2) -- (.8, .2) -- (.4, .2) -- (.4, -.2) -- (.8, -.2);
	\node at (.6, 0) {$x$};
\end{tikzpicture}  \, .
$$
Joining the leftmost strings to the top implies that $x$ is a scalar multiple of the identity.
\end{proof}

%%%%%%%%%%%%%%%%%%%%%%%%%%%%%%%%%%%%%%%%%%%%%%%%%%%%%%%%%%%%%%%%%%%%
\subsection{Jones' Towers Associated to $M_{0}$}\label{sec:towers}

If $\alpha$ and $\beta$ are in $\Lambda$ such that $\alpha = \beta\gamma $ with $\gamma \in \Lambda$ then we have a unital, trace-preserving inclusion $\iota: M_{\beta} \rightarrow M_{\alpha}$ given by
$$
\iota(x) =
\begin{tikzpicture}[baseline = -.1cm]
	\draw (0,0)--(0,.8);
	\draw (-.8,0)--(.8,0);
	\draw (-.8,-.6)--(.8,-.6);
	\filldraw[unshaded,thick] (-.4,.4)--(.4,.4)--(.4,-.4)--(-.4,-.4)--(-.4,.4);
	\node at (0,-.75) {{\scriptsize{$\gamma$}}};
	\node at (-.6,.2) {{\scriptsize{$\beta$}}};
	\node at (.6,.2) {{\scriptsize{$\beta$}}};
	\node at (0,0) {$x$};
\end{tikzpicture}\,.
$$
When we write $M_{\beta} \subset M_{\alpha}$ we mean that $M_{\beta}$ is included into $M_{\alpha}$ in the manner described above.  We have the following theorem:

\begin{thm} \label{thm:tower}
The following is a Jones' tower of factors:
$$
M_{0} \subset M_{\alpha} \subset M_{\alpha\overline{\alpha}} \subset \cdots \subset M_{(\alpha\overline{\alpha})^n\alpha} \subset M_{(\alpha\overline{\alpha})^{n+1}} \subset \cdots .
$$
Moreover $[M_{\alpha}:M_{0}] = \delta_{\alpha}^{2}$.
\end{thm}

\begin{proof}
This proof closely follows Section 4 in \cite{MR2645882}.  We will show that $M_{\alpha\overline{\alpha}}$ is the basic construction of $M_{0} \subset M_{\alpha}$.  The proof for higher steps in the tower is the same but with heavier notation.  To begin, set
$$
e_{0} =  \frac{1}{\delta_{\alpha}}
\begin{tikzpicture} [baseline = -.1cm]
	\draw [thick, \alphacolor] (-.4, .25) arc(90:-90: .25cm);
	\draw [thick, \alphacolor] (.4, .25) arc(90:270: .25cm);
	\draw[thick] (-.4, -.4) -- (-.4, .4) -- (.4, .4) -- (.4, -.4) -- (-.4, -.4);
\end{tikzpicture}
\in
P^{op}_{\alpha\overline{\alpha}\to \alpha\overline{\alpha}}\subset
M_{\alpha\overline{\alpha}}
$$
then $e_{0}$ is a projection.  It is a straightforward diagrammatic computation to show that if $x \in M_{\alpha}$ then $e_{0}xe_{0} = E_{M_{0}}(x)$ with $E_{M_{0}}$ the trace preserving conditional expectation.

 We now claim that $(M_{\alpha}, e_{0})''$ is a $II_{1}$ factor.  Indeed, if $y$ were in the center of $(M_{\alpha}, e_{0})''$ it would have to commute with $M_{0}$, implying that $y \in P_{\alpha\overline{\alpha}\to\alpha\overline{\alpha}}$.  The element $y$ also has to commute with $M_{\alpha}$, in particular it has to commute with the element
$$
\begin{tikzpicture}[baseline = -.1cm]
	\draw [thick, \alphacolor] (-.4, -.1) -- (.4, -.1);
	\draw [thick, \alphacolor] (-.4, .1) arc(-90:0: .3cm);
	\draw [thick, \alphacolor] (.4, .1) arc(-90:-180: .3cm);
	\node at (0, -.25) {\scriptsize{$\overline{\alpha}$}};
	\draw[thick] (-.4, -.4) -- (-.4, .4) -- (.4, .4) -- (.4, -.4) -- (-.4, -.4);
\end{tikzpicture}\,.
$$
This implies $y$ must be of the form
$$
\begin{tikzpicture} [baseline = 0cm]
	\draw [thick, \alphacolor] (-1, .3) -- (1, .3);
	\node at (0, -.4) {$y$};
	\draw [thick, \alphacolor] (-1, -.4) -- (-.4, -.4);
	\draw [thick, \alphacolor] (.4, -.4) -- (1, -.4);
	\node at (.7, -.6) {\scriptsize{$\overline{\alpha}$}};
	\node at (-.7, -.6) {\scriptsize{$\overline{\alpha}$}};
	\draw[thick] (-.4, -.8) -- (-.4, 0) -- (.4, 0) -- (.4, -.8) -- (-.4, -.8);
\end{tikzpicture}\,,
$$
but commuting with $e_{0}$ forces this diagram to be a scalar multiple of the identity.

One then observes that if $z \in (M_{\alpha}, e_{0})''$, then $ze_{0} = (\delta^{\alpha})^{2}E_{M_{\alpha}}(ze_{0})e_{0}$.  This is done by realizing that as a von Neumann algebra, $(M_{\alpha}, e_{0})''$ is generated by $M_{\alpha}$ and $M_{\alpha}e_{0}M_{\alpha}$ and hence one can assume $z$ is in either of these spaces.  The equality then becomes a straightforward diagrammatic check.  From this, one can deduce that the map $x \mapsto \delta^{u}xe_{0}$ from $M_{\alpha}$ to $(M_{\alpha}, e_{0})''$ is a surjective isometry intertwining $E_{M_{0}}$ on $L^{2}(M_{\alpha})$ and left multiplication by $e_{0}$.  From this, we deduce that $(M_{\alpha}, e_{0})''$ is the basic construction of $M_{0} \subset M_{\alpha}$ and evaluating $[M_{\alpha}:M_{0}]$ is a matter of calculating the trace of $e_{0}$.

The same arguments applied to $M_{\alpha} \subset M_{\alpha\overline{\alpha}}$ implies that $[M_{\alpha\overline{\alpha}}: M_{\alpha}] = (\delta^{\alpha})^{2}$ hence $[M_{\alpha\overline{\alpha}}: (M_{\alpha}, e_{0})''] = 1$ i.e. $M_{\alpha\overline{\alpha}} = (M_{\alpha}, e_{0})''$.
\end{proof}

As an aside, since $M_{0}' \cap M_{\beta} = P_{\beta\to\beta}$ for any $\beta$, it follows that the sequence of vector spaces $P_{0}, \, P_{\alpha\to\alpha}, \, P_{\alpha\overline{\alpha}\to\alpha\overline{\alpha}},\cdots$ forms a subfactor planar algebra.

%%%%%%%%%%%%%%%%%%%%%%%%%%%%%%%%%%%%%%%%%%%%%%%%%%%%%%%%%%%%%%
\subsection{Some bifinite bimodules over $M$} \label{sec:MoreBimodules}

\begin{nota}
We now use the notation $M$ instead of $M_0$ as it makes the rest of this section easier on the eyes.
\end{nota}

In this subsection and the next, we define a category $\sC_{bim}$ of bifinite bimodules over $M$.  To begin, we use $P_\bullet$ to construct another von Neumann algebra $\cM_{\alpha, \beta}$ which contains the factors $M_{\alpha}$ and $M_{\beta}$ as corners/cut-downs.  Let $\Gr_{\alpha, \beta}(P) = \bigoplus_{\gamma \in \Lambda} P_{\overline{\alpha}\gamma\beta}$.  If $x \in P_{\overline{\alpha}\gamma\beta}$, we view $x$ as
$$
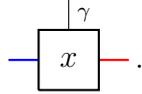
\begin{tikzpicture}[baseline=-.1cm]
	\draw [thick, \alphacolor]  (-.8,0)--(.4,0);
	\draw [thick, \betacolor] (.4, 0)--(.8, 0);
	\draw (0,.4)--(0,.8);
	\draw[thick, unshaded] (-.4, -.4) -- (-.4, .4) -- (.4, .4) -- (.4, -.4) -- (-.4, -.4);
	\node at (0, 0) {$x$};
%	\node at (-.6, .2) {\scriptsize{$\alpha$}};
	\node at (.2,.6) {\scriptsize{$\gamma$}};
%	\node at (.6, .2) {\scriptsize{$\beta$}};
\end{tikzpicture}\,.
$$
where we now draw \textcolor{\betacolor}{\betacolor} strings for strings labelled $\beta$.
There is a sesquilinear form $\langle \cdot, \cdot \rangle$ on $Gr_{\alpha, \beta}(P)$ given by
$$
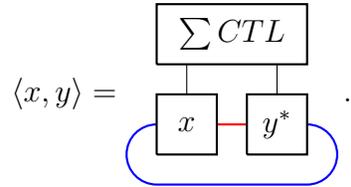
\langle x, y \rangle =
\begin{tikzpicture} [baseline = .3cm]
	\draw [thick, \alphacolor] (-.4, 0) arc(90:270: .4cm) -- (1.6, -.8) arc(-90:90: .4cm);
	\draw [thick, \betacolor]( .4, 0)--(.8, 0);
	\draw (0,.4)--(0,.8);
	\draw (1.2,.4)--(1.2,.8);
	\draw[thick, unshaded] (-.4,.8) -- (-.4, 1.6) -- (1.6,1.6) -- (1.6,.8) -- (-.4,.8);
	\draw[thick, unshaded] (-.4, -.4) -- (-.4, .4) -- (.4, .4) -- (.4, -.4) -- (-.4, -.4);
	\draw[thick, unshaded] (.8, -.4) -- (.8, .4) -- (1.6, .4) -- (1.6, -.4) -- (.8, -.4);
	\node at (0, 0) {$x$};
	\node at (1.2, 0) {$y^{*}$};
	\node at (.6, 1.2) {$\sum CTL$};
%	\node at (.6,.2) {\scriptsize{$\beta$}};
%	\node at (-.6,.2) {\scriptsize{$\alpha$}};
%	\node at (1.8,.2) {\scriptsize{$\alpha$}};
%	\node at (.6,-.6) {\scriptsize{$\overline{\alpha}$}};
\end{tikzpicture} \, .
$$
As a vector space, $\Gr_{\alpha, \alpha}(P_\bullet) = \Gr_{\alpha}(P_\bullet)$.  For $\alpha \neq \beta$, we form the $*$-algebra $\cG_{\alpha, \beta}$ which is generated by the vector spaces $\Gr_{\alpha}(P_\bullet)$, $\Gr_{\beta}(P_\bullet)$, $\Gr_{\alpha, \beta}(P_\bullet)$ and $\Gr_{\beta, \alpha}(P_\bullet)$, under the multiplication
$$
\begin{tikzpicture}[baseline=-.1cm]
	\draw (-.8,0)--(.8,0);
	\draw (0,.4)--(0,.8);
	\draw[thick, unshaded] (-.4, -.4) -- (-.4, .4) -- (.4, .4) -- (.4, -.4) -- (-.4, -.4);
	\node at (0, 0) {$x$};
	\node at (-.6, .2) {\scriptsize{$\kappa$}};
	\node at (.2,.6) {\scriptsize{$\gamma$}};
	\node at (.6, .2) {\scriptsize{$\theta$}};
\end{tikzpicture}
\wedge
\begin{tikzpicture}[baseline=-.1cm]
	\draw (-.8,0)--(.8,0);
	\draw (0,.4)--(0,.8);
	\draw[thick, unshaded] (-.4, -.4) -- (-.4, .4) -- (.4, .4) -- (.4, -.4) -- (-.4, -.4);
	\node at (0, 0) {$x$};
	\node at (-.6, .2) {\scriptsize{$\omega$}};
	\node at (.2,.6) {\scriptsize{$\gamma'$}};
	\node at (.6, .2) {\scriptsize{$\chi$}};
\end{tikzpicture}
= \delta_{\omega,\theta}
\begin{tikzpicture} [baseline = -.1cm]
	\draw (0,.4)--(0,.8);
	\draw (1.2,.4)--(1.2,.8);
	\draw (-.8,0)--(2,0);
	\draw[thick, unshaded] (-.4, -.4) -- (-.4, .4) -- (.4, .4) -- (.4, -.4) -- (-.4, -.4);
	\draw[thick, unshaded] (.8, -.4) -- (.8, .4) -- (1.6, .4) -- (1.6, -.4) -- (.8, -.4);
	\node at (0, 0) {$x$};
	\node at (1.2, 0) {$y^{*}$};
	\node at (-.6,.2) {\scriptsize{$\kappa$}};
	\node at (.2,.6) {\scriptsize{$\gamma$}};
	\node at (.6,.2) {\scriptsize{$\theta$}};
	\node at (1.4,.6) {\scriptsize{$\gamma'$}};
	\node at (1.8,.2) {\scriptsize{$\chi$}};
\end{tikzpicture}\,,
 $$
 where $\kappa,\theta,\omega,\chi\in\{\alpha,\beta\}$. There is also a (non-normalized) trace on $\cG_{\alpha, \beta}$ given by
$$
\tr(x) =
\begin{tikzpicture}[baseline=.3cm]
	\draw (0,0)--(0,.8);
	\draw (.4,0) arc (90:-90:.4cm) -- (-.4,-.8) arc (270:90:.4cm);
	\filldraw[unshaded,thick] (-.4,.4)--(.4,.4)--(.4,-.4)--(-.4,-.4)--(-.4,.4);
	\draw[thick, unshaded] (-.7, .8) -- (-.7, 1.6) -- (.7, 1.6) -- (.7,.8) -- (-.7, .8);
	\node at (0,0) {$x$};
	\node at (0,1.2) {$\Sigma CTL$};
%	\node at (.6,.2) {{\scriptsize{$\alpha$}}};
%	\node at (-.6,.2) {{\scriptsize{$\alpha$}}};
%	\node at (0,-.6) {{\scriptsize{$\overline{\alpha}$}}};
\end{tikzpicture}
$$
if $x \in \Gr_{\alpha}(P_\bullet)$ or $\Gr_{\beta}(P_\bullet)$, and is zero otherwise.  Just as in the case for the algebras $\Gr_{\alpha}(P_\bullet)$ one can show (by orthogonalizing) that the trace is positive definite and that multiplication is bounded.  Therefore, one can form the von Neumann algebra $\cM_{\alpha, \beta} = \cG_{\alpha, \beta}''$ acting on $L^{2}(\cG_{\alpha, \beta})$ by left and right multiplication.  Set
$$
p_{\alpha} =
\begin{tikzpicture} [baseline = -.1cm]
	\draw [thick, \alphacolor] (-.4, 0) -- (.4, 0);
	\node at (0, .2) {\scriptsize{$\alpha$}};
	\draw[thick] (-.4,.4)--(.4,.4)--(.4,-.4)--(-.4,-.4)--(-.4,.4);
\end{tikzpicture}
\text{ and }
p_{\beta} =
\begin{tikzpicture} [baseline = -.1cm]
	\draw [thick, \betacolor] (-.4, 0) -- (.4, 0);
	\node at (0, .2) {\scriptsize{$\beta$}};
	\draw[thick] (-.4,.4)--(.4,.4)--(.4,-.4)--(-.4,-.4)--(-.4,.4);
\end{tikzpicture} \, .
$$
We see that $p_{\gamma}\cM_{\alpha, \beta}p_{\gamma} = M_{\gamma}$ (with the non-normalized trace) for $\gamma = \{\alpha,\beta\}$ so $L^{2}(\cM_{\alpha, \beta})$ is naturally an $M_{\alpha} \oplus M_{\beta}$ bimodule.  Hence we may consider $L^{2}(\cM_{\alpha, \beta})$ as an $M-M$ bimodule via the embedding $x \mapsto \iota_{\alpha}(x) \oplus \iota_{\beta}(x) \in M_{\alpha} \oplus M_{\beta}$.  Under this identification, we define $\cH_{\alpha, \beta}$ to be the the $M-M$ bimodule
$$
\cH_{\alpha, \beta} = p_{\alpha}\wedge L^{2}(\cM_{\alpha, \beta})\wedge p_{\beta}.
$$

As mentioned above, we can give $\cG_{\alpha,\beta}$ an orthogonalized inner product and multiplication exactly as in Section \ref{sec:graded}.  We use the notation $\cF_{\alpha,\beta}$ to denote the vector space $\cG_{\alpha,\beta}$ with the orthogonalized inner product and multiplication $\star$.  The vector space $p_{\alpha} \star L^{2}(\cF_{\alpha,\beta})\star p_{\beta}$ is naturally an $M-M$ bimodule which is isomorphic to $p_{\alpha}\wedge L^{2}(\cM_{\alpha, \beta})\wedge p_{\beta}$.  Our first lemma is proven by making use of the orthogonal picture:

\begin{lem} \label{lem:central}
The vector space of $M-M$ central vectors of $p_{\alpha} \star L^{2}(\cF_{\alpha,\beta}) \star p_{\beta}$ is the vector space $P_{\overline{\alpha} \to \overline{\beta}}$ (rotating the GJS diagrams 90 degrees clockwise).
\end{lem}

\begin{proof}
As in Section \ref{sec:factor}, we let $A$ be the von Neumann subalgebra of $M$ generated by $\bluecup$ and let $W$ be the $A-A$ bimodule generated by $P_{\overline{\alpha} \to \overline{\beta}}$. With the same approach as \ref{sec:factor}, we see that as an $A-A$ bimodule, we have
$$
\cH_{\alpha, \beta} \cong W \bigoplus \left(\ell^{2}(\N) \otimes \cH \otimes \ell^{2}(\N)\right)
$$
for $\cH$ an auxiliary Hilbert space.  The operator $\bluecup$ acts on the left of the second factor by $(s + s^{*}) \otimes 1 \otimes 1$  and on the right by $1 \otimes 1 \otimes (s + s^{*})$.  Therefore the proof of Lemma \ref{lem:shift} applies here and we see that the $A-A$ central vectors of $L^{2}(\cF_{\alpha, \beta})$ are exactly $W$.  To finish the proof, one repeats the argument in Lemma \ref{lem:RelativeCommutant}.
\end{proof}

\begin{cor} \label{cor:central}
The vector space of $M-M$ central vectors of $p_{\alpha} \wedge L^{2}(\cF_{\alpha,\beta}) \wedge p_{\beta}$ is the vector space $P_{\overline{\alpha} \to \overline{\beta}}$ (rotating the GJS diagrams 90 degrees clockwise).
\end{cor}

For the rest of this section, we use the non-orthogonalized picture for $\cH_{\alpha, \beta}$.  We first write down some straightforward isomorphisms between these bimodules:

 \begin{lem} \label{lem:rotate}
 Let $\alpha, \beta, \gamma \in \Lambda$. Then as $M-M$ bimodules, $ \cH_{\alpha\beta,\gamma}\cong \cH_{\alpha, \overline{\beta}\gamma}$.
\end{lem}
 \begin{proof}
 The map
 $$
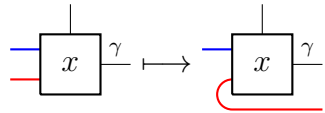
\begin{tikzpicture}[baseline=-.1cm]
	\draw (0,0)--(.8,0);
	\draw[thick, \alphacolor] (-.8,.2)--(0,.2);
	\draw[thick, \betacolor] (-.8,-.2)--(0,-.2);
	\draw (0,.4)--(0,.8);
	\draw[thick, unshaded] (-.4, -.4) -- (-.4, .4) -- (.4, .4) -- (.4, -.4) -- (-.4, -.4);
	\node at (0, 0) {$x$};
%	\node at (-.6, .4) {\scriptsize{$\alpha$}};
	\node at (.6, .2) {\scriptsize{$\gamma$}};
%	\node at (-.6, -.4) {\scriptsize{$\beta$}};
\end{tikzpicture}
 \longmapsto
\begin{tikzpicture}[baseline=-.1cm]
	\draw (0,0)--(.8,0);
	\draw[thick, \alphacolor] (-.8,.2)--(0,.2);
	\draw[thick, \betacolor] (-.4,-.2) arc(90:270:.2cm)--(.8,-.6);
	\draw (0,.4)--(0,.8);
	\draw[thick, unshaded] (-.4, -.4) -- (-.4, .4) -- (.4, .4) -- (.4, -.4) -- (-.4, -.4);
	\node at (0, 0) {$x$};
%	\node at (-.6, .4) {\scriptsize{$\alpha$}};
	\node at (.6, .2) {\scriptsize{$\gamma$}};
%	\node at (.6, -.4) {\scriptsize{$\overline{\beta}$}};
\end{tikzpicture}
$$
is a unitary operator intertwining the left and right $M$-actions.
\end{proof}

\begin{defn}
By Lemma \ref{lem:rotate}, we now define the bimodule $H_\alpha:=\cH_{\alpha,\emptyset}\cong \cH_{\emptyset,\overline{\alpha}}$. We draw elements of $H_\alpha$ with strings emanating from the bottom instead of the sides
$$
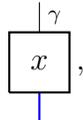
\ColorNBox{\gamma}{x}{\alphacolor}\,,
$$
and the left and right $M$-actions are given by the obvious diagrams. For $\xi\in p_{\alpha}\wedge \cM_{\alpha, \emptyset}\wedge p_{\emptyset}\cong p_{\emptyset} \wedge \cM_{\emptyset,\overline{\alpha}}\wedge p_{\overline{\alpha}}$ and $x,y\in M$,
$$
x\xi y =
\Mbox{}{x}
\,
\ColorNBox{}{\xi}{\alphacolor}
\,
\Mbox{}{y}
\,.
$$
The inner product on $H_\alpha$ is given by
\begin{equation}\label{eqn:InnerProduct}
\langle x, y \rangle_{H_\alpha} =
\begin{tikzpicture} [baseline = .3cm]
	\draw (0,.8)--(0, -.4);
	\draw (1.2,-.4) -- (1.2,.8);
	\draw[thick, \alphacolor] (0, -.4) .. controls ++(270:.4cm) and ++(270:.4cm) .. (1.2,-.4);
	\draw[thick, unshaded] (-.4,.8) -- (-.4, 1.6) -- (1.6,1.6) -- (1.6,.8) -- (-.4,.8);
	\draw[thick, unshaded] (-.4, -.4) -- (-.4, .4) -- (.4, .4) -- (.4, -.4) -- (-.4, -.4);
	\draw[thick, unshaded] (.8, -.4) -- (.8, .4) -- (1.6, .4) -- (1.6, -.4) -- (.8, -.4);
	\node at (0, 0) {$\overline{y^*}$};
	\node at (1.2, 0) {$x$};
	\node at (.6, 1.2) {$\sum CTL$};
	\node at (1.4,-.6) {\scriptsize{$\alpha$}};
	\node at (-.2,-.6) {\scriptsize{$\overline{\alpha}$}};
\end{tikzpicture}
\end{equation}
where $\overline{y^*}$ is the rotation of $y^*$ by 180 degrees.
\end{defn}

Our goal will now be to show that $H_\alpha \otimes_M H_\beta\cong H_{\alpha\beta}$. In diagrams:
$$
\ColorNBox{}{}{\alphacolor}
\,
\ColorNBox{}{}{\betacolor}
\leftrightarrow
\begin{tikzpicture}[baseline=-.1cm]
	\draw (0,0)--(0,.8);
	\draw[thick,\alphacolor] (-.2,-.8)--(-.2,0);
	\draw[thick,\betacolor] (.2,-.8)--(.2,0);
	\draw[thick, unshaded] (-.4, -.4) -- (-.4, .4) -- (.4, .4) -- (.4, -.4) -- (-.4, -.4);
	\node at (0, 0) {};
%	\node at (-.4,-.6) {\scriptsize{$\alpha$}};
%	\node at (.4,-.6) {\scriptsize{$\beta$}};
\end{tikzpicture}
$$
We will need the following lemma, whose purpose is to ``close up" the space between the boxes.

\begin{lem} \label{lem:cup}
The $\alpha-$cup element
$\,
\begin{tikzpicture} [baseline=-.1cm]
	\filldraw[unshaded,thick] (-.4,.4)--(.4,.4)--(.4,-.4)--(-.4,-.4)--(-.4,.4);
	\draw[thick, \alphacolor] (-.3, .4) arc(180:360: .3cm);
%	\node at (0, -.1) {\scriptsize{$\alpha$}};
\end{tikzpicture}
$
is positive and invertible in $M$.
\end{lem}
\begin{proof}
We induct on the size of $\alpha$.  If $|\alpha| = 1$ then it follows from \cite{MR2732052} that the operator is a free-poisson element whose spectrum is supported away from 0 hence the lemma holds for this case.  Suppose that the lemma holds for some color pattern $\beta$ and suppose $\alpha = \beta c$ for a fixed color $c$, denoted by a \textcolor{\cupcolor}{green} string.  It follows that the element
$$
x =
\begin{tikzpicture} [baseline=-.1cm]
	\filldraw[unshaded,thick] (-.4,.4)--(.4,.4)--(.4,-.4)--(-.4,-.4)--(-.4,.4);
	\draw[thick, \betacolor] (-.2, .4) arc (180:360: .2cm);
	\draw[thick, \cupcolor] (-.4,-.15)--(.4,-.15);
%	\node at (0, .05) {\scriptsize{$\beta$}};
%	\node at (0, -.25) {\scriptsize{$c$}};
\end{tikzpicture}
$$
is positive invertible in $M_{c}$.  Let $z \in \cM_{\emptyset,c}$ be the element
$$
z =
\begin{tikzpicture} [baseline=-.1cm]
	\filldraw[unshaded,thick] (-.4,.4)--(.4,.4)--(.4,-.4)--(-.4,-.4)--(-.4,.4);
	\draw[thick, \cupcolor] (-.4, 0) arc(-90:0: .4cm);
%	\node at (0, 0) {\scriptsize{$c$}};
\end{tikzpicture}
$$
Then the $\alpha-$cup element has the form $z^{*}\wedge x\wedge z \in \cM_{\emptyset, c}$ so it immediately follows that it is positive.  Note that $z^{*}\wedge z$ is invertible so there is a positive constant $k_{z}$ so that $\langle z^{*}\wedge z \wedge \xi, \xi \rangle \geq k_{z}\|\xi\|_{2}^{2}$ for all $\xi \in L^{2}(M)$.  Letting $r$ be the positive square root of $x$, there is a strictly positive constant $k_{r}$ so that for all $\eta \in L^{2}(M_{c})$, $\langle r\wedge\eta, r\wedge\eta \rangle \geq k_{r} \|\eta\|_{2}^{2}$.  Therefore, for all $\xi \in L^{2}(M)$,
\begin{align*}
\langle z^{*}\wedge x\wedge z \wedge \xi, \xi \rangle
&= \langle r\wedge z \wedge \xi, r\wedge z \wedge \xi \rangle
\geq k_{r} \langle z\wedge \xi, z\wedge \xi \rangle \\
 &= k_{r} \langle z^{*}\wedge z \wedge\xi, \xi \rangle \geq k_{r}k_{z}\|\xi\|_{2}^{2}
\end{align*}
implying invertibility.
\end{proof}

\begin{defn}
Recall from \cite{MR561983,MR1424954,bimodules} that given a bimodule $\sb{M}K_M$, a vector $\xi\in K$ is called \underline{right $M$-bounded} if the map $L(\xi)^0\colon M\to K$ given by $m\mapsto \xi m$ extends to a bounded linear operator $L(\xi)\colon L^2(M)\to K$. There is a similar definition of a left $M$-bounded vector, and if $K$ is bifinite, then the sets of left and right $M$-bounded vectors agree. For such a bimodule $K$, we denote the set of left/right $M$-bounded vectors by $D(K)$.
\end{defn}

\begin{ex}
For all $\alpha,\beta\in\Lambda$, we have
$D(H_{\alpha\beta})=D(\cH_{\alpha,\overline{\beta}})=p_\alpha \wedge\cM_{\alpha,\overline{\beta}}\wedge p_{\overline{\beta}}$.
\end{ex}

\begin{lem} \label{lem:TensorProduct}
As $M - M$ bimodules, the Connes' fusion $H_\alpha\otimes_M H_\beta\cong H_{\alpha\beta}$.
\end{lem}
\begin{proof}
By  \cite{MR1424954}, $H_\alpha\otimes_M H_\beta$ is the completion of the algebraic tensor product
$
D(H_\alpha) \odot D(H_\beta)
$
(first modding out by vectors of length zero) under the semi-definite inner product given by
$$
\langle \xi_{1} \otimes \eta_{1}, \xi_{2} \otimes \eta_{2} \rangle_{H_\alpha \otimes_M H_\beta}
= \langle  \langle \xi_2|\xi_1 \rangle_M \cdot \eta_{1}, \eta_{2} \rangle_{H_\beta}
$$
where $\langle \xi_2|\xi_1 \rangle_M$ is the unique element in $M$ satisfying
$$
\langle \xi_{1}x, \xi_{2} \rangle_{H_\alpha} = \tr(x\langle \xi_2|\xi_1 \rangle_M) \text{ for all }x\in M.
$$
We define a map $u$ by the linear extension of
\begin{align*}
D(H_\alpha) \odot D(H_\beta)
&\longrightarrow \cH_{\alpha,\overline{\beta}}\cong H_{\alpha\beta}
\text{ by}\\
\ColorNBox{}{\xi}{\alphacolor}
\otimes
\ColorNBox{}{\eta}{\betacolor}
&\longmapsto
\ColorNBox{}{\xi}{\alphacolor}
\,
\ColorNBox{}{\eta}{\betacolor}
\,.
\end{align*}
Obviously the map $u$ intertwines the left and right $M$-actions and is $M$-middle linear. We will show it is isometric with dense range, and thus $u$ has a unique extension to a $M-M$ bilinear unitary giving the desired isomorphism.

First, it is not hard to check directly that the $M$-valued inner product of the $M$-bounded vectors $\xi_{1}, \xi_{2} \in D(H_\alpha)$ is given by
$$
\langle \xi_2|\xi_1\rangle_M
=
\begin{tikzpicture}[baseline=-.1cm]
	\draw (0,.6)--(0, -.4);
	\draw[thick, \alphacolor] (0, -.4) .. controls ++(270:.4cm) and ++(270:.4cm) .. (1.2,-.4);
	\draw (1.2,-.4) -- (1.2,.6);
	\draw[thick, unshaded] (1.6, -.4) -- (1.6, .4) -- (.8, .4) -- (.8, -.4) -- (1.6, -.4);
	\draw[thick, unshaded] (-.4, -.4) -- (-.4, .4) -- (.4, .4) -- (.4, -.4) -- (-.4, -.4);
	\node at (1.2, 0) {$\xi_1$};
	\node at (-.2,-.6) {\scriptsize{$\overline{\alpha}$}};
	\node at (1.4,-.6) {\scriptsize{$\alpha$}};
	\node at (0, 0) {$\overline{\xi_2^*}$};
\end{tikzpicture}
\, .
$$
Hence if
$$
\zeta = \sum_{i=1}^{n} \xi_{i} \otimes \eta_{i} \in
D(H_\alpha) \odot D(H_\beta),
$$
we immediately have
\begin{equation}\label{eqn:InnerProductForTensorProduct}
\|\zeta\|_{2}^{2}
= \sum_{i,j=1}^n \langle  \langle \xi_j|\xi_i \rangle_M \cdot \eta_{i}, \eta_{j} \rangle_{H_\beta}
=
\begin{tikzpicture} [baseline = .3cm]
	\draw (0,.8)--(0, -.4);
	\draw (1.2,-.4) -- (1.2,.8);
	\draw[thick, \alphacolor] (0, -.4) .. controls ++(270:.4cm) and ++(270:.4cm) .. (1.2,-.4);
	\draw (-1.2,.8)-- (-1.2,-.4);
	\draw (2.4,-.4)--(2.4,.8);
	\draw[thick, \betacolor] (-1.2,-.4) .. controls ++(270:1cm) and ++(270:1cm) .. (2.4,-.4);
	\draw[thick, unshaded] (-1.6,.8) -- (-1.6, 1.6) -- (2.8,1.6) -- (2.8,.8) -- (-1.6,.8);
	\draw[thick, unshaded] (-.4, -.4) -- (-.4, .4) -- (.4, .4) -- (.4, -.4) -- (-.4, -.4);
	\draw[thick, unshaded] (-.8, -.4) -- (-.8, .4) -- (-1.6, .4) -- (-1.6, -.4) -- (-.8, -.4);
	\draw[thick, unshaded] (.8, -.4) -- (.8, .4) -- (1.6, .4) -- (1.6, -.4) -- (.8, -.4);
	\draw[thick, unshaded] (2, -.4) -- (2, .4) -- (2.8, .4) -- (2.8, -.4) -- (2, -.4);
	\node at (-1.2, 0) {$\overline{\eta_j^*}$};
	\node at (0, 0) {$\overline{\xi_j^*}$};
	\node at (1.2, 0) {$\xi_i$};
	\node at (2.4, 0) {$\eta_i$};
	\node at (.6, 1.2) {$\sum CTL$};
%	\node at (.6,-.5) {\scriptsize{$\alpha$}};
%	\node at (.6,-1) {\scriptsize{$\beta$}};
\end{tikzpicture} \, .
\end{equation}
Since
$$
u(\zeta) =
\sum_{i=1}^n
\ColorNBox{}{\xi_i}{\alphacolor}
\,
\ColorNBox{}{\eta_i}{\betacolor}
\,,
$$
we see that $\|u(\zeta)\|_{2}^{2}$ is the same diagram as in \eqref{eqn:InnerProductForTensorProduct}.  Hence $u$ can be extended to be an isometry on $H_{\alpha}\otimes_M H_\beta$.

For the rest of this proof we will only use the $\wedge$ multiplication, so to clean up notation we will omit the $\wedge$. We must show the image of $u$ is dense in $H_{\alpha\beta}$. We may assume $\delta_\alpha\leq \delta_\beta$, and the proof is similar if $\delta_\beta <\delta_\alpha$.
Set
$$
w =
\begin{tikzpicture} [baseline=-.1cm]
	\filldraw[unshaded,thick] (-.4,.4)--(.4,.4)--(.4,-.4)--(-.4,-.4)--(-.4,.4);
	\draw[thick,\alphacolor] (-.4, 0) arc(-90:0: .4cm);
%	\node at (0, 0) {\scriptsize{$\alpha$}};
\end{tikzpicture}\,.
$$
Then $w^{*}w$ is positive and invertible in $M$, so if $v$ is the polar part of $w$, then
$v \in D(H_\alpha)$
with right support $p_{\emptyset}$ and left support $e$ under $p_{\alpha}$.
Choose partial isometries $x_{1}, ..., x_{n}$ in $M_{\alpha}$ such that $x_{1}, ..., x_{n-1}$ have right support $e$ and orthogonal left supports.
Choose the partial isometry $x_{n}$ so that its left support is $1 - \sum_{i=1}^{n-1}x_{i}x_{i}^{*}$ and its right support is under $e$.

Similarly, there is a partial isometry $\tilde{v} \in D(H_\beta)$ whose right support is $f \leq p_{\overline{\beta}}$ and whose left support is $p_{\emptyset}$.
Choose partial isometries $y_{1}, ..., y_{n} \in M_{\overline{\beta}}$ such that $y_{1}, ..., y_{n-1}$ have left support $f$ and orthogonal right supports.
Choose the partial isometry $y_{n}$ so its right support is orthogonal to $\sum_{i=1}^{n-1}y_{i}^{*}y_{i}$ and its left support is the right support of $x_{n}v\tilde{v}$ (we can do this since $\delta_{\alpha} \leq \delta_{\overline{\beta}}$).

It now follows that $z = \sum_{i=1}^{n}x_{i}v\tilde{v}y_{i}$ is a partial isometry in $D(H_{\alpha\beta})$ with full left support $p_{\alpha}$.  Note that every element, $r$ in $D(H_{\alpha\beta})$ is of the form $zx$ for $x \in M_{\overline{\beta}}$ since we simply choose $x = z^{*}r$.  Since $zx = \sum_{i=1}^{n}(x_{i}v)\cdot(\tilde{v}y_{i}x)$ is in the image of $u$, we are finished.
\end{proof}

\begin{lem} \label{lem:bifinite}
As an $M-M$ bimodule, $H_\alpha$ is bifinite. Thus by Lemma \ref{lem:rotate}, so is $\cH_{\alpha,\beta}$.
\end{lem}

\begin{proof}
By the proof of Lemma \ref{lem:TensorProduct}, there is a partial isometry $v \in p_{\alpha}\wedge \cM_{\alpha, \emptyset}\wedge p_{\emptyset}$ with right support $p_{\emptyset}$.  The argument at the end of Lemma \ref{lem:TensorProduct} shows that every element of $p_{\alpha}\wedge\cM_{\alpha, \emptyset}\wedge p_{\emptyset}$ is of the form $x\wedge v$ for $x \in M_{\alpha}$.  Therefore $v$ is cyclic for the left action of $M_{\alpha}$ on $\cH_{\alpha,\emptyset}\cong H_\alpha$, and since $[M_{\alpha}:M] < \infty$, $\dim_{M-}(H_{\alpha})<\infty$. Similarly $\dim_{-M}(H_{\alpha})<\infty$.
\end{proof}

%%%%%%%%%%%%%%%%%%%%%%%%%%%%%%%%%%%%%%%%%%%%%%%%%%%%%%%%%%%%%%
\subsection{The categories $\Bim(P_\bullet)$ and $\CF(P_\bullet)$} \label{sec:CategoriesOfBimodules}

We now define two rigid $C^*$-tensor categories $\Bim(P_\bullet)$ and $\CF(P_\bullet)$ whose objects are bifinite $M-M$ bimodules. We show given a factor planar algebra $P_\bullet$, we have equivalences $\Pro(P_\bullet)\cong \Bim(P_\bullet)\cong \CF(P_\bullet)$.

%%%%%%%%%%%%%%%%%%%%%%%%%%%%%%%%%%%%%
\begin{defn}
If $p$ is a projection in $P_{\alpha\to \alpha}$, we define $H_p = \cH_{\emptyset, \overline{\alpha}}\wedge p$, which is a bifinite $M-M$ bimodule. Note that elements of $H_p$ are obtained from elements in $H_\alpha$ by putting a $p$ on the bottom. Thus the linear span of elements of the form
$$
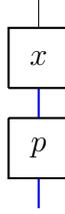
\ColorMultiply{x}{p}{\alphacolor}
$$
where $x\in p_{\emptyset}\wedge \cM_{\emptyset, \overline{\alpha}}\wedge p$ forms a dense subset of $H_p$. Recall that elements of $M$ act on the left and right as in the non-orthogonal picture. These actions clearly do not affect the $p$ on the bottom.
\end{defn}

\begin{defn}\label{defn:Bim}
Let $\Bim(P_\bullet)$ (abbreviated $\Bim$) be the strict rigid $C^*$-tensor category defined as follows.
\begin{enumerate}
\item[\text{\underline{Objects:}}] The objects of $\Bim$ are finite direct sums of the bimodules $H_p$ for the projections $p\in P_{\alpha\to \alpha}$ for $\alpha\in \Lambda$. Note that the unit object $1_\Bim=H_\emptyset\cong L^2(M)$.
\item[\text{\underline{Tensor:}}]
For $p\in P_{\alpha\to \alpha}$ and $q\in P_{\beta\to \beta}$, we define $H_p\otimes_{\Bim} H_q=H_{p\otimes q}$ where $p\otimes q$ is the tensor product in $P_\bullet$. The tensor product is extended to direct sums linearly.

Note that since the tensor product of projections in $P_\bullet$ is strict, so is the tensor product in $\Bim$.
\item[\text{\underline{Morphisms:}}]
For $p\in P_{\alpha\to \alpha}$ and $q\in P_{\beta\to \beta}$, we define $\Bim(H_p\to H_q)=qP_{\alpha\to \beta}p$ and composition is the usual composition in $P_\bullet$. Morphisms between direct sums are matrices of such maps.
 \item[\text{\underline{Tensoring:}}]
For $x_i\in \Bim(H_{p_i}\to H_{q_i})=q_iP_{\alpha_i\to \beta_i}p_i$ for $i=1,2$, we define $x_1\otimes_\Bim x_2$ as the tensor product of morphisms in $P_\bullet$. Similarly for matrices of such maps.
\item[\text{\underline{Duality:}}]
The dual of $H_p$ is $H_{\overline{p}}$. The evaluation map $\ev_{H_p}\colon H_{\overline{p}}\otimes_\Bim H_p \to H_\emptyset=1_\Bim$ is given by the $\alpha$-cup with projections $\overline{p},p$ on top:
$$
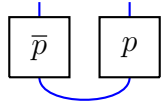
\MInnerProduct{\overline{p}}{\alphacolor}{p}
$$
Of course, the coevalutation map $1_\Bim\to H_p\otimes H_{\overline{p}}$ is given by the adjoint of $\ev_{H_{\overline{p}}}$, which is the $\alpha$-cap with projections underneath:
$$
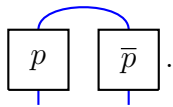
\MInnerProductOp{p}{\alphacolor}{\overline{p}}\,.
$$
One easily checks that the necessary relations hold.

The dual map is extended to direct sums linearly.
\item[\text{\underline{Adjoint:}}]
The adjoint map $*$ is the identity on all objects, and the adjoint of a morphism $x\in \Bim(H_p\to H_q)=qP_{\alpha\to \beta}p$ is the adjoint in the planar algebra $x^*\in pP_{\beta\to \alpha}q=\Bim(H_q\to H_p)$. The adjoint of a matrix of maps is the $*$-transpose of the matrix.
\end{enumerate}
\end{defn}

\begin{thm}
The map $\Pro\to \Bim$ by $p\mapsto H_p$ and the identity on morphisms is an equivalence of categories.
\end{thm}
\begin{proof}
Note that $\Pro(p\to q) = qP_{\alpha\to \beta}p=\Bim(H_p\to H_q)$.
One now checks that the described map is an additive, monoidal, dual-preserving, $*$-preserving, fully faithful, essentially surjective functor.
\end{proof}

\begin{defn}\label{defn:CF}
Let $\CF(P_\bullet)$ (abbreviated $\CF$, which stands for \emph{Connes' fusion}) be the rigid $C^*$-tensor category defined as follows.
\begin{enumerate}
\item[\text{\underline{Objects:}}] The objects of $\CF$ are finite direct sums of the bimodules $H_p$ as in $\Bim$.
\item[\text{\underline{Tensor:}}]
For bimodules $K,L\in \CF$, we define $K\otimes_{\CF} L=K\otimes_M L$, the Connes' fusion of $II_1$-factor bimodules.

The associator $a_\CF$ is defined by restricting to $M$-bounded vectors as in \cite{MR1424954}.
\item[\text{\underline{Morphisms:}}]
$\CF(K\to L)$ is the set of $M-M$ bilinear maps $K\to L$. Composition is the usual composition of linear maps.
\item[\text{\underline{Tensoring:}}]
For $M-M$ bilinear maps $\varphi_i\colon K_{i}\to L_{i}$ for $i=1,2$, we define $\varphi_1\otimes_{M} \varphi_2$ by the Connes' fusion of intertwiners. If we have $M$-bounded vectors $\xi_i\in K_{i}$ for $i=1,2$, then the map $\xi_1\otimes \xi_2\mapsto \varphi_1(\xi_1)\otimes \varphi_2(\xi_2)$ is clearly $M$-middle linear and bounded, so it extends to a unique $M-M$ bilinear map.
\item[\text{\underline{Duality:}}]
The dual of $K$ is the contragredient bimodule $\overline{K}=\set{\overline{\xi}}{\xi\in K}$ where  $\lambda \overline{\xi}+\overline{\eta}=\overline{\overline{\lambda}\xi+\eta}$ for all $\lambda\in\C$ and $\eta,\xi\in K$, and the action is given by $a\overline{\xi}b = \overline{b^*\xi a^*}$. The evaluation map
$$
\ev_{K}\colon \overline{K}\otimes_M K\longrightarrow H_\emptyset=L^2(M)=1_\CF
$$
is the unique extension of the map $\overline{\xi}\otimes \eta\mapsto \langle \xi|\eta\rangle_M$ where $\xi,\eta$ are $M$-bounded vectors in $K$.
The coevaluation map $\coev_{K}$ is the unique map in $\CF(1\to K\otimes_M \overline{K})$ corresponding to $\id_{K}\in\CF(K\to K)$ under the natural isomorphism given by Frobenius reciprocity. For an explicit formula, just pick an orthonormal left $M$-basis $\{\zeta\}\subset K$ (e.g., see \cite{MR561983,1110.3504}), and we have that
$$
\coev_{K}(1_M)=\sum_\zeta \zeta\otimes \overline{\zeta}
$$
is $M$-central and independent of the choice of $\{\zeta\}$. The zig-zag relation is now given by
$$
\sum_\zeta \zeta \langle \zeta|\xi\rangle_M = \xi
$$
for all $M$-bounded $\xi\in K$.

\item[\text{\underline{Adjoint:}}]
The adjoint map $*$ is the identity on all objects and on a morphism $x\in \CF(K\to L)$ is the adjoint linear operator $x^*\in \CF(L\to K)$.
\end{enumerate}
\end{defn}

Note $\CF$ is a rigid $C^*$-tensor category by well known properties of Connes' fusion (e.g., see \cite{MR1424954,1110.3504}).
It is now our task to prove the following theorem:

\begin{thm}\label{thm:BimCF}
Define a map $\Phi\colon \Bim\to \CF$ as follows. First, $\Phi$ is the identity on objects. Second, for a morphism $x\in qP_{\alpha\to \beta}p$, we get an $M-M$ bimodule map $\Phi_x\colon H_p\to H_q$ by
\begin{equation}\label{eqn:BimoduleMaps}
\ColorMultiply{\xi}{p}{\alphacolor}
\longmapsto
\ColorMultiplyThree{\xi}{p}{x}{\alphacolor}{\betacolor}\in H_q,
\end{equation}
Finally, $\Phi$ is applied entry-wise to matrices over such morphisms.

The map $\Phi$ is an equivalence of categories $\Bim\simeq \CF$.
\end{thm}

In the lemmas below, unless otherwise stated, $p,q$ are projections in $P_{\alpha \to \alpha},P_{\beta \to \beta}$ respectively.

\begin{rem}\label{rem:Functor}
Note that composition of the bimodule maps given by Equation \eqref{eqn:BimoduleMaps} corresponds to the usual composition in $P_\bullet$, i.e., if we have $x\in qP_{\alpha\to \beta}p$ and $y\in rP_{\beta\to \gamma}q$, then $\Phi_y\circ \Phi_x =\Phi_{yx}\colon H_p\to H_q$, where
$$
\nbox{\alpha}{yx}{\gamma}=\PAMultiply{\alpha}{x}{\beta}{y}{\gamma}\,.
$$
It is obvious that $\Phi_{p}=\id_{H_p}$, so $\Phi$ is a functor.
\end{rem}

\begin{lem}\label{lem:TensorProduct2}
For $p,q$ projections in $P_\bullet$, the maps
$$
\phi_{p,q}\colon H_p\otimes_M H_q\to H_p\otimes_\Bim H_q=H_{p\otimes q}
$$
(where $p\otimes q$ is the tensor product in $P_\bullet$) given by the unique extension of
$$
\ColorMultiply{\xi}{p}{\alphacolor}
\otimes
\ColorMultiply{\eta}{q}{\betacolor}
\longmapsto
\ColorMultiply{\xi}{p}{\alphacolor}
\,
\ColorMultiply{\eta}{q}{\betacolor}
$$
for $\xi\in D(H_p)$ and $\eta\in D(H_q)$, are $M-M$ bilinear isomorphisms which satisfy associativity, i.e., the following diagram commutes:
$$
\xymatrix{
(H_p\otimes_M H_q)\otimes_M  H_r \ar[rr]^{a_\CF}\ar[d]_{\phi_{p,q}\otimes_M\id_{H_r}} && H_p\otimes_M (H_q\otimes_M H_r)\ar[d]^{\id_{H_p}\otimes_M \phi_{q,r}}\\
(H_{p}\otimes_\Bim H_q) \otimes_M H_r\ar[d]_{\phi_{p\otimes q,r}} && H_p\otimes_M ( H_q \otimes_\Bim H_r)\ar[d]^{\phi_{p,q\otimes r}}\\
(H_p\otimes_\Bim H_q)\otimes_\Bim H_r \ar[rr]^{=} && H_p\otimes_\Bim (H_q\otimes_\Bim H_r).
}
$$
\end{lem}
\begin{proof}
As in the proof of Lemma \ref{lem:TensorProduct}, the map for $\xi\in D(H_p),\eta\in D(H_q)$ is an isometry intertwining the left and right $M$ actions.  Since the linear span of elements of the form $x\wedge y$ with $x\in D(\cH_{\alpha, \emptyset})$ and $y\in D(\cH_{\emptyset, \overline{\beta}})$ is equal to $D(\cH_{\alpha, \overline{\beta}}) \cong D(H_{\alpha\beta})$ by the proof of Lemma \ref{lem:TensorProduct}, the above map is also surjective.

Associativity follows from looking at $M$-bounded vectors (see \cite{MR1424954}).
\end{proof}

\begin{lem}\label{lem:FullyFaithful}
As complex vector spaces,
$$
\CF(H_p\to H_q)\cong q(P_{\alpha \to \beta})p=\Bim(H_p\to H_q).
$$
Moreover, the composition of maps $\varphi\in \CF(H_p\to H_q)$ and $\psi\in \CF(H_q\to H_r)$ corresponds to the composition in $\Bim$.
\end{lem}
\begin{proof}
Recall that an element $x\in q(P_{\alpha \to \beta})p=\Bim(H_p\to H_q)$ gives a map $\Phi_x\in \CF(H_p\to H_q)$ as in Equation \eqref{eqn:BimoduleMaps}, and composition of morphisms is exactly multiplication in $P_\bullet$ by Remark \ref{rem:Functor}. Note further that $\Bim(H_p\to H_q)$ and $\CF(H_p\to H_q)$ are finite dimensional, so it remains to show they have the same dimension.

By Frobenius reciprocity, we have a natural isomorphism
$$
\Hom_{M-M}(H_p\to H_q)\cong \Hom_{M-M}(1\to \overline{H_p}\otimes H_q),
$$
and the latter space is naturally identified with the $M-M$ central vectors in $\overline{H_p}\otimes H_q$. Note that
$$
\overline{H_{p}} \otimes_M H_{q}
\cong
H_{\overline{p}} \otimes_{M} H_{q}
\cong
H_{\overline{p}\otimes q}
\cong
p\wedge\cH_{\overline{\alpha},\overline{\beta}}\wedge q.
$$
From Corollary \ref{cor:central}, the set of central vectors in $p\wedge\cH_{\overline{\alpha},\overline{\beta}}\wedge q$ is $q(P_{\alpha \to \beta})p$
$$
p\wedge_{\overline{\alpha}}\cH_{\overline{\alpha},\overline{\beta}}\wedge_{\overline{\beta}} q
\ni
\begin{tikzpicture} [baseline = -.1cm]
	\draw[thick, \betacolor] (0,0)--(2,0);
	\draw[thick, \alphacolor] (-2,0)--(0,0);
	\draw[thick, unshaded] (-.4, -.4) -- (-.4, .4) -- (.4, .4) -- (.4, -.4) -- (-.4, -.4);
	\draw[thick, unshaded] (.8, -.4) -- (.8, .4) -- (1.6, .4) -- (1.6, -.4) -- (.8, -.4);
	\draw[thick, unshaded] (-.8, -.4) -- (-.8, .4) -- (-1.6, .4) -- (-1.6, -.4) -- (-.8, -.4);
	\node at (0, 0) {$x$};
	\node at (1.2, 0) {$q$};
	\node at (-1.2, 0) {$p$};
\end{tikzpicture}
\longleftrightarrow
\begin{tikzpicture}[baseline=-.1cm]
	\draw[thick, \betacolor] (0,0)--(0,-2);
	\draw[thick, \alphacolor] (0,0)--(0,2);
	\filldraw[unshaded,thick] (-.4,.4)--(.4,.4)--(.4,-.4)--(-.4,-.4)--(-.4,.4);
	\draw[thick, unshaded] (-.4, .8) -- (-.4, 1.6) -- (.4, 1.6) -- (.4,.8) -- (-.4, .8);
	\draw[thick, unshaded] (-.4, -.8) -- (-.4, -1.6) -- (.4, -1.6) -- (.4,-.8) -- (-.4, -.8);
%	\node at (.2,1.8) {{\scriptsize{$#1$}}};
	\node at (0,1.2) {$p$};
%	\node at (.2,.6) {{\scriptsize{$#3$}}};
	\node at (0,0) {$x$};	
%	\node at (.2,-.6) {{\scriptsize{$#5$}}};
	\node at (0,-1.2) {$q$};
%	\node at (.2,-1.8) {{\scriptsize{$#7$}}};
\end{tikzpicture}
\in q(P_{\alpha \to \beta})p,
$$
proving $\dim( \CF(H_p\to H_q))=\dim(\Bim(H_p\to H_q))$.
\end{proof}

\begin{lem}\label{lem:StarPreserving}
If $x\in \Bim(H_p\to H_q)=qP_{\alpha\to \beta}p$ and $\Phi_x\in\CF( H_p\to H_q)$ is as in Equation \eqref{eqn:BimoduleMaps}, then $\Phi_x^* = \Phi_{x^*}$.
\end{lem}
\begin{proof}
By Equation \eqref{eqn:InnerProduct}, if $\xi\in D(H_p)$ and $\eta\in D(H_q)$, then
$$
\langle \Phi_x \xi,\eta\rangle_{H_q}
=
\begin{tikzpicture}[baseline=1.7cm]
	\draw (0,2.4)--(0,3.2);
	\draw (1.2,2.4)--(1.2,3.2);
	\draw [thick, red] (0,2.4)-- (0, -.4) .. controls ++(270:.4cm) and ++(270:.4cm) .. (1.2,-.4);
	\draw [thick, blue] (1.2,0) -- (1.2,2.8);
	\draw[thick, unshaded] (-.4,3.2) -- (-.4, 4) -- (1.6, 4) -- (1.6, 3.2) -- (-.4, 3.2);
	\draw[thick, unshaded] (1.6, -.4) -- (1.6, .4) -- (.8, .4) -- (.8, -.4) -- (1.6, -.4);
	\draw[thick, unshaded] (1.6, .8) -- (1.6, 1.6) -- (.8, 1.6) -- (.8, .8) -- (1.6, .8);
	\draw[thick, unshaded] (1.6, 2) -- (1.6, 2.8) -- (.8, 2.8) -- (.8, 2) -- (1.6, 2);
%	\draw[thick, unshaded] (-.4, -.4) -- (-.4, .4) -- (.4, .4) -- (.4, -.4) -- (-.4, -.4);
	\draw[thick, unshaded] (-.4, .8) -- (-.4, 1.6) -- (.4, 1.6) -- (.4, .8) -- (-.4, .8);
	\draw[thick, unshaded] (-.4, 2) -- (-.4, 2.8) -- (.4, 2.8) -- (.4, 2) -- (-.4, 2);
	\node at (.6, 3.6) {$\sum CTL$};
	\node at (0, 2.4) {$\overline{\eta^*}$};
	\node at (1.2, 2.4) {$\xi$};
	\node at (0, 1.2) {$\overline{q}$};
	\node at (1.2, 1.2) {$p$};
	\node at (1.2,0) {$x$};
%	\node at (.2,.6) {{\scriptsize{$\overline{\beta}$}}};
%	\node at (1.4,.6) {{\scriptsize{$\alpha$}}};
%	\node at (.6,-.5) {\scriptsize{$\beta$}};
\end{tikzpicture}
=
\begin{tikzpicture}[baseline=1.7cm]
	\draw (0,2.4)--(0,3.2);
	\draw (1.2,2.4)--(1.2,3.2);
	\draw [thick, red] (0,2.4)--(0,0);
	\draw [thick, blue] (0, -.4) .. controls ++(270:.4cm) and ++(270:.4cm) .. (1.2,-.4) -- (1.2,2.4);;
	\draw[thick, unshaded] (-.4,3.2) -- (-.4, 4) -- (1.6, 4) -- (1.6, 3.2) -- (-.4, 3.2);
%	\draw[thick, unshaded] (1.6, -.4) -- (1.6, .4) -- (.8, .4) -- (.8, -.4) -- (1.6, -.4);
	\draw[thick, unshaded] (1.6, .8) -- (1.6, 1.6) -- (.8, 1.6) -- (.8, .8) -- (1.6, .8);
	\draw[thick, unshaded] (1.6, 2) -- (1.6, 2.8) -- (.8, 2.8) -- (.8, 2) -- (1.6, 2);
	\draw[thick, unshaded] (-.4, -.4) -- (-.4, .4) -- (.4, .4) -- (.4, -.4) -- (-.4, -.4);
	\draw[thick, unshaded] (-.4, .8) -- (-.4, 1.6) -- (.4, 1.6) -- (.4, .8) -- (-.4, .8);
	\draw[thick, unshaded] (-.4, 2) -- (-.4, 2.8) -- (.4, 2.8) -- (.4, 2) -- (-.4, 2);
	\node at (.6, 3.6) {$\sum CTL$};
	\node at (0, 2.4) {$\overline{\eta^*}$};
	\node at (1.2, 2.4) {$\xi$};
	\node at (0, 1.2) {$\overline{q}$};
	\node at (1.2, 1.2) {$p$};
	\node at (0,0) {$\overline{x}$};
%	\node at (.2,.6) {{\scriptsize{$\overline{\beta}$}}};
%	\node at (1.4,.6) {{\scriptsize{$\alpha$}}};
%	\node at (.6,-.5) {\scriptsize{$\beta$}};
\end{tikzpicture}
=
\langle \xi, \Phi_{x^*} \eta\rangle_{H_p}
$$
where $\overline{x}$ is the 180 degree rotation of $x$.
\end{proof}

For the next lemma, recall that the conjugate Hilbert space of $K$, is the set of formal symbols $\set{\overline{\xi}}{\xi\in H_p}$ such that $\lambda \overline{\xi}+\overline{\eta}=\overline{\overline{\lambda}\xi+\eta}$ for all $\lambda\in\C$ and $\eta,\xi\in H_p$, together with left and right $M$-actions given by $x \overline{\xi} y = \overline{y^* \xi x^*}$ for $x,y\in M$.

\begin{lem} \label{lem:DualPreserving}
For $p\in P_{\alpha\to\alpha}$, define $\psi_p\colon H_{\overline{p}}\to \overline{H_p}$ (where $\overline{p}$ is the dual projection of $p$ in $P_\bullet$) by the unique extension of the map
$$
\ColorMultiply{x}{\overline{p}}{\alphacolor}
\longmapsto
\overline{
\ColorMultiply{x^*}{p}{\alphacolor}
}
$$
for $\xi\in D(H_{\overline{p}})$ (the {\alphacolor} strand on the left is labelled $\overline{\alpha}$). Then $\psi_p$ is an $M-M$ bilinear isomorphism such that the following diagrams commute:
$$
\xymatrix{
H_{\overline{p}}\otimes_M H_{\overline{q}}
\ar[rr]^{\psi_p\otimes_M \psi_q}
\ar[d]_{\phi_{\overline{p},\overline{q}}}
&&
\overline{H_p}\otimes_M \overline{H_q}
\ar[rr]^{\cong}
&&
\overline{H_q\otimes_M H_p}
\ar[d]^{\overline{\phi_{q,p}}}
\\
H_{\overline{p}\otimes \overline{q}}
\ar[rr]^{=}
&&
H_{\overline{q\otimes p}}
\ar[rr]^{\psi_{p\otimes q}}
&&
\overline{H_{q\otimes p}}
}
$$
and
$$
\xymatrix{
H_{\overline{\overline{p}}}
\ar[rr]^{\psi_p}\ar[d]_{=}
&&
\overline{H_{\overline{p}}}
\ar[d]^{\overline{\psi_{p}}}
\\
H_p
\ar[rr]^{\cong}
&&
\overline{\overline{H_p}}.
}
$$
where we just write $\cong$ for the obvious isomorphism.
\end{lem}
\begin{proof}
That $\psi_p$ is a well-defined $M-M$ bilinear isomorphism is trivial. Commutativity of the diagrams follows by looking at $M$-bounded vectors.
\end{proof}

\begin{lem}\label{lem:Rigid}
For $p\in P_{\alpha\to\alpha}$, $\Phi_{\ev_{H_p}^\Bim}=\ev_{H_p}^\CF$ and $\Phi_{\coev_{H_p}^\Bim}=\coev_{H_p}^\CF$. Hence $\Phi$ preserves the rigid structure, and $\overline{\Phi_x}=\Phi_{\overline{x}}\in \CF(H_{\overline{q}}\to H_{\overline{p}})$ for all $x\in \Bim(H_p\to H_q)$.
\end{lem}
\begin{proof}
In diagrams, $\ev_{H_p}^\CF$ is given for $M$-bounded vectors by
$$
\ev_{H_p}^\CF(\overline{x}\otimes y)=
\ColorMultiply{\overline{x}}{\overline{p}}{\alphacolor}
\otimes
\ColorMultiply{y}{p}{\alphacolor}
\longmapsto
\begin{tikzpicture}[baseline=.5cm]
	\draw [thick, \alphacolor] (0,1)--(0, -.4) .. controls ++(270:.4cm) and ++(270:.4cm) .. (1.2,-.4) -- (1.2,1);
	\draw (0,2)--(0,1);
	\draw (1.2,2)--(1.2,1);
	\draw[thick, unshaded] (1.6, -.4) -- (1.6, .4) -- (.8, .4) -- (.8, -.4) -- (1.6, -.4);
	\draw[thick, unshaded] (1.6, .8) -- (1.6, 1.6) -- (.8, 1.6) -- (.8, .8) -- (1.6, .8);
	\draw[thick, unshaded] (-.4, -.4) -- (-.4, .4) -- (.4, .4) -- (.4, -.4) -- (-.4, -.4);
	\draw[thick, unshaded] (-.4, .8) -- (-.4, 1.6) -- (.4, 1.6) -- (.4, .8) -- (-.4, .8);
	\node at (0, 0) {$\overline{p}$};
	\node at (1.2, 0) {$p$};
	\node at (0, 1.2) {$\overline{x}$};
	\node at (1.2, 1.2) {$y$};
%	\node at (.2,.6) {{\scriptsize{$\overline{\alpha}$}}};
%	\node at (1.4,.6) {{\scriptsize{$\alpha$}}};
%	\node at (.6,-.5) {\scriptsize{$\alpha$}};
\end{tikzpicture}
=\langle x| y\rangle_M.
$$
Hence using the isomorphism
$$
\xymatrix{
\overline{H_p}\otimes_M H_p
\ar[rr]^{\psi_p\otimes_M \id_{H_p}}
&&
H_{\overline{p}}\otimes_M H_p
\ar[rr]^(.55){\phi_{\overline{p},p}}
&&
H_{\overline{p}\otimes p}
}
$$
from Lemmas \ref{lem:TensorProduct2} and \ref{lem:DualPreserving}, we have that $\ev_{H_p}^\CF=\Phi_{\ev_{H_p}^\Bim}$.
Since $\Phi$ is $*$-preserving by Lemma \ref{lem:StarPreserving}, we must have that $\Phi_{\coev_{H_p}^\Bim}=\coev_{H_p}^\CF$. Thus
$$
\coev_{H_p}^\CF(1_M)=\MInnerProductOp{p}{\alphacolor}{\overline{p}}\,.
$$
It is now easy to check that $\overline{\Phi_x}=\Phi_{\overline{x}}\in \CF(H_{\overline{q}}\to H_{\overline{p}})$ for all $x\in \Bim(H_p\to H_q)$.
\end{proof}

\begin{proof}[Proof of Theorem \ref{thm:BimCF}]
We must show $\Phi$ is additive, monoidal, dual-preserving, $*$-preserving, fully faithful, and essentially surjective.

Additivity on objects and essentially surjective come for free, since the objects are the same. Additivity on morphisms and fully faithful follows from Lemma \ref{lem:FullyFaithful}. Monoidal follows from Lemma \ref{lem:TensorProduct2},  $*$-preserving follows from Lemmas \ref{lem:FullyFaithful} and \ref{lem:StarPreserving}, and dual-preserving follows from Lemmas \ref{lem:DualPreserving} and \ref{lem:Rigid}. Note that the results of the previous five lemmas extend in the obvious ways to direct sums by looking at matrices over the maps $H_p\to H_q$.
\end{proof}

%%%%%%%%%%%%%%%%%%%%%%%%%%%%%%%%%%%%%%%%%%%%%%%%%%%%%%%%%%%%%%%%%%%%
%%%%%%%%%%%%%%%%%%%%%%%%%%%%%%%%%%%%%%%%%%%%%%%%%%%%%%%%%%%%%%%%%%%%
%%%%%%%%%%%%%%%%%%%%%%%%%%%%%%%%%%%%%%%%%%%%%%%%%%%%%%%%%%%%%%%%%%%%
\section{The isomorphism class of $M$}\label{sec:vNa}

We now determine the isomorphism class of the $II_1$-factor $M$ from Section \ref{sec:GJS}.

\begin{assumption}
Note that if our rigid $C^*$-tensor category $\sC$ is countably generated, then the isomorphism classes of objects of $\sC$ form a countable set. Let $\cS$ consist of a set of representatives for the isomorphism classes of objects in $\sC$, and let $\cL=\set{X\oplus\overline{X}}{X\in\cS}$ as in Assumption \ref{assume:Countable}.
Again, the objects in $\cL$ are not simple, and all have dimension greater than 1.
In particular, $\cF_\sC(\cL)$ is not locally finite, and each vertex has self loops, since objects of the form $2\oplus X\oplus \overline{X}$ are in $\cL$.
\end{assumption}

With this assumption, we prove in Section \ref{sec:vNAGraph} that $M \cong L(\F_{\I})$. In the case that $\sC$ has finitely many isomorphism classes of simple objects, i.e., $\sC$ is a unitary fusion category, we can find a single object $X$ that generates $\sC$. We explain in Remark \ref{rem:Finite} that if we choose $\cL=\{X\oplus\overline{X}\}$, then $M\cong L(\F_t)$ with
$$
t = 1 + \dim(\sC)(2\dim(X) - 1),
$$
similar to the result in \cite{MR2807103}.

To begin, we describe in the next two sections a semifinite algebra associated to $P_\bullet$.  Many of the ideas in these two sections mirror \cite{MR2807103}, except that the semifinite algebra makes all box-spaces orthogonal. One also must be more careful since the fusion graph $\cF_{\sC}(\cL)$ is not bipartite.

%%%%%%%%%%%%%%%%%%%%%%%%%%%%%%%%%%%%%%%%%%%%%%%%%%%%%%%%%%%%%%%%%%%%
\subsection{A semifinte algebra associated to $P_\bullet$}

Set $\displaystyle \cG_{\infty} = \bigoplus_{\alpha, \beta \in \Lambda} \Gr_{\alpha, \beta}(P_\bullet)$.  We endow $\cG_{\infty}$ with the multiplication
$$
\begin{tikzpicture}[baseline=-.1cm]
	\draw (-.8,0)--(.8,0);
	\draw (0,.4)--(0,.8);
	\draw[thick, unshaded] (-.4, -.4) -- (-.4, .4) -- (.4, .4) -- (.4, -.4) -- (-.4, -.4);
	\node at (0, 0) {$x$};
	\node at (-.6, .2) {\scriptsize{$\kappa$}};
	\node at (.2,.6) {\scriptsize{$\gamma$}};
	\node at (.6, .2) {\scriptsize{$\theta$}};
\end{tikzpicture}
\wedge
\begin{tikzpicture}[baseline=-.1cm]
	\draw (-.8,0)--(.8,0);
	\draw (0,.4)--(0,.8);
	\draw[thick, unshaded] (-.4, -.4) -- (-.4, .4) -- (.4, .4) -- (.4, -.4) -- (-.4, -.4);
	\node at (0, 0) {$x$};
	\node at (-.6, .2) {\scriptsize{$\omega$}};
	\node at (.2,.6) {\scriptsize{$\gamma'$}};
	\node at (.6, .2) {\scriptsize{$\chi$}};
\end{tikzpicture}
= \delta_{\omega,\theta}
\begin{tikzpicture} [baseline = -.1cm]
	\draw (0,.4)--(0,.8);
	\draw (1.2,.4)--(1.2,.8);
	\draw (-.8,0)--(2,0);
	\draw[thick, unshaded] (-.4, -.4) -- (-.4, .4) -- (.4, .4) -- (.4, -.4) -- (-.4, -.4);
	\draw[thick, unshaded] (.8, -.4) -- (.8, .4) -- (1.6, .4) -- (1.6, -.4) -- (.8, -.4);
	\node at (0, 0) {$x$};
	\node at (1.2, 0) {$y^{*}$};
	\node at (-.6,.2) {\scriptsize{$\kappa$}};
	\node at (.2,.6) {\scriptsize{$\gamma$}};
	\node at (.6,.2) {\scriptsize{$\theta$}};
	\node at (1.4,.6) {\scriptsize{$\gamma'$}};
	\node at (1.8,.2) {\scriptsize{$\chi$}};
\end{tikzpicture}\,,
 $$
where $\kappa,\theta,\omega,\chi\in\Lambda$.  There is a (semifinite) trace, $\Tr$ on $\cG_{\infty}$ given by
$$
\Tr(x) =
\begin{tikzpicture}[baseline=.3cm]
	\draw (0,0)--(0,.8);
	\draw (.4,0) arc (90:-90:.4cm) -- (-.4,-.8) arc (270:90:.4cm);
	\filldraw[unshaded,thick] (-.4,.4)--(.4,.4)--(.4,-.4)--(-.4,-.4)--(-.4,.4);
	\draw[thick, unshaded] (-.7, .8) -- (-.7, 1.6) -- (.7, 1.6) -- (.7,.8) -- (-.7, .8);
	\node at (0,0) {$x$};
	\node at (0,1.2) {$\Sigma CTL$};
%	\node at (.6,.2) {{\scriptsize{$\alpha$}}};
%	\node at (-.6,.2) {{\scriptsize{$\alpha$}}};
%	\node at (0,-.6) {{\scriptsize{$\overline{\alpha}$}}};
\end{tikzpicture}
$$
for $x \in \Gr_{\alpha}$, and $\Tr(x) = 0$ for $x \in \Gr_{\alpha, \beta}$ with $\alpha \neq \beta$.

\begin{defn}

As in Section \ref{sec:graded}, one argues (by once again orthogonalizing) that $\Tr$ is positive definite on $\cG_{\infty}$ and multiplication is bounded on $L^{2}(\cG_{\infty})$.  Therefore, we form the (semifinite) von Neumann algebra $\cM_{\infty}=\cG_{\infty}''$ acting on $L^{2}(\cG_{\infty},\Tr)$.

We also use the von Neumann subalgebra $\cA_{\infty} \subset \cM_{\infty}$ which is generated by all boxes in $\cG_{\infty}$ with \underline{no} strings on top.
\end{defn}

As in Section \ref{sec:CategoriesOfBimodules}, let
$$p_{\alpha} =
\begin{tikzpicture} [baseline = -.1cm]
	\filldraw[unshaded,thick] (-.4,.4)--(.4,.4)--(.4,-.4)--(-.4,-.4)--(-.4,.4);
	\draw [thick, \alphacolor] (-.4, 0) -- (.4, 0);
%	\node at (0, .2) {\scriptsize{$\alpha$}};
\end{tikzpicture}\,.
$$
It is easy to see that $p_{\alpha}\wedge \cM_{\infty} \wedge p_{\alpha} = M_{\alpha}$ so that all factors in the various Jones towers in Section \ref{sec:towers} appear as cut-downs/corners of $\cM_{\infty}$.   We now record some lemmas about the structure of $\cM_{\infty}$ and $\cA_{\infty}$.

\begin{lem} \label{lem:2inftyfactor}
$\cM_{\infty}$ is a $II_{\infty}$ factor.
\end{lem}

\begin{proof}
Note that $p_{\emptyset}\wedge \cM_{\infty}\wedge p_{\emptyset} = M$ is a $II_{1}$ factor and the trace of $1 = 1_{\cM_{\infty}} = \sum_{\alpha \in \Lambda} p_{\alpha}$ (with convergence of the orthogonal sum in the strong operator topology) is infinite, so we only need to show that $\cM_{\infty}$ is a factor.  To do so, we show that the central support of $p_{\emptyset}$ in $\cM_{\infty}$ is 1, which is enough since $p_{\emptyset}\wedge \cM_{\infty}\wedge p_{\emptyset}$ is a factor.  We let
$$
w_{\alpha} =
\begin{tikzpicture} [baseline=-.1cm]
	\filldraw[unshaded,thick] (-.4,.4)--(.4,.4)--(.4,-.4)--(-.4,-.4)--(-.4,.4);
	\draw [thick, \alphacolor] (-.4, 0) arc(-90:0: .4cm);
%	\node at (0, 0) {\scriptsize{$\alpha$}};
\end{tikzpicture}\,,
$$
and we let $v_\alpha$ be the polar part of $w_\alpha$.  As was discussed in the proof of Lemma \ref{lem:TensorProduct}, $v_\alpha$ induces an equivalence of projections between $p_{\emptyset}$ and $e \leq p_{\alpha}$.  Let $z$ be the central support of $p_{\emptyset}$ (and $e$).  Since $p_{\alpha}\wedge \cM_{\I}\wedge p_{\alpha}$ is a factor, the central support of $e$ must lie above $p_{\alpha}$, so $z \geq p_{\alpha}$.  Since this holds for all $\alpha \in \Lambda$, we have $z \geq 1$. Hence $z = 1$, and we are finished.
\end{proof}

\begin{cor}
The factor $M_{\alpha}$ is a $\delta_{\alpha}$-amplification of $M$.
\end{cor}

\begin{cor}
The algebras $\cM_{\alpha, \beta}$ are $II_{1}$ factors.
\end{cor}
\begin{proof}
$\cM_{\alpha, \beta}$ is the compression of $\cM_{\I}$ by the finite projection $p_{\alpha} + p_{\beta}$.
\end{proof}

\begin{lem} \label{lem:1infty}
We have a direct sum decomposition
$$\displaystyle \cA_{\infty} = \bigoplus_{v \in V(\cF_\sC(\cL))} \cA_{v}
$$
where the sum is over all vertices $v$ in the fusion graph $\cF_\sC(\cL)$, and each $\cA_{v}$ is a type $I_{\infty}$ factor.  If $p$ is a minimal projection in $P_{\alpha \to \alpha}$ whose equivalence class represents the vertex $v$, then $p \leq 1_{\cA_{v}}$.
\end{lem}

\begin{proof}
For each vertex $v$, choose a minimal projection $p_{v} \in P_{\alpha_{v} \to \alpha_{v}}$ whose equivalence class corresponds to the vertex $v$.  We first see that
$$
p_{v} \wedge \cA_{\infty} \wedge p_{v} = p_{\alpha_v}P_{\alpha_{v} \to \alpha_{v}}p_{\alpha_{v}} \cong \C
$$
so $\cA_{\infty}$ has minimal projections. Letting $p\in P_{\alpha \to \alpha}$ and $q \in P_{\beta \to \beta}$, then the observation that $p\wedge \cA_{\I} \wedge q = qP_{\alpha \to \beta}p$ implies $p$ is equivalent to $q$ in the planar algebra sense if and only if $p$ is equivalent to $q$ in $\cA_{\I}$.  Let $1_{v}$ be the central support of $p_{v}$.  Then by construction, $\cA_{v} = 1_{v}\wedge \cA_{\I}\wedge 1_{v}$ is a type $I$ factor which must be type $I_{\I}$ since there are infinitely many mutually orthogonal projections equivalent to $p_{v}$ in $P_\bullet$.  It is easy to see that by construction, if $v \neq w$ then $1_{v}\wedge 1_{w} = 0$, which implies that
$$
\cA_{\I} = \bigoplus_{v \in V(\cF_\sC(\cL))} \cA_{v} \oplus \cB
$$
for some von Neumann algebra $\cB$.

We claim that $\cB = \{0\}$.  Indeed we know that $p_{\alpha}$ can be written as an orthogonal sum of projections equivalent to some of the $p_{w}$'s.  Since $1 = \sum_{\alpha \in \Lambda} p_{\alpha}$, this implies $\sum_{v} 1_{v} = 1$, and thus $\cB = \{0\}$.
\end{proof}

\begin{rem}
For the rest of this section, all multiplication will be the $\wedge$ multiplication in the GJS picture. Hence for the rest of this section, we just write $xy$ for $x\wedge y$ for convenience.
\end{rem}

We now show that we can obtain $\cM_{\I}$ from a base ``building block" $\cA_{\infty}$ and various free ``corner elements."  Fixing a color $c$, which we again represent by the color \textcolor{\cupcolor}{green}, we define
$$
X_{c} = \sum_{\substack{ \alpha \in \Lambda\\
|\alpha| \in 2\N}}
\,
\begin{tikzpicture}[baseline=-.1cm]
    \draw [thick] (-.4, -.4)--(-.4, .4)--(.4, .4)--(.4, -.4)--(-.4, -.4);
    \draw [thick, \cupcolor] (-.4, .1) arc(-90:0: .3cm);
    \draw [thick, \alphacolor] (-.4, -.2)--(.4, -.2);
%    \node at (-.3, .3) {\scriptsize{$c$}};
%    \node at (.2, 0) {\scriptsize{$\alpha$}};
\end{tikzpicture}
+
\begin{tikzpicture}[baseline=-.1cm]
    \draw [thick] (-.4, -.4)--(-.4, .4)--(.4, .4)--(.4, -.4)--(-.4, -.4);
    \draw [thick, \cupcolor] (.4, .1) arc(-90:-180: .3cm);
    \draw [thick, \alphacolor] (-.4, -.2)--(.4, -.2);
%    \node at (-.3, .3) {\scriptsize{$c$}};
%    \node at (-.2, 0) {\scriptsize{$\alpha$}};
\end{tikzpicture}
\,.
$$

\begin{rem}

We note that this sum defines a bounded operator.  Indeed, the individual terms in the sum are supported under the mutually orthogonal family of projections $\{p_{c\alpha} + p_{\alpha} : |\alpha| \in 2\N\}$ and each term in the sum has operator norm
$$\left\|\,
\begin{tikzpicture} [baseline = -.1cm]
    \draw [thick] (-.4, -.4)--(-.4, .4)--(.4, .4)--(.4, -.4)--(-.4, -.4);
    \draw [thick, \cupcolor] (-.3, .4) arc(-180:0: .3cm);
    \draw [thick, \alphacolor] (-.4, -.2)--(.4, -.2);
\end{tikzpicture} +
\begin{tikzpicture} [baseline = -.1cm]
    \draw [thick] (-.4, -.4)--(-.4, .4)--(.4, .4)--(.4, -.4)--(-.4, -.4);
    \draw [thick, \cupcolor] (.4, .1) arc(-90:-180: .3cm);
    \draw [thick, \cupcolor] (-.4, .1) arc(-90:0: .3cm);
    \draw [thick, \alphacolor] (-.4, -.2)--(.4, -.2);
    \end{tikzpicture}
\,
\right\|_{\infty}^{1/2} =
\left\|\,
\begin{tikzpicture} [baseline = -.1cm]
    \draw [thick] (-.4, -.4)--(-.4, .4)--(.4, .4)--(.4, -.4)--(-.4, -.4);
    \draw [thick, \cupcolor] (-.3, .4) arc(-180:0: .3cm);
\end{tikzpicture} +
\begin{tikzpicture} [baseline = -.1cm]
    \draw [thick] (-.4, -.4)--(-.4, .4)--(.4, .4)--(.4, -.4)--(-.4, -.4);
    \draw [thick, \cupcolor] (.4, .1) arc(-90:-180: .3cm);
    \draw [thick, \cupcolor] (-.4, .1) arc(-90:0: .3cm);
\end{tikzpicture}
\,
\right\|_{\infty}^{1/2}.
$$
We also note the simple fact that for $|\alpha|$ even, $p_{c\alpha}X_{c}p_{\alpha}$ is the corner diagram
$$
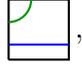
\begin{tikzpicture}[baseline=-.1cm]
    \draw [thick] (-.4, -.4)--(-.4, .4)--(.4, .4)--(.4, -.4)--(-.4, -.4);
    \draw [thick, \cupcolor] (-.4, .1) arc(-90:0: .3cm);
    \draw [thick, \alphacolor] (-.4, -.2)--(.4, -.2);
%    \node at (-.3, .3) {\scriptsize{$c$}};
%    \node at (.2, 0) {\scriptsize{$\alpha$}};
\end{tikzpicture}\,,
$$
so each term in the sum defining $X_{c}$ appears in the von Neumann algebra $W^{*}(\cA_{\infty}, X_{c})$.  We only sum over $\alpha$ with $|\alpha|$ even as it makes computations involving freeness in Section \ref{sec:free1} much easier.
\end{rem}

The $X_{c}$ elements give us a very nice way of obtaining $\cM_{\infty}$ from $\cA_{\infty}$.

\begin{lem}
$\cM_{\I} \cong W^{*}(\cA_{\I}, \, \{X_{c} : c \in \cL\})$.
\end{lem}

\begin{proof}
Every element in $\cG_{\infty}$ is a linear combination of elements of the following form:
$$
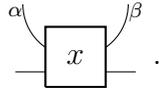
\begin{tikzpicture}[baseline=-.1cm]
    \draw (-.7, .7) arc(-180:0: .7cm);
    \draw (-.8, -.2)--(.8, -.2);
    \draw [thick, unshaded] (-.4, -.4)--(-.4, .4)--(.4, .4)--(.4, -.4)--(-.4, -.4);
    \node at (0, 0) {$x$};
    \node at (-.8, .6) {\scriptsize{$\alpha$}};
    \node at (.8, .6) {\scriptsize{$\beta$}};
\end{tikzpicture}.
$$
The above diagram is $x \in \cA_{\infty}$ multiplied on the left and right by diagrams of the form
$$
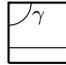
\begin{tikzpicture}[baseline=-.1cm]
    \draw [thick] (-.4, -.4)--(-.4, .4)--(.4, .4)--(.4, -.4)--(-.4, -.4);
    \draw (-.4, .1) arc(-90:0: .3cm);
    \draw (-.4, -.2)--(.4, -.2);
%    \node at (-.3, .3) {\scriptsize{$c$}};
    \node at (0, .2) {\scriptsize{$\gamma$}};
\end{tikzpicture}
$$
and their adjoints.  This diagram is a product of diagrams of the form
$$
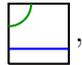
\begin{tikzpicture}[baseline=-.1cm]
    \draw [thick] (-.4, -.4)--(-.4, .4)--(.4, .4)--(.4, -.4)--(-.4, -.4);
    \draw [thick, \cupcolor] (-.4, .1) arc(-90:0: .3cm);
    \draw [thick, \alphacolor] (-.4, -.2)--(.4, -.2);
%    \node at (-.3, .3) {\scriptsize{$c$}};
%    \node at (.2, 0) {\scriptsize{$\alpha$}};
\end{tikzpicture}\,,
$$
so we only need to check that the above diagram is in $W^{*}(\cA_{\I}, \{X_{c} : c \in \cL\})$ when $|\alpha|$ is odd.  This is easy since it can be written as the product
$$
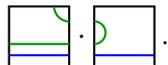
\begin{tikzpicture} [baseline = -.1cm]
    \draw [thick] (-.4, -.4)--(-.4, .4)--(.4, .4)--(.4, -.4)--(-.4, -.4);
    \draw[thick, \cupcolor] (.4,.2) arc (270:180:.2cm);
    \draw[thick, \cupcolor] (-.4, -.1)--(.4, -.1);
    \draw [thick, \alphacolor] (-.4, -.25)--(.4, -.25);
%    \node at (0, -.2) {\scriptsize{$\alpha$}};
%    \node at (-.25, .2) {\scriptsize{$c$}};
\end{tikzpicture}
\cdot
\begin{tikzpicture} [baseline = -.1cm]
    \draw [thick] (-.4, -.4)--(-.4, .4)--(.4, .4)--(.4, -.4)--(-.4, -.4);
    \draw [thick, \alphacolor] (-.4, -.25)--(.4, -.25);
    \draw [thick, \cupcolor] (-.4, .2) arc(90:-90: .15cm);
%    \node at (0, -.15) {\scriptsize{$\alpha$}};
%    \node at (-.25, .2) {\scriptsize{$c$}};
\end{tikzpicture} \, .
$$
\end{proof}

There is a $\Tr$-preserving conditional expectation $E: \cM_{\I} \rightarrow \cA_{\I}$ given by
$$
E(x) =
\begin{tikzpicture}[baseline=.3cm]
	\draw (0,0)--(0,.8);
	\draw (-.8, 0)--(.8, 0);
	\filldraw[unshaded,thick] (-.4,.4)--(.4,.4)--(.4,-.4)--(-.4,-.4)--(-.4,.4);
	\draw[thick, unshaded] (-.7, .8) -- (-.7, 1.6) -- (.7, 1.6) -- (.7,.8) -- (-.7, .8);
	\node at (0,0) {$x$};
	\node at (0,1.2) {$\Sigma CTL$};
%	\node at (.6,.2) {{\scriptsize{$\alpha$}}};
%	\node at (-.6,.2) {{\scriptsize{$\alpha$}}};
%	\node at (0,-.6) {{\scriptsize{$\overline{\alpha}$}}};
\end{tikzpicture}\,,
$$
and $E$ induces normal completely positive maps $(\eta_{c,c} = \eta_{c})_{c \in \cL}$ on $\cA_{\infty}$ satisfying
$$
\eta_{c}(y) = E(X_{c}yX_{c}) =
\begin{tikzpicture}[baseline = 0cm]
    \draw(-2, 0)--(2, 0);
    \filldraw[unshaded,thick] (-.4,.4)--(.4,.4)--(.4,-.4)--(-.4,-.4)--(-.4,.4);
    \filldraw[unshaded,thick] (-1.6,.4)--(-.8,.4)--(-.8,-.4)--(-1.6,-.4)--(-1.6,.4);
    \filldraw[unshaded,thick] (1.6,.4)--(.8,.4)--(.8,-.4)--(1.6,-.4)--(1.6,.4);
    \draw[thick, \cupcolor] (-1.2, .4) arc(180:90: .3cm) -- (.9, .7) arc(90:0: .3cm);
    \node at (0, 0) {$y$};
    \node at (-1.2, 0) {$X_{c}$};
    \node at (1.2, 0) {$X_{c}$};
\end{tikzpicture}\,.
$$
For $b \neq a$, we have trivial ``off-diagonal" maps $\eta_{a, b}$ on $\cA_{\I}$ satisfying $\eta_{a, b}(y) = E(X_{a}yX_{b}) = 0$.  This gives a straightforward diagrammatic procedure for evaluating $E(y_{0}X_{c_{1}}y_{1}X_{c_{2}}\cdots X_{c_{n}}y_{n})$ for $y_{i} \in \cA_{\infty}$.  First, write the word $y_{0}X_{c_{1}}y_{1}X_{c_{2}}\cdots X_{c_{n}}y_{n}$ as
$$
\begin{tikzpicture} [baseline = 0cm]
    \draw(-2, 0)--(2, 0);
    \filldraw[unshaded,thick] (-.4,.4)--(.4,.4)--(.4,-.4)--(-.4,-.4)--(-.4,.4);
    \filldraw[unshaded,thick] (-1.6,.4)--(-.8,.4)--(-.8,-.4)--(-1.6,-.4)--(-1.6,.4);
    \filldraw[unshaded,thick] (1.6,.4)--(.8,.4)--(.8,-.4)--(1.6,-.4)--(1.6,.4);
    \node at (0, 0) {$X_{c_{1}}$};
    \node at (-1.2, 0) {$y_{0}$};
    \node at (1.2, 0) {$y_{1}$};
    \draw (0, .4)--(0, .8);
\end{tikzpicture}
\cdots
\begin{tikzpicture} [baseline = 0cm]
    \draw(-2, 0)--(2, 0);
    \filldraw[unshaded,thick] (-.4,.4)--(.4,.4)--(.4,-.4)--(-.4,-.4)--(-.4,.4);
    \filldraw[unshaded,thick] (-1.6,.4)--(-.8,.4)--(-.8,-.4)--(-1.6,-.4)--(-1.6,.4);
    \filldraw[unshaded,thick] (1.6,.4)--(.8,.4)--(.8,-.4)--(1.6,-.4)--(1.6,.4);
    \node at (0, 0) {$X_{c_{n}}$};
    \node at (-1.2, 0) {$y_{n-1}$};
    \node at (1.2, 0) {$y_{n}$};
    \draw (0, .4)--(0, .8);
\end{tikzpicture}\,.
$$
Then sum over all planar ways to connect the strings on top.  Whenever we see a term of the form
$$
\begin{tikzpicture}[baseline = 0cm]
    \draw(-2, 0)--(2, 0);
    \filldraw[unshaded,thick] (-.4,.4)--(.4,.4)--(.4,-.4)--(-.4,-.4)--(-.4,.4);
    \filldraw[unshaded,thick] (-1.6,.4)--(-.8,.4)--(-.8,-.4)--(-1.6,-.4)--(-1.6,.4);
    \filldraw[unshaded,thick] (1.6,.4)--(.8,.4)--(.8,-.4)--(1.6,-.4)--(1.6,.4);
    \draw (-1.2, .4) arc(180:90: .3cm) -- (.9, .7) arc(90:0: .3cm);
    \node at (0, 0) {$y$};
    \node at (-1.2, 0) {$X_{c_{i}}$};
    \node at (1.2, 0) {$X_{c_{j}}$};
\end{tikzpicture}\, ,
$$
we replace it with $\eta_{c_{i}, c_{j}}(y)$. It is straightforward to check that $E$ and the $\eta_{c_{i}, c_{j}}$ satisfy the following recurrence relation:

\begin{align}
E(y_{0}X_{c_{1}}y_{1}X_{c_{2}}&\cdots y_{n-1}X_{c_{n}}y_{n}) \notag\\
&= \sum_{k=2}^{n} y_{0} \cdot \eta_{c_{1}, c_{k}}(E(y_{1}X_{c_{2}}\cdots X_{c_{k-1}}y_{k-1}))\cdot E(y_{k}X_{c_{k+1}}...X_{c_{n}}y_{n}).\label{eqn:recurrence}
\end{align}
Also, by definition of $\eta_{c_{i}, c_{j}}$, it follows that the map on $\cA_{\I} \otimes B(\cH)$ given by
$(y_{i, j}) \mapsto (\eta_{c_{i}, c_{j}}(y_{i,j}))$
is normal and completely positive.  This, combined with Recurrence \eqref{eqn:recurrence} implies that the elements $(X_{c})_{c \in \cL}$ form an $\cA_{\I}$-valued semicircular family with covariance $(\eta_{c_{i}, c_{j}})$ as in \cite{MR1704661}. Since $\eta_{c_{i}, c_{j}} = 0$ for $c_{i}\neq c_{j}$, the family $(X_{c})_{c \in \cL}$ is free with amalgamation over $\cA_{\I}$ with respect to $E$ \cite{MR1704661}.  We record what we have established above in the following lemma.

\begin{lem} \label{lem:free1}
$\cM_{\I} = W^{*}(\cA_{\I}, \{X_{c}: c \in \cL\})$, and the elements  $(X_{c})_{c \in \cL}$ form an $\cA_{\I}$-valued semicircular family with covariance $(\eta_{c_{i}, c_{j}})$ and are free with amalgamation over $\cA_{\I}$.
\end{lem}

%%%%%%%%%%%%%%%%%%%%%%%%%%%%%%%%%%%%%%%%%%%%%%%%%%%%%%%%%%%%%%%%%%%%
\subsection{$M$ as an amalgamated free product} \label{sec:free1}

By compressing the algebra $\cA_{\infty}$ by the projection $1_{e} = \sum_{|\alpha| \in 2\N} p_{\alpha}$, Lemma \ref{lem:1infty} implies that if we set $\cA_{e} = W^{*}((P_{\beta \rightarrow \alpha})_{\alpha, \beta \in 2\N})$ then
$$
\cA_{e} = \bigoplus_{v \in V(\cF_\sC(\cL))} \cB_{v}.
$$
$\cB_{v}$ is a type $I_{\infty}$ factor which is a cut-down of $\cA_{v}$ by $1_{e}$.
Note that every vertex in $v\in V(\cF_\sC(\cL))$ appears in the direct sum because every vertex possesses at least one self-loop.

Similarly, if one sets $1_{o} = \sum_{|\alpha| \in (1 + 2\N)} p_{\alpha}$ and $\cA_{o} = 1_{o}\cA_{\I}1_{o}$, then we have
$$
\cA_{o} = \bigoplus_{v \in V(\cF_\sC(\cL))} \cC_{v}
$$
where $\cC_{v}$ is a type $I_{\I}$ factor which is a cut-down of $\cA_{v}$ by $1_{o}$.

For each vertex $v \in \cF_\sC(\cL)$, we choose a minimal projection $p_{v} \in \cA_{e}$ whose equivalence class is represented by $v$ with $p_{\emptyset}$ the empty diagram.  If $x \in \cA_{e}$, it follows that
$$
\eta_{c}(x) =
\begin{tikzpicture} [baseline = 0cm]
    \draw (-.8, 0)--(.8, 0);
    \filldraw[thick, unshaded] (-.4, .4)--(-.4, -.4)--(.4, -.4)--(.4,.4)--(-.4, .4);
    \draw [thick, \cupcolor] (-.8, .6)--(.8, .6);
    \node at (0,0) {$x$};
\end{tikzpicture}
$$
so that $\eta_{c}(p_{v})$ is a finite projection in $\cA_{o}$, and each $\eta_{c}$ acts as a (non-unital) $W^{*}$ algebra homomorphism from $\cA_{e}$ into $\cA_{o}$.  Set $Q = \sum_{v} p_{v}$.  Then there is a family of partial isometries $(V_{i})_{i \in I}$ satisfying $V_{i}^{*}V_{i} = Q$ and $\sum_{i \in I} V_{i}V_{i}^{*} = 1_{\cA_{e}}$.  Note that $\sum_{c \in \cL}\eta_{c}(Q)$ defines a projection in $\cA_{o}$ since $\eta_{a}(Q) \perp \eta_{b}(Q)$ for $a \neq b$. We examine the compression of $\cM_{\I}$ by $T = Q + \sum_{c \in \cL}\eta_{c}(Q)$.

\begin{lem} \label{lem:compression1}
$T\cM_{\infty}T = W^{*}(T\cA_{\infty}T, \, (TX_{c}T)_{x \in \cL})$
\end{lem}

\begin{proof}
By the choices of the partial isometries $V_{i}$, we know that if we let $W_{i} = \sum_{c \in \cL} \eta_{c}(V_{i})$ then $\sum_{i \in I}(V_{i} + W_{i})(V_{i}+W_{i})^{*} = 1$, so
$$
T\cM_{\infty}T = W^{*}\left(T\cA_{\I}T, \, \left\{\left(V_{i} + W_{i}\right)^{*} X_{c} \left(V_{j} + W_{j}\right): i,j \in I \,  \, c \in \cL\right\}\right)
$$
The relation $y \cdot X_{c} = X_{c} \eta_{c}(y)$ for any $y \in \cA_{e}$ is a straightforward diagrammatic check.  Applying this relation to each $\left(V_{i} + W_{i} \right)^{*} X_{c} \left(V_{j} + W_{j}\right)$ gives $TX_{c}T$.
\end{proof}

Since $\sum_{i \in I} W_{i}W_{i}^{*} = 1_{\cA_{o}}$, it follows that if $v \in V(\cF_{\sC}(\cL))$ then there is a minimal projection $q_{v} \in \cA_{o}$ whose equivalence class represents $v$ and sits under $1_{\cA_{o}}\cdot T$.  Note that this means that for all $v$, $q_{v}$ and $p_{v}$ are orthogonal minimal projections that sit under $1_{v}$.   We will next examine the cut-down of $T\cM_{\I}T$ by $F = \sum_{v} p_{v} + q_{v}$.

Suppose the vertices $v$ and $w$ are distinct and connected in $\cF_{\sC}(\cL)$ and let $C(e)$ be the color of an edge connecting $v$ and $w$.  Let $e_{1},...,e_{k}$ be the (necessarily finite) collection of edges connecting $v$ and $w$ with $C(e_{i}) = c$.  We let $u_{e_{i}}^{v}$ be a collection of partial isometries such that
$$
(u_{e_{i}}^{v})^{*}u_{e_{j}}^{v} = \delta_{ij}q_{v} \text{ and } \sum_{i=1}^{k} u_{e_{i}}^{v}(u_{e_{i}}^{v})^{*} = \eta_{c}(p_{w})\cdot 1_{v}
$$

Let $r_{v} \in F\cA_{\I}F$ be a partial isometry satisfying $r_{v}r_{v}^{*} = p_{v}$ and $r_{v}^{*}r_{v} = q_{v}$ and $r_{w} \in F\cA_{\I}F$ be a partial isometry satisfying $r_{w}r_{w}^{*} = p_{w}$ and $r_{w}^{*}r_{w} = q_{w}$.  The elements defined can be described diagrammatically as follows:
\begin{align*}
p_{v} =
\begin{tikzpicture} [baseline = 0cm]
    \draw  (-.8, 0)--(.8, 0);
    \filldraw [thick, unshaded] (.4, .4) --(.4, -.4)--(-.4, -.4)--(-.4, .4)--(.4, .4);
    \node at (0, 0) {$p_{v}$};
    \node at (-.6, .2) {\scriptsize{$\alpha_{v}$}};
    \node at (.6, .2) {\scriptsize{$\alpha_{v}$}};
\end{tikzpicture}\, , \,
p_{w} =
\begin{tikzpicture} [baseline = 0cm]
    \draw  (-.8, 0)--(.8, 0);
    \filldraw [thick, unshaded] (.4, .4) --(.4, -.4)--(-.4, -.4)--(-.4, .4)--(.4, .4);
    \node at (0, 0) {$p_{w}$};
    \node at (-.6, .2) {\scriptsize{$\alpha_{w}$}};
    \node at (.6, .2) {\scriptsize{$\alpha_{w}$}};
\end{tikzpicture}\, , \,
q_{v} &=
\begin{tikzpicture} [baseline = 0cm]
    \draw  (-.8, 0)--(.8, 0);
    \filldraw [thick, unshaded] (.4, .4) --(.4, -.4)--(-.4, -.4)--(-.4, .4)--(.4, .4);
    \node at (0, 0) {$q_{v}$};
    \node at (-.6, .2) {\scriptsize{$\beta_{v}$}};
    \node at (.6, .2) {\scriptsize{$\beta_{v}$}};
\end{tikzpicture}\, , \,
q_{w} =
\begin{tikzpicture} [baseline = 0cm]
    \draw  (-.8, 0)--(.8, 0);
    \filldraw [thick, unshaded] (.4, .4) --(.4, -.4)--(-.4, -.4)--(-.4, .4)--(.4, .4);
    \node at (0, 0) {$q_{w}$};
    \node at (-.6, .2) {\scriptsize{$\beta_{w}$}};
    \node at (.6, .2) {\scriptsize{$\beta_{w}$}};
\end{tikzpicture}\, , \\
u_{e_{i}}^{v} =
\begin{tikzpicture} [baseline = 0cm]
    \draw  (-.8, -.2)--(.8, -.2);
    \draw [thick, \cupcolor] (-.8, .2)--(-.4, .2);
    \filldraw [thick, unshaded] (.4, .4) --(.4, -.4)--(-.4, -.4)--(-.4, .4)--(.4, .4);
    \node at (0, 0) {$u_{e_{i}}^{v}$};
    \node at (-.6, 0) {\scriptsize{$\alpha_{w}$}};
    \node at (.6, 0) {\scriptsize{$\beta_{v}$}};
\end{tikzpicture}\, , \,
r_{v} &=
\begin{tikzpicture} [baseline = 0cm]
    \draw  (-.8, 0)--(.8, 0);
    \filldraw [thick, unshaded] (.4, .4) --(.4, -.4)--(-.4, -.4)--(-.4, .4)--(.4, .4);
    \node at (0, 0) {$r_{v}$};
    \node at (-.6, .2) {\scriptsize{$\alpha_{v}$}};
    \node at (.6, .2) {\scriptsize{$\beta_{v}$}};
\end{tikzpicture}\, , \,
r_{w} =
\begin{tikzpicture} [baseline = 0cm]
    \draw  (-.8, 0)--(.8, 0);
    \filldraw [thick, unshaded] (.4, .4) --(.4, -.4)--(-.4, -.4)--(-.4, .4)--(.4, .4);
    \node at (0, 0) {$r_{w}$};
    \node at (-.6, .2) {\scriptsize{$\alpha_{w}$}};
    \node at (.6, .2) {\scriptsize{$\beta_{w}$}};
\end{tikzpicture}\, .
\end{align*}

Set $x_{i}$ to be the following diagram:
$$
x_{i} =
\begin{tikzpicture} [baseline = 0cm]
    \draw  (-.8, -.2)--(.8, -.2);
    \draw [thick, \cupcolor] (-.4, .2) arc(270:90: .2cm) -- (.8, .6);
    \filldraw [thick, unshaded] (.4, .4) --(.4, -.4)--(-.4, -.4)--(-.4, .4)--(.4, .4);
    \node at (0, 0) {$u_{e_{i}}^{v}$};
    \node at (-.6, 0) {\scriptsize{$\alpha_{w}$}};
    \node at (.6, 0) {\scriptsize{$\beta_{v}$}};
\end{tikzpicture}
$$
We have the following lemma about the left and right supports of $x_{i}$:

\begin{lem} \label{lem:orthogonalsupport}
The left support of $x_{i}$ is $p_{w}$ and the right support of $x_{i}$ lies under the projection $1_{w} \cdot i_{c}(q_{v})$ where we have for $z \in P_{\overline{\gamma} \rightarrow \overline{\chi}}$
$$
i_{c}(z) =
\begin{tikzpicture} [baseline = 0cm]
    \draw  (-.8, 0)--(.8, 0);
    \filldraw [thick, unshaded] (.4, .4) --(.4, -.4)--(-.4, -.4)--(-.4, .4)--(.4, .4);
    \node at (0, 0) {$z$};
    \node at (-.6, .2) {\scriptsize{$\gamma$}};
    \node at (.6, .2) {\scriptsize{$\chi$}};
    \draw [thick, \cupcolor] (-.8, .6) -- (.8, .6);
\end{tikzpicture}\, .
$$
Furthermore if $i \neq j$ then the right supports of $x_{i}$ and $x_{j}$ are orthogonal.
\end{lem}

\begin{proof}
Note that $x_{i}x_{i}^{*}$ is the diagram
$$
\begin{tikzpicture} [baseline = 0cm]
    \draw  (-.8, -.2)--(2.2, -.2);
    \draw [thick, \cupcolor] (-.4, .2) arc(270:90: .2cm) -- (1.8, .6) arc(90:-90: .2cm);
    \filldraw [thick, unshaded] (.4, .4) --(.4, -.4)--(-.4, -.4)--(-.4, .4)--(.4, .4);
    \filldraw [thick, unshaded] (1.8, .4)--(1.8, -.4)--(.8, -.4)--(.8, .4)--(1.8,.4);
    \node at (1.3, 0) {$(u_{e_{i}}^{v})^{*}$};
    \node at (0, 0) {$u_{e_{i}}^{v}$};
    \node at (-.6, 0) {\scriptsize{$\alpha_{w}$}};
    \node at (.6, 0) {\scriptsize{$\beta_{v}$}};
    \node at (2, 0) {\scriptsize{$\alpha_{w}$}};
\end{tikzpicture}\, .
$$
For any $\gamma \in \Lambda$, let $\Phi: P\op_{\overline{\gamma}c \rightarrow \overline{\gamma}c} \rightarrow P\op_{\overline{\gamma} \rightarrow \overline{\gamma}}$ be the trace preserving conditional expectation where $P\op_{\overline{\gamma} \rightarrow \overline{\gamma}}$ includes unitally into $P\op_{\overline{\gamma}c \rightarrow \overline{\gamma}c}$ via $i_{c}$.    The above diagram is a scalar multiple of $\Phi(u_{e_{i}}^{v}\cdot (u_{e_{i}}^{v})^{*})$.  As $u_{e_{i}}^{v}\cdot (u_{e_{i}}^{v})^{*}$ lies under $\eta_{c}(p_{w}) = i_{c}(p_{w})$ it follows that this diagram is a scalar multiple of $p_{w}$, proving the claim about the left support.  As left and right supports are equivalent projections, it follows that the right support of $x_{i}$ lies under $1_{w}$.  We note that $\Phi(x_{i}^{*}x_{i}) = k (u_{e_{i}}^{v})^{*}u_{e_{i}}^{v} = k q_{v}$ for some scalar $k$, implying that the right support of $x_{i}$ must lie under $i_{c}(q_{v})$.  Finally, if $i \neq j$ then $x_{i}x_{j}^{*}$ is a constant multiple of $\Phi(u_{e_{i}}^{v}\cdot (u_{e_{j}}^{v})^{*})$.  We know this must be a scalar multiple of $p_{w}$, but that scalar must be zero as $u_{e_{i}}^{v}\cdot (u_{e_{j}}^{v})^{*}$ has trace zero.  This proves the orthogonality of the right supports of the $x_{i}$.
\end{proof}

Let $f_{i}$ be the right support of $x_{i}$.  As $i_{c}(p_{v})\cdot1_{w}$ can be written as a sum of $k$ orthogonal projections equivalent to $p_{w}$, we conclude that $\sum_{i=1}^{k} f_{i} = i_{c}(p_{v})\cdot1_{w}$.  We also recognize that $i_{c}(r_{v})$ is a partial isometry with left support $i_{c}(q_{v})$ and right support $i_{c}(p_{v})$. We consider the elements $y_{i} = r_{w}^{*}\cdot x_{i} \cdot i_{c}(r_{v}^{*})$ which diagrammatically look like
$$
y_{i} =
\begin{tikzpicture} [baseline = 0cm]
    \draw (-.8, -.2)--(3.2, -.2);
    \filldraw [thick, unshaded] (.4, .4)--(.4, -.4)--(-.4, -.4)--(-.4, .4)--(.4, .4);
    \filldraw [thick, unshaded] (1.6, .4)--(1.6, -.4)--(.8, -.4)--(.8, .4)--(1.6, .4);
    \filldraw [thick, unshaded] (2.8, .4)--(2.8, -.4)--(2, -.4)--(2, .4)--(2.8, .4);
    \draw [thick, \cupcolor] (.8, .2) arc(270:90: .2cm)--(3.2, .6);
    \node at (0,0) {$r_{w}^{*}$};
    \node at (1.2, 0) {$u_{e}^{v}$};
    \node at (2.4, 0) {$r_{v}^{*}$};
    \node at (-.6, 0) {\scriptsize{$\beta_{w}$}};
    \node at (.6, 0) {\scriptsize{$\alpha_{w}$}};
    \node at (1.8, 0) {\scriptsize{$\beta_{v}$}};
    \node at (3, 0) {\scriptsize{$\alpha_{v}$}};
\end{tikzpicture}
$$
note that the left support of $y_{i}$ is $q_{w}$.  From Lemma \ref{lem:orthogonalsupport}, the right supports of the $y_{i}$ are orthogonal and sum up to $i_{c}(p_{v})\cdot1_{w}$.  As $q_{w}$ is minimal in $\cA_{\I}$, it follows that $y_{i}$ is a scalar multiple of a partial isometry.  We define the partial isometry $u_{e}^{w}$ by the equation $(u_{e_{i}}^{w})^{*} = l \cdot y_{i}$ with $l$ an appropriate constant.  It follows that $(u_{e_{i}}^{w})^{*}u_{e_{j}}^{w} = \delta_{ij}q_{w}$ and $\sum_{i=1}^{k} u_{e_{i}}^{w}(u_{e_{i}}^{w})^{*} = \eta_{c}(p_{v})\cdot 1_{w}$.

One can manipulate the diagram above to show that the relation
$$
(u_{e}^{v})^{*} = (\overline{l})^{-1} \cdot
\begin{tikzpicture} [baseline = 0cm]
    \draw (-.8, -.2)--(3.2, -.2);
    \filldraw [thick, unshaded] (.4, .4)--(.4, -.4)--(-.4, -.4)--(-.4, .4)--(.4, .4);
    \filldraw [thick, unshaded] (1.6, .4)--(1.6, -.4)--(.8, -.4)--(.8, .4)--(1.6, .4);
    \filldraw [thick, unshaded] (2.8, .4)--(2.8, -.4)--(2, -.4)--(2, .4)--(2.8, .4);
    \draw [thick, \cupcolor] (.8, .2) arc(270:90: .2cm)--(3.2, .6);
    \node at (0,0) {$r_{v}^{*}$};
    \node at (1.2, 0) {$u_{e}^{w}$};
    \node at (2.4, 0) {$r_{w}^{*}$};
    \node at (-.6, 0) {\scriptsize{$\beta_{v}$}};
    \node at (.6, 0) {\scriptsize{$\alpha_{v}$}};
    \node at (1.8, 0) {\scriptsize{$\beta_{w}$}};
    \node at (3, 0) {\scriptsize{$\alpha_{w}$}};
\end{tikzpicture}
$$
holds.  This discussion thus proves the following useful lemma:

\begin{lem} \label{lem:relation}
Let $e_{1}, \dots, e_{k}$ be all of the edges of color $c$ in the fusion graph connecting distinct vertices $v$ and $w$.  Then one can find partial isometries $\{u_{e_{i}}^{v}: 1\leq i \leq k\}$ and $\{u_{e_{i}}^{w}: 1\leq i \leq k\}$ in $\cA_{o}$ satisfying
\begin{align*}
(u_{e_{i}}^{v})^{*}u_{e_{j}}^{v} = \delta_{ij}q_{v},& \, \, \, \sum_{i=1}^{k} u_{e_{i}}^{v}(u_{e_{i}}^{v})^{*} = \eta_{c}(p_{w})\cdot 1_{v}\\
(u_{e_{i}}^{w})^{*}u_{e_{j}}^{w} = \delta_{ij}q_{w},& \, \, \, \sum_{i=1}^{k} u_{e_{i}}^{w}(u_{e_{i}}^{w})^{*} = \eta_{c}(p_{v})\cdot 1_{w}
\end{align*}
and the relations
$$
(u_{e}^{v})^{*} = (\overline{l})^{-1} \cdot
\begin{tikzpicture} [baseline = 0cm]
    \draw (-.8, -.2)--(3.2, -.2);
    \filldraw [thick, unshaded] (.4, .4)--(.4, -.4)--(-.4, -.4)--(-.4, .4)--(.4, .4);
    \filldraw [thick, unshaded] (1.6, .4)--(1.6, -.4)--(.8, -.4)--(.8, .4)--(1.6, .4);
    \filldraw [thick, unshaded] (2.8, .4)--(2.8, -.4)--(2, -.4)--(2, .4)--(2.8, .4);
    \draw [thick, \cupcolor] (.8, .2) arc(270:90: .2cm)--(3.2, .6);
    \node at (0,0) {$r_{v}^{*}$};
    \node at (1.2, 0) {$u_{e}^{w}$};
    \node at (2.4, 0) {$r_{w}^{*}$};
    \node at (-.6, 0) {\scriptsize{$\beta_{v}$}};
    \node at (.6, 0) {\scriptsize{$\alpha_{v}$}};
    \node at (1.8, 0) {\scriptsize{$\beta_{w}$}};
    \node at (3, 0) {\scriptsize{$\alpha_{w}$}};
\end{tikzpicture}\, \text{ and }
(u_{e}^{w})^{*} = l \cdot
\begin{tikzpicture} [baseline = 0cm]
    \draw (-.8, -.2)--(3.2, -.2);
    \filldraw [thick, unshaded] (.4, .4)--(.4, -.4)--(-.4, -.4)--(-.4, .4)--(.4, .4);
    \filldraw [thick, unshaded] (1.6, .4)--(1.6, -.4)--(.8, -.4)--(.8, .4)--(1.6, .4);
    \filldraw [thick, unshaded] (2.8, .4)--(2.8, -.4)--(2, -.4)--(2, .4)--(2.8, .4);
    \draw [thick, \cupcolor] (.8, .2) arc(270:90: .2cm)--(3.2, .6);
    \node at (0,0) {$r_{w}^{*}$};
    \node at (1.2, 0) {$u_{e}^{v}$};
    \node at (2.4, 0) {$r_{v}^{*}$};
    \node at (-.6, 0) {\scriptsize{$\beta_{w}$}};
    \node at (.6, 0) {\scriptsize{$\alpha_{w}$}};
    \node at (1.8, 0) {\scriptsize{$\beta_{v}$}};
    \node at (3, 0) {\scriptsize{$\alpha_{v}$}};
\end{tikzpicture}
$$
for some nonzero constant $l$.
\end{lem}

We now turn our attention over to the edges $e$ which are loops.  Suppose $v$ is a vertex and $e_{i},...,e_{k}$ represent all of the loops of color $c$ connecting $v$ to itself.  As above, we can find partial isometries $u_{e_{i}}^{v}$ with right support $q_{v}$ and orthogonal left supports under $1_{v}\cdot i_{c}(p_{v})$, however for reasons that will become apparent later we desire a stronger property about these partial isometries.

\begin{lem} \label{lem:loop}
There exists a set of partial isometries $\{u_{e_{i}}^{v}: 1 \leq i \leq k\}$ satisfying
$$
(u^{v}_{e_{i}})^{*}u^{v}_{e_{j}} = \delta_{ij} q_{v}, \, \, \, \sum_{i=1}^{k} u_{e_{i}}^{v}(u_{e_{i}}^{v})^{*} = \eta_{c}(p_{v})\cdot 1_{v},
$$
and the following relation:
$$
(u_{e}^{v})^{*} = k \cdot
\begin{tikzpicture} [baseline = 0cm]
    \draw (-.8, -.2)--(3.2, -.2);
    \filldraw [thick, unshaded] (.4, .4)--(.4, -.4)--(-.4, -.4)--(-.4, .4)--(.4, .4);
    \filldraw [thick, unshaded] (1.6, .4)--(1.6, -.4)--(.8, -.4)--(.8, .4)--(1.6, .4);
    \filldraw [thick, unshaded] (2.8, .4)--(2.8, -.4)--(2, -.4)--(2, .4)--(2.8, .4);
    \draw [thick, \cupcolor] (.8, .2) arc(270:90: .2cm)--(3.2, .6);
    \node at (0,0) {$r_{v}^{*}$};
    \node at (1.2, 0) {$u_{e}^{v}$};
    \node at (2.4, 0) {$r_{v}^{*}$};
    \node at (-.6, 0) {\scriptsize{$\beta_{v}$}};
    \node at (.6, 0) {\scriptsize{$\alpha_{v}$}};
    \node at (1.8, 0) {\scriptsize{$\beta_{v}$}};
    \node at (3, 0) {\scriptsize{$\alpha_{v}$}};
\end{tikzpicture}
$$
for $k$ a unimodular constant.
\end{lem}

\begin{proof}
Let $u_{1} \in \cA_{\I}$ be a partial isometry with right support $q_{v}$ and left support under $1_{v}\cdot i_{c}(p_{v})$.  Consider the operator $y_{1}$ satisfying
$$ y_{1}^{*} =
\begin{tikzpicture} [baseline = 0cm]
    \draw (-.8, -.2)--(.8, -.2);
    \filldraw [thick, unshaded] (-.4, -.4) -- (-.4, .4)--(.4, .4)--(.4, -.4)--(-.4, -.4);
    \draw [thick, \cupcolor] (.4, .2)--(.8, .2);
    \node at (0, 0) {$u^{*}_{1}$};
    \node at (-.6, 0) {\scriptsize{$\beta_{v}$}};
    \node at (.6, 0) {\scriptsize{$\alpha_{v}$}};
\end{tikzpicture} +
\begin{tikzpicture} [baseline = 0cm]
    \draw (-.8, -.2)--(3.2, -.2);
    \filldraw [thick, unshaded] (.4, .4)--(.4, -.4)--(-.4, -.4)--(-.4, .4)--(.4, .4);
    \filldraw [thick, unshaded] (1.6, .4)--(1.6, -.4)--(.8, -.4)--(.8, .4)--(1.6, .4);
    \filldraw [thick, unshaded] (2.8, .4)--(2.8, -.4)--(2, -.4)--(2, .4)--(2.8, .4);
    \draw [thick, \cupcolor] (.8, .2) arc(270:90: .2cm)--(3.2, .6);
    \node at (0,0) {$r_{v}^{*}$};
    \node at (1.2, 0) {$u_{1}$};
    \node at (2.4, 0) {$r_{v}^{*}$};
    \node at (-.6, 0) {\scriptsize{$\beta_{v}$}};
    \node at (.6, 0) {\scriptsize{$\alpha_{v}$}};
    \node at (1.8, 0) {\scriptsize{$\beta_{v}$}};
    \node at (3, 0) {\scriptsize{$\alpha_{v}$}};
\end{tikzpicture}.
$$
If this is zero, then by setting $k = -1$, we have produced a partial isometry satisfying the appropriate diagrammatic relation.  If not, then $y_{1}^{*}$ has left support $q_{v}$ and right support under $1_{v} \cdot i_{c}(p_{v})$ so it is a scalar multiple of a partial isometry.  The operator $y_{1}^{*}$ is also fixed under the map
$$
x \mapsto
\begin{tikzpicture} [baseline = 0cm]
    \draw (-.8, -.2)--(3.2, -.2);
    \filldraw [thick, unshaded] (.4, .4)--(.4, -.4)--(-.4, -.4)--(-.4, .4)--(.4, .4);
    \filldraw [thick, unshaded] (1.6, .4)--(1.6, -.4)--(.8, -.4)--(.8, .4)--(1.6, .4);
    \filldraw [thick, unshaded] (2.8, .4)--(2.8, -.4)--(2, -.4)--(2, .4)--(2.8, .4);
    \draw [thick, \cupcolor] (.8, .2) arc(270:90: .2cm)--(3.2, .6);
    \node at (0,0) {$r_{v}^{*}$};
    \node at (1.2, 0) {$x^{*}$};
    \node at (2.4, 0) {$r_{v}^{*}$};
    \node at (-.6, 0) {\scriptsize{$\beta_{v}$}};
    \node at (.6, 0) {\scriptsize{$\alpha_{v}$}};
    \node at (1.8, 0) {\scriptsize{$\beta_{v}$}};
    \node at (3, 0) {\scriptsize{$\alpha_{v}$}};
\end{tikzpicture}
$$
so a scalar multiple of $y_{1}$ satisfies the appropriate diagrammatic relation and is a partial isometry. In either case, we have produced a partial isometry $u_{e_{1}}^{v}$ satisfying the diagrammatic relation. To produce another partial isometry $u_{e_{2}}^{v}$ with right support $q_{v}$, left support orthogonal to that of $u_{e_{1}}^{v}$, and satisfying the diagrammatic relation, we pick a partial isometry $u_{2}$ with right support $q_{v}$ and left support orthogonal to that of $u_{e_{1}}^{v}$.  By Lemma \ref{lem:orthogonalsupport} and the discussion afterwards, the right supports of
$$
\begin{tikzpicture} [baseline = 0cm]
    \draw (-.8, -.2)--(3.2, -.2);
    \filldraw [thick, unshaded] (.4, .4)--(.4, -.4)--(-.4, -.4)--(-.4, .4)--(.4, .4);
    \filldraw [thick, unshaded] (1.6, .4)--(1.6, -.4)--(.8, -.4)--(.8, .4)--(1.6, .4);
    \filldraw [thick, unshaded] (2.8, .4)--(2.8, -.4)--(2, -.4)--(2, .4)--(2.8, .4);
    \draw [thick, \cupcolor] (.8, .2) arc(270:90: .2cm)--(3.2, .6);
    \node at (0,0) {$r_{v}^{*}$};
    \node at (1.2, 0) {$u_{2}$};
    \node at (2.4, 0) {$r_{v}^{*}$};
    \node at (-.6, 0) {\scriptsize{$\beta_{v}$}};
    \node at (.6, 0) {\scriptsize{$\alpha_{v}$}};
    \node at (1.8, 0) {\scriptsize{$\beta_{v}$}};
    \node at (3, 0) {\scriptsize{$\alpha_{v}$}};
\end{tikzpicture} \text{ and }k\cdot
\begin{tikzpicture} [baseline = 0cm]
    \draw (-.8, -.2)--(3.2, -.2);
    \filldraw [thick, unshaded] (.4, .4)--(.4, -.4)--(-.4, -.4)--(-.4, .4)--(.4, .4);
    \filldraw [thick, unshaded] (1.6, .4)--(1.6, -.4)--(.8, -.4)--(.8, .4)--(1.6, .4);
    \filldraw [thick, unshaded] (2.8, .4)--(2.8, -.4)--(2, -.4)--(2, .4)--(2.8, .4);
    \draw [thick, \cupcolor] (.8, .2) arc(270:90: .2cm)--(3.2, .6);
    \node at (0,0) {$r_{v}^{*}$};
    \node at (1.2, 0) {$u_{e_{1}}^{v}$};
    \node at (2.4, 0) {$r_{v}^{*}$};
    \node at (-.6, 0) {\scriptsize{$\beta_{v}$}};
    \node at (.6, 0) {\scriptsize{$\alpha_{v}$}};
    \node at (1.8, 0) {\scriptsize{$\beta_{v}$}};
    \node at (3, 0) {\scriptsize{$\alpha_{v}$}};
\end{tikzpicture} = (u_{e_{1}}^{v})^{*}
$$
are orthogonal so it follows that by considering the element
$$ y_{2}^{*} =
\begin{tikzpicture} [baseline = 0cm]
    \draw (-.8, -.2)--(.8, -.2);
    \filldraw [thick, unshaded] (-.4, -.4) -- (-.4, .4)--(.4, .4)--(.4, -.4)--(-.4, -.4);
    \draw [thick, \cupcolor] (.4, .2)--(.8, .2);
    \node at (0, 0) {$u^{*}_{2}$};
    \node at (-.6, 0) {\scriptsize{$\beta_{v}$}};
    \node at (.6, 0) {\scriptsize{$\alpha_{v}$}};
\end{tikzpicture} +
\begin{tikzpicture} [baseline = 0cm]
    \draw (-.8, -.2)--(3.2, -.2);
    \filldraw [thick, unshaded] (.4, .4)--(.4, -.4)--(-.4, -.4)--(-.4, .4)--(.4, .4);
    \filldraw [thick, unshaded] (1.6, .4)--(1.6, -.4)--(.8, -.4)--(.8, .4)--(1.6, .4);
    \filldraw [thick, unshaded] (2.8, .4)--(2.8, -.4)--(2, -.4)--(2, .4)--(2.8, .4);
    \draw [thick, \cupcolor] (.8, .2) arc(270:90: .2cm)--(3.2, .6);
    \node at (0,0) {$r_{v}^{*}$};
    \node at (1.2, 0) {$u_{2}$};
    \node at (2.4, 0) {$r_{v}^{*}$};
    \node at (-.6, 0) {\scriptsize{$\beta_{v}$}};
    \node at (.6, 0) {\scriptsize{$\alpha_{v}$}};
    \node at (1.8, 0) {\scriptsize{$\beta_{v}$}};
    \node at (3, 0) {\scriptsize{$\alpha_{v}$}};
\end{tikzpicture}.
$$
and arguing as in the beginning of the proof, we produce the desired element $u_{e_{2}}^{v}$.  Iterating this procedure produces the set $\{u_{e_{i}}^{v}: 1 \leq i \leq k\}$.
\end{proof}

We therefore assume in the rest of the section that our partial isometries satisfy the relations in either Lemma \ref{lem:relation} or Lemma \ref{lem:loop}.  The identity $\sum_{v} \sum_{e\sim v} u_{e}^{v}(u_{e}^{v})^{*} = \sum_{c \in \cL} \eta_{c}(Q)$ follows by the choices of the partial isometries $u_{e}^{v}$.  We therefore have the following lemma:

\begin{lem} \label{lem:compression2}
For each edge $e$ of the fusion graph $\cF_\sC(\cL))$, we define operators $Y_{e}$ as follows:  If $C(e) = c$ and $e$ connects distinct vertices $v$ and $w$ then $Y_{e} = p_{v}X_{c}u_{e}^{w} + (u_{e}^{w})^{*}X_{v}p_{v} + p_{w}X_{c}u_{e}^{v} + (u_{e}^{v})^{*}X_{c}p_{w}$ and if $e$ connects the vertex $v$ to itself then $Y_{e} = p_{v}X_{c}u_{e}^{v} + (u_{e}^{v})^{*}X_{c}p_{v}$.  Then we have $F\cM_{\I}F = W^{*}(F\cA_{\I}F, \, (Y_{e})_{e \in E})$
\end{lem}

\begin{proof}
First note that by compressing by the appropriate elements, of $F\cA_{\I}F$, each term in the sums defining $Y_{e}$ is in $W^{*}(F\cA_{\I}F, \, (Y_{e})_{e \in E(\cF_\sC(\cL))})$.  The identity
$$
\sum_{v \in V(\cF_\sC(\cL))} \sum_{\substack{e: s(e) = v \\ C(e) = c}} = p_{v}X_{c}u_{e}^{t(e)}(u_{e}^{t(e)})^{*} = TX_{c}T
$$
implies that words in $TX_{c}T$ (compressed on either end by $F$) can be approximated by words in $p_{v}X_{c}u_{e}^{w}$ and their adjoints so by Lemma \ref{lem:compression1} we are done.
\end{proof}

\begin{lem} \label{lem:EasyGenerators}
If $v$ and $w$ are distinct vertices connected by an edge $e$ with $C(e) = c$ then $r_{v}(u^{v}_{e})^{*}X_{c}p_{w}r_{w}$ and $p_{v}X_{c}(u_{e}^{w})$ are nonzero scalar multiples of each other.

If $f$ is a loop at $v$ then $r_{v}(u^{v}_{f})^{*}X_{c}p_{v}r_{v}$ and $p_{v}X_{c}(u_{f}^{v})$ are nonzero scalar multiples of each other.

\end{lem}

\begin{proof}
First observe that
$$
p_{v}X_{c}u_{e}^{w} = X_{c}\eta_{c}(p_{v})u_{e}^{w} = X_{c}u_{e}^{w}.
$$
and $r_{v}(u^{v}_{e})^{*}X_{c}p_{w}r_{w} = r_{v}(u_{v}^{e})^{*}X_{c}r_{w}$.  The element $r_{v}(u_{v}^{e})^{*}X_{c}r_{w}$ is represented diagrammatically as
$$
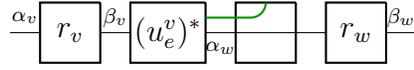
\begin{tikzpicture} [baseline = 0cm]
    \draw (-.8, 0) -- (4.6, 0);
    \filldraw [thick, unshaded] (.4, .4) --(.4, -.4)--(-.4, -.4)--(-.4, .4)--(.4, .4);
    \filldraw [thick, unshaded] (1.8, .4) --(1.8, -.4)--(.8, -.4)--(.8, .4)--(1.8, .4);
    \filldraw [thick, unshaded] (4.2, .4) --(4.2, -.4)--(3.4, -.4)--(3.4, .4)--(4.2, .4);
    \draw [thick] (3, .4) --(3, -.4)--(2.2, -.4)--(2.2, .4)--(3, .4);
    \draw [thick, \cupcolor] (1.8, .2)--(2.4, .2) arc(-90:0: .2cm);
    \node at (0,0) {$r_{v}$};
    \node at (1.3, 0) {$(u_{e}^{v})^{*}$};
    \node at (3.8, 0) {$r_{w}$};
    \node at (-.6, .2) {\scriptsize{$\alpha_{v}$}};
    \node at (.6, .2) {\scriptsize{$\beta_{v}$}};
    \node at (2, -.2) {\scriptsize{$\alpha_{w}$}};
    \node at (4.4, .2) {\scriptsize{$\beta_{w}$}};
\end{tikzpicture}
$$
which can be rewritten as
$$
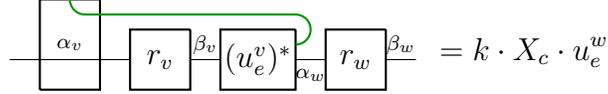
\begin{tikzpicture} [baseline = 0cm]
    \draw (-.8, 0) -- (4.6, 0);
    \draw [thick] (-.4, -.4)--(-.4, .8)--(.4, .8)--(.4, -.4)--(-.4, -.4);
    \filldraw [thick, unshaded] (1.6, .4) --(1.6, -.4)--(.8, -.4)--(.8, .4)--(1.6, .4);
    \filldraw [thick, unshaded] (3, .4) --(3, -.4)--(2, -.4)--(2, .4)--(3, .4);
    \filldraw [thick, unshaded] (4.2, .4) --(4.2, -.4)--(3.4, -.4)--(3.4, .4)--(4.2, .4);
    \node at (0, .2) {\scriptsize{$\alpha_{v}$}};
    \node at (1.8, .2) {\scriptsize{$\beta_{v}$}};
    \node at (3.2, -.2) {\scriptsize{$\alpha_{w}$}};
    \node at (4.4, .2) {\scriptsize{$\beta_{w}$}};
    \draw [thick, \cupcolor] (3, .2) arc(-90:90: .2cm) -- (.2, .6) arc(270:180: .2cm);
    \node at (1.2, 0) {$r_{v}$};
    \node at (2.5, 0) {$(u_{e}^{v})^{*}$};
    \node at (3.8, 0) {$r_{w}$};
\end{tikzpicture}\,  = k \cdot X_{c}\cdot u_{e}^{w}
$$
for an appropriate constant $k \neq 0$, proving the first statement.  The proof of the second statement uses the exact same diagrammatic argument.
\end{proof}

We now examine the algebras
$$
\cM_{e} = W^{*}(F\cA_{\I}F, Y_{e}) = W^{*}(F\cA_{\I}F, p_{v}X_{c}u_{e}^{w} + (u_{e}^{w})^{*}X_{c}p_{v})
$$
where the last equality follows from Lemma \ref{lem:EasyGenerators}. First assume $v$ and $w$ are distinct and $\Tr(p_{v}) \geq \Tr(p_{w}) = \Tr(q_{w})$.  One can show, using exactly the same techniques as in \cite{MR2732052} and \cite{MR2807103} that $p_{v}X_{c}u_{e}^{w}(u_{e}^{w})^{*}X_{c}p_{v}$ is a free Poisson element with an atom of size $\frac{\Tr(p_{v}) - \Tr(p_{w})}{\Tr(p_{v})}$ at 0 in the compressed algebra $p_{v}\cM_{\infty}p_{v}$. Also, $(u_{e}^{w})^{*}X_{c}p_{v}X_{c}u_{e}^{w}$ is a free poisson element with no atoms in the compressed algebra $q_{w}\cM_{\I}q_{w}$.  Using the polar part of $p_{v}X_{c}u_{e}^{w}$ as well as the partial isometries $r_{v}$ and $r_{w}$ and Lemma \ref{lem:EasyGenerators}, we see that
$$
(1_{v} + 1_{w})F\cM_{e}F(1_{v} + 1_{w}) = \left(L(\Z)\otimes M_{4}(\C)\right) \bigoplus \overset{p'_{v}, q'_{v}}{\underset{\Tr(p_{v}) - \Tr(p_{w})}{M_{2}(\C)}}
$$
where the notation means that the copy of $M_{2}(\C)$ has two orthogonal minimal projections $p_{v}'$ and $q_{v}'$ such that $p_{v}' \leq p_{v}$ and $q'_{v} \leq q_{v}$ and $\Tr(p_{v}') = \Tr(p_{v}) - \Tr(p_{w})$.

If $e$ is a loop at $v$ then using the partial isometry $r_{v}$ and Lemma \ref{lem:EasyGenerators}, we see that
$$
1_{v}\cdot F\cM_{e}F\cdot 1_{v} = L(\Z) \otimes M_{2}(\C).
$$
For simplicity, we set $X_{e} = p_{v}X_{c}u_{e}^{w} + (u_{e}^{w})^{*}X_{c}p_{v}$ so that $\cM_{e} = W^{*}(F\cM_{\I}F, X_{e})$.

\begin{lem} \label{lem:free2}
The algebras $(\cM_{e})_{e \in E(\cF_{\sC}(\cL))}$ are free with amalgamation over $F\cA_{\I}F$.
\end{lem}

\begin{proof}
From \cite{MR1704661} the elements $X_{e}$ are $F\cA_{\I}F$ semicircular, so to show freeness we need only to show that if $e \neq e'$, $E(X_{e} y X_{e'}) = 0$ for all $y \in F\cA_{\I}F$.  This can only be nonzero if $e$ and $e'$ have the came color, so we assume $C(e) = c = C(e')$.  Also, this expectation can be nonzero only if $e$ and $e'$ share a common vertex, $v$.

We first assume that $e$ and $e'$ are not loops.  Assume $e$ connects $v$ and $w$ and $e'$ connects $v$ and $w'$.  We see that $E(X_{e}p_{v}X_{e'}) = (u_{e}^{w})^{*}i_{c}(p_{v})u_{e'}^{w'} = (u_{e}^{w})^{*}u_{e'}^{w'}$ which is zero since $u_{e}^{w}$ and $u_{e'}^{w'}$ have orthogonal left supports if $w = w'$ and is clearly zero if $w \neq w'$.  For $E(X_{e}q_{w}X_{e'})$ to be nonzero, we need $w = w'$.  Assuming this, we get $p_{v}\Phi(u_{e}^{w}\cdot q_{w} \cdot (u_{e'}^{w'})^{*}) p_{v}$.  As was discussed in the proof of Lemma \ref{lem:orthogonalsupport},
$$
\Phi(u_{e}^{w}\cdot q_{w} \cdot (u_{e'}^{w'})^{*}) = \Phi(u_{e}^{w}(u_{e'}^{w'})^{*}) = 0
$$
since $u_{e}^{w}$ and $u_{e'}^{w'}$ have orthogonal left supports.  It is also straightforward to check that the only elements $y$ in $F\cA_{\I}F$ where $X_{e}yX_{e'} \neq 0$ are of the form $y = a\cdot p_{v} + b\cdot q_{w}$.

We now assume that $e$ is a loop at $v$ and $e'$ is an arbitrary edge connected to $v$.  The discussion in the previous paragraph implies that $E(X_{e}pX_{e'}) = 0$ for any projection $p \in F\cA_{\I}F$; however, we also have to consider $E(X_{e}r_{v}^{*}X_{e'})$.  This is the following diagram:
$$
\begin{tikzpicture}[baseline = 0cm]
    \draw (-.8, -.2)--(5.6, -.2);
    \filldraw [thick, unshaded] (-.4, -.4)--(-.4, .4)--(.4, .4)--(.4, -.4)--(-.4, -.4);
    \filldraw [thick, unshaded] (.8, -.4)--(.8, .4)--(1.6, .4)--(1.6, -.4)--(.8, -.4);
    \filldraw [thick, unshaded] (2, -.4)--(2, .4)--(2.8, .4)--(2.8, -.4)--(2, -.4);
    \filldraw [thick, unshaded] (3.2, -.4)--(3.2, .4)--(4, .4)--(4, -.4)--(3.2, -.4);
    \filldraw [thick, unshaded] (4.4, -.4)--(4.4, .4)--(5.2, .4)--(5.2, -.4)--(4.4, -.4);
    \draw [thick, \cupcolor] (.8, .2) arc(270:90: .2cm) -- (4, .6) arc(90:0: .2cm) arc(180:270:.2cm);
    \node at (0, 0) {$p_{v}$};
    \node at (1.2, 0) {$u_{e}^{v}$};
    \node at (2.4, 0) {$r_{v}^{*}$};
    \node at (3.6, 0) {$p_{v}$};
    \node at (4.8, 0) {$u_{e'}^{v}$};
    \node at (-.6, 0) {\scriptsize{$\alpha_{v}$}};
    \node at (.6, 0) {\scriptsize{$\alpha_{v}$}};
    \node at (1.8, 0) {\scriptsize{$\beta_{v}$}};
    \node at (3, 0) {\scriptsize{$\alpha_{v}$}};
    \node at (4.2, 0) {\scriptsize{$\alpha_{v}$}};
    \node at (5.4, 0) {\scriptsize{$\beta_{v}'$}};
\end{tikzpicture} =
\begin{tikzpicture} [baseline = 0cm]
    \draw (-.8, -.2)--(5.6, -.2);
    \filldraw [thick, unshaded] (-.4, -.4)--(-.4, .4)--(.4, .4)--(.4, -.4)--(-.4, -.4);
    \filldraw [thick, unshaded] (.8, -.4)--(.8, .4)--(1.6, .4)--(1.6, -.4)--(.8, -.4);
    \filldraw [thick, unshaded] (2, -.4)--(2, .4)--(2.8, .4)--(2.8, -.4)--(2, -.4);
    \filldraw [thick, unshaded] (3.2, -.4)--(3.2, .4)--(4, .4)--(4, -.4)--(3.2, -.4);
    \filldraw [thick, unshaded] (4.4, -.4)--(4.4, .4)--(5.2, .4)--(5.2, -.4)--(4.4, -.4);
    \draw [thick, \cupcolor] (2, .2) arc(270:90: .2cm) -- (4, .6) arc(90:0: .2cm) arc(180:270:.2cm);
    \node at (0, 0) {$r_{v}$};
    \node at (1.2, 0) {$r_{v}^{*}$};
    \node at (2.4, 0) {$u_{e}^{v}$};
    \node at (3.6, 0) {$r_{v}^{*}$};
    \node at (4.8, 0) {$u_{e'}^{v}$};
    \node at (-.6, 0) {\scriptsize{$\alpha_{v}$}};
    \node at (.6, 0) {\scriptsize{$\beta_{v}$}};
    \node at (1.8, 0) {\scriptsize{$\alpha_{v}$}};
    \node at (3, 0) {\scriptsize{$\beta_{v}$}};
    \node at (4.2, 0) {\scriptsize{$\alpha_{v}$}};
    \node at (5.4, 0) {\scriptsize{$\beta_{v}'$}};
\end{tikzpicture}\, ,
$$
where $\beta_{v}' = \beta_{v}$ if $e'$ is also a loop at $v$.  By our choice of $u_{v}^{e}$, this diagram is a scalar multiple of $r_{v}\cdot (u_{e}^{v})^{*}u_{e'}^{v}$ which is 0.  Similarly, $E(X_{e}r_{v}^{*}X_{e'}) = 0$ and we are done.
\end{proof}

Note that we have
$$
F\cM_{\I}F = \underset{F\cA_{\I}F}{*} (\cM_{e})_{e \in E(\cF_{\sC}(\cL))}.
$$
The algebra $F\cA_{\I}F$ has the decomposition
$$
F\cA_{\I}F = \bigoplus_{v \in \cF_{\sC}(\cL)} \overset{p_{v}, q_{v}}{M_{2}(\C)}.
$$
Recall that $Q = \sum_{v} p_{v}$.  Thus $Q$ is an abelian projection with central support 1 in $F\cA_{\I}F$.  Therefore (see for example \cite{MR2765550} or \cite{MR2051399})
$$
Q\cM_{\I}Q = \underset{Q\cA_{\I}Q}{*} (Q\cM_{e}Q)_{e \in E(\cF_{\sC}(\cL))}.
$$
The algebra $Q\cA_{\I}Q$ is simply $\ell^{\I}(V(\cF_{\sC}(\cL)))$, the bounded functions on the vertices of $\cF_{\sC}(\cL)$.  If $e$ connects two distinct vertices $v$ and $w$ with $\Tr(p_{v}) \geq \Tr(p_{w})$ then
$$
Q\cM_{e}Q = L(\Z) \otimes M_{2}(\C) \bigoplus \overset{p_{v}'}{\C} \bigoplus \ell^{\I}(V(\cF_{\sC}(\cL)) \setminus \{v, w\})
$$
and if $e$ is a loop at $v$ then
$$
Q\cM_{e}Q = L(\Z) \bigoplus \ell^{\I}(V(\cF_{\sC}(\cL)) \setminus \{v\}).
$$

In the next section, we will use the free product expression for $Q\cM_{\I}Q$ to determine the isomorphism class of $M = p_{\emptyset}(Q\cM_{\I}Q)p_{\emptyset}$.

%%%%%%%%%%%%%%%%%%%%%%%%%%%%%%%%%%%%%%%%%%%%%%%%%%%%%%%%%%%%%%%%%%%%
\subsection{von Neumann algebras associated to graphs} \label{sec:vNAGraph}

The following notation will be useful in this section:

\begin{nota} \label{nota:vNA}
Throughout this section, we will be concerned with finite von Neumann algebras $(\cN, tr)$ which can be written in the form
$$
 \cN = \overset{p_{0}}{\underset{\gamma_{0}}{\cN_{0}}} \oplus \bigoplus_{j \in J} \overset{p_{j}}{\underset{\gamma_{j}}{L(\F_{t_{j}})}} \oplus  \bigoplus_{k \in K} \overset{q_{k}}{\underset{\alpha_{k}}{M_{n_{k}}}}
 $$
 where $\cN_{0}$ is a diffuse hyperfinite von Neumann algebra, $L(\F_{t_{j}})$ is an interpolated free group factor with parameter $t_{j}$, $M_{n_{k}}$ is the algebra of $n_{k} \times n_{k}$ matrices over the scalars, and the sets $J$ and $K$ are at most finite and countably infinite respectively.  We use $p_{j}$ to denote the projection in $L(\F_{t_{j}})$ corresponding to the identity of $L(\F_{t_{j}})$ and $q_{k}$ to denote a minimal projection in $M_{n_{k}}$. The projections $p_{j}$ and $q_{k}$ have traces $\gamma_{j}$ and $\alpha_{k}$ respectively.  Let $p_{0}$ be the identity in $\cN_{0}$ with trace $\gamma_{0}$.   We write $\overset{p,q}{M_{2}}$ to mean $M_{2}$ with a choice of minimal orthogonal projections $p$ and $q$.
 \end{nota}

We begin by letting $\Gamma$ be a connected weighted graph with weighting $\gamma$.  Let $\ell^{\I}(\Gamma)$ be the bounded functions on the vertices of $\Gamma$ and let $p_{v}$ be the delta function at $v$.  We endow $\ell^{\I}(\Gamma)$ with a trace $\Tr$ satisfying $\Tr(p_{v}) = \gamma_{v}$.  As in \cite{1208.2933} we now describe how to use the edges to associate a free product von Neumann algebra $\cN(\Gamma)$ to $\Gamma$.

\begin{defn}
Let $e$ be an edge in $\Gamma$.  We define algebras $\cN_{e}$ as follows:  If $e$ connects two distinct edges $v$ and $w$ with $\gamma_{v} \geq \gamma_{w}$ then
$$
\cN_{e} = \underset{2\gamma_{w}}{M_{2}(\C) \otimes L(\Z)} \oplus \underset{\gamma_{v} - \gamma_{w}}{\overset{p^{e}_{v}}{\C}} \oplus \ell^{\infty}(\Gamma \setminus \{v, w\}).
$$
If $e$ is a loop at the vertex $v$ then
$$
\cN_{e} = \overset{p_{v}}{\underset{\gamma_{v}}{L(\Z)}} \oplus \ell^{\infty}(\Gamma \setminus \{v\}).
$$
If $e$ connects two different vertices then the trace on $\underset{2\gamma_{w}}{M_{2}(\C) \otimes L(\Z)} \oplus \underset{\gamma_{v} - \gamma_{w}}{\overset{p^{e}_{v}}{\C}}$ is given by $\tr_{M_{2}} \otimes \tr_{L(\Z)} + \tr_{\C}$. $\cN_{e}$ includes $\ell^{\I}$ $\Gamma$ by letting $p_{w} = e_{1, 1}\otimes 1$ and $p_{v} = e_{2, 2}\otimes 1 + p_{v}^{e}$ so that $p^{e}_{v} \leq p_{v}$.  If $e$ is a loop that $\cN_{e}$ then $\cN_{e}$ includes $\ell^{\I}(\Gamma)$ in the obvious way.

Let $E_{e}$ be the $\Tr-$preserving conditional expectation from $\cN_{e}$ to $\ell^{\I}(\Gamma)$.  We define $\cN(\Gamma) = \underset{\ell^{\I}(\Gamma)}{*} (\cN_{e}, E_{e})_{e \in E(\Gamma)}$.
\end{defn}

Note that if $\Gamma$ has no loops then this definition of $\cN(\Gamma)$ is the same as $\cM(\Gamma)$ in \cite{1208.2933}.  Also observe that $Q\cM_{\I}Q$ in section \ref{sec:free1} is $\cN(\cF_{\sC}(\cL))$.  We now define some notation which will be useful in determining the isomorphism class of $\cN(\Gamma)$.

\begin{defn} \label{defn:H2} We write $v \sim w$ if $v$ and $w$ are connected by at least 1 edge in $\Gamma$ and denote $n_{v, w}$ be the number of edges joining $v$ and $w$.  We set $\alpha^{\Gamma}_{v} = \sum_{w\sim v} n_{v, w}\gamma_{w}$, and define $B(\Gamma) = \{ v \in V(\Gamma) : \gamma_{v} > \alpha^{\Gamma}_{v}\}$. Note that if there is a loop, $e$, at $v$ then $v \not\in B(\Gamma)$. \end{defn}

Assume for the moment that $\Gamma \subset \Gamma'$ are finite, connected, weighted graphs with at least two edges (so that $\cN(\Gamma)$ and $\cN(\Gamma')$ are finite von Neumann algebras).  There is a natural inclusion $\cN(\Gamma) \rightarrow \cN(\Gamma')$ which will not be unital if $\Gamma'$ has a larger vertex set.  We will prove the following theorem, which is along the same lines as \cite{1208.2933}.

\begin{thm} \label{thm:standardembedding}
$\cN(\Gamma)$ has the form
$$
\cN(\Gamma) \cong \overset{p^{\Gamma}}{L(\F_{t_{\Gamma}})} \oplus \underset{{v \in B(\Gamma)}}{\bigoplus} \overset{r_{v}^{\Gamma}}{\underset{\gamma_{v} - \alpha^{\Gamma}_{v}}{\C}}
$$
where $r_{v}^{\Gamma} \leq p_{v}$ and $t_{\Gamma}$ is such that this algebra has the appropriate free dimension.  If $\Gamma$ is a proper subgraph of $\Gamma'$ then the unital inclusion $p^{\Gamma}\cN(\Gamma)p^{\Gamma} \rightarrow p^{\Gamma}\cN(\Gamma')p^{\Gamma}$ is a standard embedding of interpolated free group factors.
\end{thm}

See \cite{MR1201693} and \cite{MR2765550} for a discussion of free dimension and rules for computing it.  Also we refer the reader to \cite{MR1201693} for the definition of a standard embedding of interpolated free group factors.  Whenever $A$ and $B$ are interpolated free group factors and $A$ is unitally included into $B$ then we write $A \overset{s.e.}{\hookrightarrow} B$ to indicate that the inclusion of $A$ into $B$ is a standard embedding.  In this section, we will extensively use the following properties of standard embeddings which can be found in \cite{MR1201693} and \cite{MR1363079}.

\begin{itemize}
\item[(1)]
If $A$ is an interpolated free group factor, the canonical inclusion $A \rightarrow A * \cN$ is a standard embedding whenever $\cN$ is of the form in Notation \ref{nota:vNA}.
\item[(2)]
A composite of standard embeddings is a standard embedding.
\item[(3)]
If $A_{n} = L(\F_{s_{n}})$ with $s_{n} < s_{n+1}$ for all $n$ and $\phi_{n}: A_{n} \overset{s.e.}{\hookrightarrow} A_{n+1}$, then the inductive limit of the $A_{n}$ with respect to the $\phi_{n}$ is $L(\F_{s})$ where $s = \displaystyle \lim_{n \rightarrow \infty}s_{n}$.
\item[(4)]
If $t > s$  then $\phi: L(\F_{s}) \overset{s.e.}{\hookrightarrow} L(\F_{t})$ if and only if for any nonzero projection $p \in L(\F_{s})$, $\phi|_{pL(\F_{s})p}: pL(\F_{s})p \overset{s.e.}{\hookrightarrow} \phi(p)L(\F_{t})\phi(p)$.
\end{itemize}

Theorem \ref{thm:standardembedding} was proved in \cite{1208.2933} in the case that $\Gamma$ contained no loops, so we only to need to modify the arguments there to incorporate what happens when $\Gamma$ contains loops.  We will prove Theorem \ref{thm:standardembedding} by inducting on the number of edges in the graph.  We divide this into two lemmas.

\begin{lem} \label{lem:basecase}
Suppose $\Gamma$ is a connected graph with 2 edges, one of which is a loop.  Then $\cN(\Gamma)$ is of the form in Theorem \ref{thm:standardembedding}.
\end{lem}

\begin{proof}
There are two cases to consider.  The first when $\Gamma$ has one vertex with two loops and the other when $\Gamma$ has two vertices, $w$ and $w$ with a loop at $v$ and an edge connecting $v$ to $w$.

\item[\text{\underline{Case 1:}}]
Assume $\Gamma$ has one vertex with two loops.  It follows immediately from the definition of $\cN(\Gamma)$ that $\cN(\Gamma) = L(\F_{2})$ which is in agreement with Theorem \ref{thm:standardembedding}.\\

\item[\text{\underline{Case 2:}}]
Assume $\Gamma$ has two vertices $w$ and $w$ with a loop, $e$, at $v$ and an edge, $f$, connecting $v$ to $w$.  There are two subcases to consider:  when $\gamma_{v} \geq \gamma_{w}$ and $\gamma_{v} < \gamma_{w}$.

\item[\text{\underline{Case 2a:}}]
Assume $\gamma_{v} \geq \gamma_{w}$ and set $D  = \ell^{\I}(\Gamma)$.  $\cN(\Gamma)$ has the form
$$
\cN(\Gamma) = \left(\overset{p_{v}}{\underset{\gamma_{v}}{L(\Z)}} \oplus \overset{p_{w}}{\underset{\gamma_{w}}{\C}}\right) \underset{D}{*} \left(\underset{2\gamma_{w}}{M_{2} \otimes L(\Z)} \oplus \overset{p^{f}_{v}}{\underset{\gamma_{v} - \gamma_{w}}{\C}}\right).
$$
Note that the central support of $p_{v}$ is 1 in $\cN_{f}$.  By \cite{1110.5597},
$$
p_{v}\cN(\Gamma)p_{v} = L(\Z) * p_{v}\cN_{f}p_{v} = L(\Z) * \left(\underset{\gamma_{w}}{L(\Z)} \oplus \overset{p_{f}^{v}}{\underset{\gamma_{v} - \gamma_{w}}{\C}}\right)
$$
which is an interpolated free group factor $L(\F_{t})$ for some appropriate $t$.  By amplifying, it follows that $\cN(\Gamma)$ is an interpolated free group factor, in agreement with Theorem \ref{thm:standardembedding}.

\item[\text{\underline{Case 2b:}}]
Assume $\gamma_{v} < \gamma_{w}$.  We note that the central support of $p_{v}$ in $\cN_{f}$ is $1 - p_{w}^{f}$ and following the same algorithm as in Case 2a gives
$$
p_{v}\cN(\Gamma)p_{v} = L(\Z) * p_{v}\cN_{f}p_{v} = L(\Z) * L(\Z) = L(\F_{2}),
$$
so by amplifying, $\cN(\Gamma) = L(\F_{t}) \oplus \overset{p_{w}^{f}}{\underset{\gamma_{w}}{\C}}$ for appropriate $t$.  This agrees with Theorem \ref{thm:standardembedding}.
\end{proof}

\begin{lem} \label{lem:inductivestep}
Suppose $\Gamma \subset \Gamma'$ are finite weighted graphs and assume $\cN(\Gamma)$ is of the form in Theorem \ref{thm:standardembedding}.  In addition, suppose $\Gamma'$ is obtained from $\Gamma$ by one of the following three operations:

\begin{itemize}
\item[(1)]
$\Gamma$ and $\Gamma'$ have the same edge set and $\Gamma'$ is obtained from $\Gamma$ be adding a loop $e$ at a vertex $v$,
\item[(2)]
$\Gamma$ and $\Gamma'$ have the same edge set and $\Gamma'$ is obtained from $\Gamma$ be adding an edge $e$ between two distinct vertices $v$ and $w$, or
\item[(3)]
$\Gamma'$ is obtained from $\Gamma$ be adding a vertex $v$ and an edge $e$ connecting $w \in V(\Gamma)$ to $v$.
\end{itemize}

Then $\cN(\Gamma')$ is also of the form in Theorem \ref{thm:standardembedding} and $p^{\Gamma}\cN(\Gamma)p^{\Gamma} \overset{s.e.}{\hookrightarrow} p^{\Gamma}\cN(\Gamma')p^{\Gamma}$.
\end{lem}

\begin{proof}
The steps in proving the lemma assuming operations (2) or (3) are exactly the same as in \cite{1208.2933}, so we only need to assume that $\Gamma'$ is obtained from $\Gamma$ via operation (1).  Set $D = \ell^{\I}(\Gamma) = \ell^{\I}(\Gamma')$.  We have $\cN(\Gamma') = \cN_{e} \underset{D}{*} \cN(\Gamma)$ and by assumption, $$
p_{v}\cN(\Gamma)p_{v} = \overset{p_{v}^{\Gamma}}{L(\F_{t})} \oplus \overset{r_{v}^{\Gamma}}{\C}
$$
where $r_{v}^{\Gamma}$ can possibly be zero but $p_{v}^{\Gamma} = p_{v}\cdot p^{\Gamma} \neq 0$.  By \cite{1110.5597},
$$
p_{v}\cN(\Gamma')p_{v} = L(\Z) * p_{v}\cN(\Gamma)p_{v} = L(\Z) * \left(\overset{p_{v}^{\Gamma}}{L(\F_{t})} \oplus \overset{r_{v}^{\Gamma}}{\C}\right)
$$
which is an interpolated free group factor.  By amplification, it follows that $\cN(\Gamma')$ is of the form in Theorem \ref{thm:standardembedding}.  By \cite{MR1201693}, The inclusion $p^{\Gamma}_{v}\cN(\Gamma)p^{\Gamma}_{v} \rightarrow p^{\Gamma}_{v}\cN(\Gamma')p^{\Gamma}_{v}$ is equivalent to the inclusion
$$
L(\F_{t}) \rightarrow L(\F_{t}) * p^{\Gamma}_{v}\left(L(\Z) * \left(\overset{p_{v}^{\Gamma}}{\C} \oplus \overset{r_{v}^{\Gamma}}{\C}\right)\right)p^{\Gamma}_{v}
$$
which is a standard embedding.  As $p_{v}^{\Gamma} \leq p^{\Gamma}$ it follows that $p^{\Gamma}\cN(\Gamma)p^{\Gamma} \overset{s.e.}{\hookrightarrow} p^{\Gamma}\cN(\Gamma')p^{\Gamma}$.
\end{proof}

We notice that Theorem \ref{thm:standardembedding} follows from Lemmas \ref{lem:basecase} and \ref{lem:inductivestep}.  Indeed, if $\Gamma \subset \Gamma'$ are finite, connected, weighted graphs then $\Gamma'$ can be obtained from $\Gamma$ by applying operations (1), (2), and (3) above.  Also, the composite of standard embeddings is a standard embedding and standard embeddings are preserved by cut-downs.

Before determining the isomorphism $M \cong L(\F_{\I})$ we note the following lemma whose proof is a straightforward induction exercise using the algorithms in \cite{1110.5597}.

\begin{lem} \label{lem:star}
Suppose $\Gamma$ consists of a vertex $v$ connected to vertices $w_{1}$ through $w_{k}$ by a total of $n$ edges.  Assume further that $\gamma_{v} \leq \gamma_{w_{i}}$ for all $i$.  Then $p_{v}\cN(\Gamma)p_{v} = L(\F_{t})$ where $t \geq n$.
\end{lem}

\begin{thm}
$M \cong L(\F_{\I})$.
\end{thm}

\begin{proof}
Let $\Gamma$ denote $\cF_{\sC}(\cL)$ so that $M \cong p_{\emptyset}\cN(\Gamma)p_{\emptyset}$.  We build up $\Gamma$ by an increasing union of finite connected subgraphs $\Gamma_{k}$, each of which contain the vertex * so that $M$ is the inductive limit of the algebras $p_{\emptyset}\cN(\Gamma_{k})p_{\emptyset}$.  All vertices in $\Gamma_{k}$ must have weight larger than $1$ since for $p$ an irreducible projection in $P_{\bullet}$, $p \otimes \overline{p}$ must have a subprojection equivalent to the trivial one.  It follows that for $k$ sufficiently large, $p_{\emptyset}\cN(\Gamma_{k})p_{\emptyset}$ is an interpolated free group factor and since $p_{\emptyset}\cN(\Gamma_{k})p_{\emptyset} \overset{s.e.}{\hookrightarrow} p_{\emptyset}\cN(\Gamma_{k+1})p_{\emptyset}$ we know that $M = p_{\emptyset}\cN(\Gamma)p_{\emptyset}$ must be an interpolated free group factor.  Let $\tilde{\Gamma_{n}}$ be the subgraph of $\Gamma_{n}$ whose vertices are * and $v_{1}$, ..., $v_{k}$ which are connected to $*$ in $\Gamma_{k}$ and whose edges are exactly the edges connecting $*$ to each $v_{i}$ to * in $\Gamma_{n}$.  Assuming there are $k_{n}$ such vertices, we see from Lemma \ref{lem:star} that $p_{\emptyset}\cN(\tilde{\Gamma_{n}})p_{\emptyset} = L(\F_{r_{n}})$ for $r_{n} \geq k_{n}$. Since either $p_{\emptyset}\cN(\tilde{\Gamma_{n}})p_{\emptyset} =  p_{\emptyset}\cN(\Gamma_{n})p_{\emptyset}$ or $p_{\emptyset}\cN(\tilde{\Gamma_{n}})p_{\emptyset} \overset{s.e.}{\hookrightarrow} p_{\emptyset}\cN(\Gamma_{n})p_{\emptyset}$ it follows that $p_{\emptyset}\cN(\Gamma_{n})p_{\emptyset} = L(\F_{t_{n}})$ for $t_{n} \geq r_{n}$.  There are infinitely many edges emanating from $*$.  Therefore $k_{n}$ and hence $t_{n}$ can be arbitrarily large so it follows that $M = L(\F_{t})$ where $t = \lim_{n} t_{n} = \I$.
\end{proof}

\begin{cor}
$M_{\alpha} \cong L(\F_{\I})$ for all $\alpha \in \Lambda$.
\end{cor}
\begin{proof}
$M_{\alpha}$ is an amplification of $M$.
\end{proof}

\begin{rem}\label{rem:Finite}
With a bit more careful analysis, using the techniques in \cite{1208.2933}, one can show that if $\sC$ has infinitely many simple objects then the factor $M$ is $L(\F_{\I})$ no matter the choice of the generating set $\cS$ (using $\cL=\cS\oplus \overline{\cS}$).  When $\sC$ has finitely many simple objects, then we can find $t$ finite with $M \cong L(\F_{t})$.  To do this, we can find a single object $X$ which generates $\sC$, and we let $\cL=\{X\oplus \overline{X}\}$.  Applying the analysis in Section \ref{sec:free1} shows that $M$ is a cutdown of $\cN(\cF_\sC(X\oplus\overline{X}))$.  Keeping track of the free dimension yields the isomorphism $M \cong L(\F(1 + \dim(\sC)(2\dim(X) - 1)))$.
\end{rem}

%%%%%%%%%%%%%%%%%%%%%%%%%%%%%%%%%%%%%%%%%%%%%%%%%
%%%%%%%%%%%%%%%%%%%%%%%%%%%%%%%%%%%%%%%%%%%%%%%%%
%%%%%%%%%%%%%%%%%%%%%%%%%%%%%%%%%%%%%%%%%%%%%%%%%
\appendix
\section{Appendix}\label{sec:Appendix}

\subsection{The factors in \cite{MR1960417} have property $\Gamma$}
First, we sketch the proof that the factors in \cite{MR1960417} have property $\Gamma$, and thus are not interpolated free group factors. We use some notation from \cite{MR1960417}.

\begin{proof}[Sketch of proof]
Let $\sC$ be a countably generated rigid $C^*$-tensor category.
Let $\cS$ be the set of isomorphism classes of simple objects.
Let $\set{e_s}{s\in \cS}$ be a family of non-zero, pairwise orthogonal projections in the hyperfinite $II_1$ factor $R$.
Consider the von Neumann algebras
$$
A=\bigoplus_{s\in \cS}e_sRe_s\text{ and }B=\bigoplus_{s,t\in\cS}\mathcal \sC(s,t)\overline{\otimes} e_sRe_t.
$$
Note that we have a unital, connected inclusion $A\subset B$.
The factors considered in \cite{MR1960417} are of the form
$$
N=(Q\overline\otimes A)*_A B,
$$
where $Q$ is an arbitrary finite von Neumann algebra.
To show that $N$ admits a non-trivial central sequence, it suffices to find a sequence in $A$ which commutes asymptotically with $B$.
Let $x=(x_n)$ be a central sequence of $R$.
One can check that the sequence $(y_n)$ given by
$$
y_n=\bigoplus_{s\in\cS}e_sx_ne_s\in A
$$
has the desired property.
Hence $N$ has property $\Gamma$ and is not isomorphic to an interpolated free group factor.
\end{proof}

\subsection{Proofs using $\lambda$-lattices}
We now sketch a proof of Theorem \ref{thm:Main} using the techniques of \cite{MR2051399} in the case where the rigid $C^*$-tensor category $\sC$ is finitely generated. This sketch closely follows \cite[Theorem 4.1]{1112.4088}.
\begin{proof}[Sketch of proof]
Suppose $\cS=\{X_1,\dots,X_r\}$ generates $\sC$.
First, set
$$
X=X_1\oplus \overline{X_1}\oplus \cdots \oplus X_r \oplus \overline{X_r}
$$
which is a symmetrically self-dual object in $\sC$. Setting
\begin{align*}
P_{n,+} &= \End_{\sC}(X\otimes \overline{X}\otimes \cdots \otimes X^{\pm})\\
P_{n,-} &= \End_{\sC}(\overline{X}\otimes X\otimes \cdots \otimes X^{\pm})
\end{align*}
where $X^\pm= X,\overline{X}$ depending on the parity of $n$, $P_\bullet$ is a subfactor planar algebra. (This follows by results similar to those in Section \ref{sec:TCandPA}).
A subfactor planar algebra is naturally a $\lambda$-lattice \cite{MR1334479,math/9909027}, and thus \cite{MR2051399} gives us a subfactor $N\subset M$ where $N,M$ are both isomorphic to $L(\F_\I)$ and whose subfactor planar algebra is isomorphic to $P_\bullet$. The result now follows by a modified version of Theorem \ref{thm:PAandTC}.
\end{proof}

Difficulties arise when trying to adapt an approach along these lines for $\sC$ not finitely generated. One wants to define an inductive limit Popa system and use  \cite{MR2051399} to get an inclusion of factors isomorphic to $L(\F_\I)$ which remembers $\sC$.

One hope is to look at finitely generated subcategories of $\sC$ as follows. Set
$$
Y^{r} = X_{1} \oplus \overline{X_1} \oplus \cdots \oplus X_r \oplus \overline{X_r},
$$
and define $Z^{r}$ inductively by setting $Z^{1} = Y^{1}$ and $Z^{r} = Z^{r-1} \otimes Y^{r}$. Now for all $r$, we have subfactor planar algebras $P_\bullet^r$ given by
\begin{align*}
P^{r}_{n,+} &= \End_{\sC}(Z^{r}\otimes \overline{Z^{r}}\otimes \cdots \otimes Z^{r,\pm})\\
P^{r}_{n,-} &= \End_{\sC}(\overline{Z^{r}}\otimes Z^{r}\otimes \cdots \otimes Z^{r,\pm}).
\end{align*}
Again, using \cite{MR2051399}, we get an inclusion $N^{r} \subset M^{r}$ where both factors are isomorphic to $L(\F_{\I})$, and the associated category of $N^r-N^r$ bimodules is equivalent to the subcategory of $\sC$ generated by $X_{1}, \overline{X_1},\dots, X_{r},\overline{X_r}$.

One problem with this approach is that while $P^{r}_{n, \pm}$ includes unitally into $P^{r+1}_{n, \pm}$, the inductive limit planar algebra $P_\bullet^\I$ has infinite dimensional $n$-box spaces, i.e., $\dim(P_{n,\pm}^\I)=\I$ for all $n$. Hence we cannot directly use \cite{MR2051399} to get a subfactor. Another problem is that the inclusion  $P^{r}_{\bullet}\hookrightarrow P_\bullet^{r+1}$ does not induce an inclusion of $N^{r}$ into $N^{r+1}$ nor an inclusion $M^r$ into $M^{r+1}$. Hence we do not get an inductive limit subfactor $N^\I\subset M^\I$.

Of course, one can try other approaches along these lines, but so far, the authors have not succeeded in finding an inductive limit inclusion.

%%%%%%%%%%%%%%%%%%%%%%%%%%%%%%%%%%%%%%%%%%%%%%%%%
\bibliographystyle{amsalpha}
\bibliography{bibliography}

\providecommand{\bysame}{\leavevmode\hbox to3em{\hrulefill}\thinspace}
\providecommand{\MR}{\relax\ifhmode\unskip\space\fi MR }
\providecommand{\MRhref}[2]{%
  \href{http://www.ams.org/mathscinet-getitem?mr=#1}{#2}
}
\providecommand{\href}[2]{#2}
\begin{thebibliography}{MPS10}

\bibitem[Bis97]{MR1424954}
Dietmar Bisch, \emph{Bimodules, higher relative commutants and the fusion
  algebra associated to a subfactor}, Operator algebras and their applications
  (Waterloo, ON, 1994/1995), 13-63, Fields Inst. Commun., 13, Amer. Math. Soc.,
  Providence, RI, 1997, \mathscinet{MR1424954}, \googlebooks{_InIRTO8Y7gC}.

\bibitem[Bro12]{1202.1298}
Arnaud Brothier, \emph{Unshaded planar algebras and their associated ${II}_1$
  factors}, Journal of Functional Analysis \textbf{262} (2012), no.~9, 3839 --
  3871, \doi{10.1016/j.jfa.2012.02.002}, \arXiv{1202.1298}.

\bibitem[Con80]{MR561983}
Alain Connes, \emph{On the spatial theory of von {N}eumann algebras}, J. Funct.
  Anal. \textbf{35} (1980), no.~2, 153--164, \mathscinet{MR561983}.

\bibitem[DR11]{1110.5597}
Ken Dykema and Daniel Redelmeier, \emph{The amalgamated free product of
  hyperfinite von {N}eumann algebras over finite dimensional subalgebras},
  2011, \arXiv{1110.5597}.

\bibitem[Dyk93]{MR1201693}
Ken Dykema, \emph{Free products of hyperfinite von {N}eumann algebras and free
  dimension}, Duke Math. J. \textbf{69} (1993), no.~1, 97--119,
  \mathscinet{MR1201693}, \doi{10.1215/S0012-7094-93-06905-0}. \MR{1201693
  (93m:46071)}

\bibitem[Dyk95]{MR1363079}
Kenneth~J. Dykema, \emph{Amalgamated free products of multi-matrix algebras and
  a construction of subfactors of a free group factor}, Amer. J. Math.
  \textbf{117} (1995), no.~6, 1555--1602, \mathscinet{MR1363079},
  \doi{10.2307/2375030}. \MR{1363079 (97b:46075)}

\bibitem[Dyk11]{MR2765550}
Ken Dykema, \emph{A description of amalgamated free products of finite von
  {N}eumann algebras over finite-dimensional subalgebras}, Bull. Lond. Math.
  Soc. \textbf{43} (2011), no.~1, 63--74, \mathscinet{MR2765550},
  \doi{10.1112/blms/bdq079}. \MR{2765550 (2012b:46142)}

\bibitem[EK98]{MR1642584}
David~E. Evans and Yasuyuki Kawahigashi, \emph{Quantum symmetries on operator
  algebras}, Oxford Mathematical Monographs. Oxford Science Publications. The
  Clarendon Press, Oxford University Press, New York, 1998, xvi+829 pp. ISBN:
  0-19-851175-2, \mathscinet{MR1642584}.

\bibitem[FR12]{1112.4088}
S\'{e}bastien Falgui\`{e}res and Sven Raum, \emph{Tensor {C}*-categories
  arising as bimodule categories of ${II}_1$ factors}, 2012, \arXiv{1112.4088}.

\bibitem[FV11]{0811.1764}
S\'{e}bastien Falgui\`{e}res and Stefaan Vaes, \emph{The representation
  category of any compact group is the bimodule category of a ${II}_1$ factor},
  2011, \arXiv{1112.4088}.

\bibitem[Gho11]{MR2811311}
Shamindra~Kumar Ghosh, \emph{Planar algebras: a category theoretic point of
  view}, J. Algebra \textbf{339} (2011), 27--54, \mathscinet{MR2811311},
  \arXiv{0810.4186}, \doi{10.1016/j.jalgebra.2011.04.017}. \MR{2811311
  (2012g:46093)}

\bibitem[GJS10]{MR2732052}
Alice Guionnet, Vaughan F.~R. Jones, and Dimitri Shlyakhtenko, \emph{Random
  matrices, free probability, planar algebras and subfactors}, Quanta of maths,
  Clay Math. Proc., vol.~11, Amer. Math. Soc., Providence, RI, 2010,
  \mathscinet{MR2732052}, \arXiv{0712.2904v2}, pp.~201--239.

\bibitem[GJS11]{MR2807103}
\bysame, \emph{A semi-finite algebra associated to a subfactor planar algebra},
  J. Funct. Anal. \textbf{261} (2011), no.~5, 1345--1360, \arXiv{0911.4728},
  \mathscinet{MR2807103}, \doi{10.1016/j.jfa.2011.05.004}. \MR{2807103
  (2012j:46091)}

\bibitem[Har12]{1208.2933}
Michael Hartglass, \emph{Free product von neumann algebras associated to graphs
  and {G}uionnet, {J}ones, {S}hlyakhtenko subfactors in infinite depth}, 2012,
  \arXiv{1208.2933}.

\bibitem[HI98]{MR1644299}
Fumio Hiai and Masaki Izumi, \emph{Amenability and strong amenability for
  fusion algebras with applications to subfactor theory}, Internat. J. Math.
  \textbf{9} (1998), no.~6, 669--722, \mathscinet{MR1644299}.

\bibitem[HY00]{MR1749868}
Tomohiro Hayashi and Shigeru Yamagami, \emph{Amenable tensor categories and
  their realizations as {AFD} bimodules}, J. Funct. Anal. \textbf{172} (2000),
  no.~1, 19--75, \mathscinet{MR1749868}.

\bibitem[Jon83]{MR696688}
Vaughan F.~R. Jones, \emph{Index for subfactors}, Invent. Math. \textbf{72}
  (1983), no.~1, 1--25, \mathscinet{MR696688}, \doi{10.1007/BF01389127}.

\bibitem[Jon99]{math/9909027}
\bysame, \emph{Planar algebras {I}}, 1999, \arXiv{math/9909027}.

\bibitem[Jon08]{bimodules}
\bysame, \emph{Two subfactors and the algebraic decomposition of bimodules over
  ${II}_1$ factors.}, 2008, pre-print available at
  \url{http://math.berkeley.edu/~vfr}.

\bibitem[Jon11]{JonesPAnotes}
\bysame, \emph{{J}ones' notes on planar algebras}, 2011, available at
  \url{http://math.berkeley.edu/~vfr/VANDERBILT/pl21.pdf}.

\bibitem[JSW10]{MR2645882}
Vaughan Jones, Dimitri Shlyakhtenko, and Kevin Walker, \emph{An orthogonal
  approach to the subfactor of a planar algebra}, Pacific J. Math. \textbf{246}
  (2010), no.~1, 187--197, \mathscinet{MR2645882}. \MR{2645882 (2011i:46075)}

\bibitem[Kat95]{MR1335452}
Tosio Kato, \emph{Perturbation theory for linear operators}, Classics in
  Mathematics, Springer-Verlag, Berlin, 1995, Reprint of the 1980 edition,
  \mathscinet{MR1335452}. \MR{1335452 (96a:47025)}

\bibitem[ML98]{MR1712872}
Saunders Mac~Lane, \emph{Categories for the working mathematician}, second ed.,
  Graduate Texts in Mathematics, vol.~5, Springer-Verlag, New York, 1998,
  \mathscinet{MR1712872}. \MR{1712872 (2001j:18001)}

\bibitem[MP12]{1208.3637}
Scott Morrison and David Penneys, \emph{Constructing spoke subfactors using the
  jellyfish algorithm}, 2012, \arXiv{1208.3637}.

\bibitem[MPS10]{MR2559686}
Scott Morrison, Emily Peters, and Noah Snyder, \emph{Skein theory for the
  {$D_{2n}$} planar algebras}, J. Pure Appl. Algebra \textbf{214} (2010),
  no.~2, 117--139, \arXiv{math/0808.0764}, \mathscinet{MR2559686},
  \doi{10.1016/j.jpaa.2009.04.010}. \MR{MR2559686}

\bibitem[M{\"u}g10]{MR2681261}
Michael M{\"u}ger, \emph{Tensor categories: a selective guided tour}, Rev. Un.
  Mat. Argentina \textbf{51} (2010), no.~1, 95--163, \mathscinet{MR2681261}.
  \MR{2681261 (2011f:18007)}

\bibitem[Ocn88]{MR996454}
Adrian Ocneanu, \emph{Quantized groups, string algebras and {G}alois theory for
  algebras}, Operator algebras and applications, Vol.\ 2, London Math. Soc.
  Lecture Note Ser., vol. 136, Cambridge Univ. Press, Cambridge, 1988,
  \mathscinet{MR996454}, pp.~119--172.

\bibitem[Pen12]{1110.3504}
David Penneys, \emph{A planar calculus for infinite index subfactors}, Comm.
  Math. Phys. (2012), \arXiv{1110.3504}, Accepted May 8, 2012.

\bibitem[Pop93]{MR1198815}
Sorin Popa, \emph{Markov traces on universal {J}ones algebras and subfactors of
  finite index}, Invent. Math. \textbf{111} (1993), no.~2, 375--405,
  \mathscinet{MR1198815} \doi{10.1007/BF01231293}.

\bibitem[Pop94]{MR1278111}
\bysame, \emph{Classification of amenable subfactors of type {II}}, Acta Math.
  \textbf{172} (1994), no.~2, 163--255, \mathscinet{MR1278111},
  \doi{10.1007/BF02392646}.

\bibitem[Pop95]{MR1334479}
\bysame, \emph{An axiomatization of the lattice of higher relative commutants
  of a subfactor}, Invent. Math. \textbf{120} (1995), no.~3, 427--445,
  \mathscinet{MR1334479} \doi{10.1007/BF01241137}.

\bibitem[Pop02]{MR1887878}
\bysame, \emph{Universal construction of subfactors}, J. Reine Angew. Math.
  \textbf{543} (2002), 39--81, \mathscinet{MR1887878},
  \doi{10.1515/crll.2002.017}. \MR{1887878 (2002k:46163)}

\bibitem[PS03]{MR2051399}
Sorin Popa and Dimitri Shlyakhtenko, \emph{Universal properties of {$L({\bf
  F}\sb \infty)$} in subfactor theory}, Acta Math. \textbf{191} (2003), no.~2,
  225--257, \mathscinet{MR2051399} \doi{10.1007/BF02392965}. \MR{MR2051399
  (2005b:46140)}

\bibitem[Shl99]{MR1704661}
Dimitri Shlyakhtenko, \emph{{$A$}-valued semicircular systems}, J. Funct. Anal.
  \textbf{166} (1999), no.~1, 1--47, \mathscinet{MR1704661},
  \doi{10.1006/jfan.1999.3424}. \MR{1704661 (2000j:46124)}

\bibitem[Yam03]{MR1960417}
Shigeru Yamagami, \emph{{$C^\ast$}-tensor categories and free product
  bimodules}, J. Funct. Anal. \textbf{197} (2003), no.~2, 323--346,
  \mathscinet{MR1960417}. \MR{1960417 (2004a:46062)}

\bibitem[Yam12]{1207.1923}
\bysame, \emph{Representations of multicategories of planar diagrams and tensor
  categories}, 2012, \arXiv{1207.1923}.

\end{thebibliography}
\end{document}